\documentclass[reqno,10pt]{amsart}
\usepackage{amssymb,amsmath,amsthm,}
\usepackage{a4wide}
\usepackage{amscd}
\usepackage{amsfonts}
\usepackage{amssymb}
\usepackage{latexsym}
\usepackage{color}
\usepackage{esint}
\usepackage{graphicx}
\usepackage{caption}
\usepackage{subcaption}
\usepackage{amsthm}
 \usepackage{verbatim}
\usepackage{hyperref}

\usepackage[makeroom]{cancel}
\usepackage{mathtools}

\let\hide\iffalse
\let\unhide\fi

\newtheorem{theorem}{Theorem}[section]

\newtheorem{definition}[theorem]{Definition}
\newtheorem{lemma}[theorem]{Lemma}
\newtheorem{proposition}[theorem]{Proposition}
\newtheorem{remark}[theorem]{Remark}

\let\p=\partial

\let\O=\Omega

\newcommand{\R}{\mathbb{R}}
\newcommand{\T}{\mathbb{T}}

\renewcommand{\S}{\mathbb{S}}

\newcommand{\be}{\begin{equation}}
\newcommand{\bm}{\begin{multline}}
\newcommand{\ee}{\end{equation}}

\newcommand{\xb}{x_{\mathbf{b}}}

\newcommand{\tb}{t_{\mathbf{b}}}
\newcommand{\vb}{v_{\mathbf{b}}}

\newcommand{\Bes}{\begin{eqnarray*}}
	\newcommand{\Ees}{\end{eqnarray*}}
\newcommand{\Be}{\begin{equation}}
\newcommand{\Ee}{\end{equation}}

\def\p{\partial}

\def\O{\Omega}
\def\R{\mathbb{R}}

\def\B{\begin{equation}}
\def\E{\end{equation}}
\def\BN{\begin{eqnarray*}}
\def\EN{\end{eqnarray*}}

\numberwithin{equation}{section}

\usepackage{tikz-cd}
\usepackage{ragged2e}
\usepackage{etoolbox}
\usepackage{lipsum}

\usepackage{color}

\begin{document}

\title[On $C^{2}$ solution of the free-transport equation in a disk]{On $C^{2}$ solution of the free-transport equation in a disk}

\author[G. Ko]{Gyounghun Ko}
\address[GK]{Department of Mathematics, Pohang University of Science and Technology, South Korea}
\email{gyounghun347@postech.ac.kr}
\author[D. Lee]{Donghyun Lee}
\address[DL]{Department of Mathematics, Pohang University of Science and Technology, South Korea}
\email{donglee@postech.ac.kr}

\begin{abstract}
	The free transport operator of probability density function $f(t,x,v)$ is one the most fundamental operator which is widely used in many areas of PDE theory including kinetic theory, in particular. When it comes to general boundary problems in kinetic theory, however, it is well-known that high order regularity is very hard to obtain in general. In this paper, we study the free transport equation in a disk with the specular reflection boundary condition. We obtain initial-boundary compatibility conditions for $C^{1}_{t,x,v}$ and $C^{2}_{t,x,v}$ regularity of the solution. We also provide regularity estimates. \\ 
\end{abstract}

\date{\today}
\keywords{}
\maketitle

\thispagestyle{empty}
\setcounter{tocdepth}{2}
\tableofcontents

\section{Introduction}
 The free transport equation (or free transport operator) is one of the most important ones in a wide area of mathematics. When we consider a probability density function $f : \R_{+}\times \O \times\R^{d}\rightarrow \R_{+}$,  the free transport equation is written by 
 \[
 	\p_{t}f + v\cdot\nabla_{x}f = 0.
 \]
 Above equation is very simple and has explicit solution $f(t,x,v) = f_0(x-vt, v)$ when initial data $f_0$ is smooth and spatial domain is $\R^{d}$ or $\T^{d}$. However, if we consider general boundary problems, it becomes very complicated. One of the most important and ideal boundary conditions in kinetic theory is the specular reflection boundary condition,  
 \Be \label{BC}
 f(t,x,v) = f(t,x,R_{x}v), \quad R_{x} = I - 2n(x)\otimes n(x),\quad x\in \p\O,
 \Ee
 where $n(x)$ is outward unit normal vector on the boundary $\p\O$ when $\p\O$ is smooth. \eqref{BC} is motivated by billiard model and we usually analyze the problem through characteristics:
 \begin{equation} \label{XV heu}
 \begin{split}
 X(s;t,x,v) &:= \text{position of a particle at time $s$ which was at phase space $(t,x,v)$}, \\
 V(s;t,x,v) &:= \text{velocity of a particle at time $s$ which was at phase space $(t,x,v)$}, \\
 \end{split}
 \end{equation}
 where $X(s;t,x,v)$ and $V(s;t,x,v)$ satisfy the following Hamiltonian structure,
 \Be \label{Ham}
 \frac{d}{ds}X(s;t,x,v) = V(s;t,x,v),\quad \frac{d}{ds}V(s;t,x,v) = 0, \\
 \Ee
 under billiard-like reflection condition on the boundary. Explicit formulation of $(X(s;t,x,v), V(s;t,x,v))$ will be given right after Definition \ref{notation}.
 Since $X(t;t,x,v) = x$, $V(t;t,x,v)=v$ by definition, we can easily guess the following solution,
 \[
 	f(t,x,v) = f_{0}(X(0;t,x,v), V(0;t,x,v)),
 \]
 which is same as $f_0(x-vt, v)$ when $\O=\R^{d}$.  However, unlike to whole space case, regularity of the solution $f(t,x,v)$ depends on the regularity of trajectory \eqref{XV heu}. More precisely, when $X(0;t,x,v)\in \p\O$, differentiability of \eqref{XV heu} break down in general. This means that for any time $t>0$, there exist corresponding $(x_{*},v_{*}) \in \O\times \R^{d}$ such that $f(t, \cdot, \cdot)$ is not differentiable at the point. Or equivalently, for any $(x,v)\in \O\times \R^{d}$, there exists some corresponding time $t$ such that $f(\cdot, x, v)$ is not differentiable at that time.  \\
 
 Now let us consider general kinetic model which has hyperbolic structure, such as hard sphere or general cut-off Boltzmann equations. (Of course, there are lots of other kinetic literature which consider various boundary condition problems.) Although the Boltzmann equation (or other general kinetic equations) is much more complicated than the free transport equation, the recent development of the Boltzmann (or kinetic) boundary problems shows the regularity issue of the problems very well. \\
 \indent In $\T^{3}$ or $\R^{3}$, many results have been known using high order regularity function spaces. We refer to some classical works such as  \cite{DV, StrainJAMS,GuoVPB,GuoVMB}. (We note that the apriori assumption of \cite{DV} also covers some boundary condition problems, including specular reflection \eqref{BC}.) More recently, in the case of non-cutoff Boltzmann equation (which has regularizing effect), it is known that the solution is $C^{
 \infty}$ by \cite{CILS}.\\
 \indent However, when it comes to general boundary condition problems, a way of getting sufficient high order regularity estimate is not known and low regularity approach has been widely used. By defining mild solution, low regularity $L^{\infty}$ solutions have been studied after \cite{Guo10}. In \cite{LY2004}, they studied the pointwise estimate for the Green function of the linearized Boltzmann equation in $\R$. They also obtained weighted $L^\infty$ decay of the Boltzmann equation using the Green function approach. In \cite{UY2006}, they constructed $L^2\cap L^\infty_{\beta}$ solution to the Boltzmann equation in the whole space using Duhamel's principle and the spectral theory. Recently, authors in \cite{LLWW2022} provide new analysis to derive $L^\infty$ estimate rather than using Green function and Duhamel's principle. In addition, there are lots of references \cite{CKLVPB,DHWY2017,DHWZ2019,DKL2020,DW2019,KimLee,KimLeeNonconvex,KLP2022,LY2017}, where the low regularity solutions were studied for the cut-off type Boltzmann equation. We also refer to recent results \cite{AMSY2020,DLSS2021,GHJO2020,KGH2020}, etc., which deal with low regularity solution of the kinetic equation whose collision operator has regularizing effect such as non-cutoff Boltzmann or Landau equation. In fact, however, there are only a few results known about regularity of the Boltzmann equation with boundary conditions. We refer to \cite{GKTT2016,GKTT2017,Kim2011,KimLee2021}. \\
\indent As briefly explained above, regularity issue of boundary condition problems is very fundamental problem. In fact, even without complicated collision type operators, regularity of the free transport equation with boundary conditions has not been studied thoroughly to the best of author's knowledge. \\

\subsection{Statements of main theorems}
In this paper, we study classical $C^{2}_{t,x,v}$ regularity of  the free transport equation in a 2D disk,
\Be \label{eq}
	\p_{t}f + v\cdot\nabla_{x}f = 0,\quad x\in \O:=\{ x\in\R^{2}  \ : \  |x| < 1\},
\Ee
with the specular reflection boundary condition \eqref{BC}. Note that $n(x)=x \in \p\O$, since we consider unit disk $\p\O = \{x\in\R^{2} : |x|=1 \}$. Our aim is to find initial-boundary compatibility conditions of initial data $f_0$ for $C^{1}$ and $C^{2}$ regularity of the solution $f(t,x,v)$. We expect the solution will be mild solution $f(t,x,v)=f_0(X(0;t,x,v), V(0;t,x,v))$ surely (See Definition \ref{notation} for $X$ and $V$.) By performing derivative of $f$ in terms of $t,x,v$ directly (up to second order), we will find some conditions of $f_0$ which contains first and second derivative in both $x$ and $v$. (See \eqref{C1 cond},  \eqref{C2 cond34}, \eqref{C2 cond 1}, and \eqref{C2 cond 2}.)  \\
\hide
\Be  
	f(t,x,v) = f(t,x,R_{x}v), \quad R_{x} = I - 2n(x)\otimes n(x),\quad x\in \p\O
\Ee
where $n(x)$ is outward unit normal vector 
\unhide

In general, for smooth bounded domain $\O$, we define
	\begin{equation*}
		\O = \{ x\in \R^{2} : \xi(x) < 0 \},\quad \p\O = \{ x\in \R^{2} : \xi(x) = 0 \}.
	\end{equation*}
	In the case of unit disk, we may choose 
	\begin{equation*}
	\xi(x) = \frac{1}{2}|x|^{2} - \frac{1}{2},
	\end{equation*}
	and hence
	\begin{equation*}
	\nabla\xi(x) = x,\quad
	\nabla^{2}\xi(x) = I.   \\
	\end{equation*}

Now let us define some notation to precisely describe characteristics $X(s;t,x,v)$ and $V(s;t,x,v)$. 
\begin{definition} \label{notation}
	Considering \eqref{Ham}, we define basic notations
	\begin{equation} \notag
	\begin{split}
	\tb(x,v)  &:=  \sup \big\{  s \geq 0 :  x - sv \in \Omega \big\} , \\
	\xb(x,v)  &:=  x - \tb(x,v)v = X ( t- \tb(t,x,v);t,x,v)  \  \text{1st bouncing point backward in time}, \\
	\vb(x,v)  &:=  v = \lim_{s\rightarrow\tb(t,x,v)}V ( t- s;t,x,v), \\
		t^{k}(t,x,v) & := t^{k-1} - \tb (x^{k-1}, v^{k-1}), \ \text{k-th bouncing time backward in time},  \ t^{1}(t,x,v) := t-\tb(x,v),\\
		x^{k}(x,v) &:= x^{k-1} - \tb(x^{k-1}, v^{k-1}) v^{k-1} = X(t^{k}; t^{k-1}, x^{k-1}, v^{k-1})  \\
		&= \text{k-th bouncing point backward in time},\quad \xb := x^{1}, \\
		v^{k}   &=  
		R_{x^{k}} v^{k-1} = R_{x^{k}} \lim_{s\rightarrow t^{k}-} V(s; t^{k-1}, x^{k-1}, v^{k-1}), \\
	\end{split}
	\end{equation}
	where $R_{x^{k}}$ is defined in \eqref{BC}. We set $(t^0,x^0,v^0)=(t,x,v)$ and define the specular characteristics as
	\begin{equation}\label{XV}
	\begin{split}
	X(s;t,x,v) &= \sum_{k} \mathbf{1}_{s \in ( t^{k+1}, t^{k}]}
	X(s;t^{k}, x^{k}, v^{k}),  \\
	V(s;t,x,v) &= \sum_{k} \mathbf{1}_{s \in ( t^{k+1}, t^{k}]}
	V(s;t^{k}, x^{k}, v^{k}).
	\end{split}
	\end{equation}
\end{definition}


We also use $\gamma_{\pm}$ and $\gamma_{0}$ notation to denote 
\begin{equation} \notag 
\begin{split}
\gamma_{+} &:= \{ (x,v)\in \p\O\times \R^{2} : v\cdot n(x) > 0  \},  \\
\gamma_{0} &:= \{ (x,v)\in \p\O\times \R^{2} : v\cdot n(x) = 0  \},  \\
\gamma_{-} &:= \{ (x,v)\in \p\O\times \R^{2} : v\cdot n(x) < 0  \}.  \\
\end{split}
\end{equation}

Note that unit disk $\O$ is uniformly convex and its linear trajectory \eqref{XV} is well-defined if $x\in\O$  (see velocity lemma \cite{Guo10} for example). However, we want to investigate regularity up to boundary $\overline{\O}$, so we carefully exclude $\gamma_0$ from $\overline{\O}\times \R^{2}$ since we do not define characteristics starting from (backward in time) $\gamma_0$. Hence, using \eqref{XV} and \eqref{Ham}, it is natural to write \eqref{eq} as the following mild formulation,
\Be \label{solution}
	f(t,x,v) = f_{0}(X(0;t,x,v), V(0;t,x,v)),\quad (x,v) \in \mathcal{I} :=  \{\overline{\O}\times \R^{2} \}\backslash \gamma_{0}. \\
\Ee
\\
\indent Meanwhile, to study regularity of \eqref{solution}, the following quantity is very important,
	\begin{equation} \label{def A}
	\begin{split}
	A_{v, y} 
	&:= \left[\left((v\cdot n(y))I+(n(y) \otimes v) \right)\left(I-\frac{v\otimes n(y)}{v\cdot n(y)}\right)\right] ,\quad (y,v)\in \{\p\O\times \R^2\} \backslash \gamma_0. \\
	\end{split}
	\end{equation}
	Notice that $A_{v,y}$ is a matrix-valued function $A_{v,y}: \{\R^d\times\partial \O\}\backslash \gamma_0 \rightarrow \R^d\times \R^d $. ($d=2$ in this paper in particular) In fact, $A_{v, y}$ can be written as 
	\begin{equation*} 
	\begin{split}
	A_{v, y} 
	&= \nabla_{y}\big( (v\cdot n(y))n(y)\big) =\left( (v \cdot n(y) ) \nabla_y n(y) + (n(y)\otimes v ) \nabla_y  n(y)\right),\\
	\end{split}
	\end{equation*}
	which is identical to \eqref{def A} by \eqref{normal}.  \\
	
	Throughout this paper, we denote the $v$-derivative of the $i$-th column of the matrix $A_{v,y}$ by $\nabla_v A^i_{v,y}$, where $A^i$ be the $i$-th column of a matrix $A$ for $1\leq i \leq d$. For fixed $i$, it means that
	\begin{equation*}
		\nabla_{v} A^i_{v,x^1} =\left. \left(\nabla_v  A^i_{v,y}\right)\right|_{y=x^1}. 
	\end{equation*}
	 It is important to note that we carefully distinguish between $\nabla_v A^i_{v,x^1}$ and $\nabla_v (A^i_{v,x^1(x,v)})$.	 \\
	
\hide

\begin{proposition}
	[Faa di Bruno formula] For higher order $n$-derivatives, the following formula would be useful.
	\[
		(f\circ H)^{(n)} = \sum_{\sum_{j=1}^{n}j m_{j}=n} \frac{n!}{m_{1}!\cdots m_{n}!} \big( f^{(m_{1}+\cdots+m_{n})}\circ H \big) \prod_{j=1}^{n} \Big( \frac{H^{(j)}}{j!} \Big)^{m_{j}}
	\]
\end{proposition}
\unhide
\begin{remark}
	Assume $f_{0}$ satisfies \eqref{BC}. If $f_{0} \in C^{1} _{x,v}( \overline{\O}\times \R^{2})$, then
	\Be \label{C1_v trivial}
	\nabla_{v}f_{0}(x,v) = \nabla_{v}f_{0}(x,R_{x}v)R_{x},\quad \forall x\in\p\O,\quad \forall v\in\R^{2},
	\Ee
	also hold. Similalry, if $f_{0} \in C^{2}_{x,v}(  \overline{\O}\times \R^{2})$, then 
	\Be \label{C2_v trivial}
	\nabla_{vv}f_{0}(x,v) = R_{x}\nabla_{vv}f_{0}(x,R_{x}v)R_{x},\quad \forall x\in\p\O,\quad \forall v\in\R^{2},
	\Ee
	also holds as well as \eqref{C1_v trivial}. \\
\end{remark}
\begin{theorem} [$C^{1}$ regularity] \label{thm 1}
	Let $f_{0}$ be $C^{1}_{x,v}( \overline{\O}\times\R^{2})$ which satisfies \eqref{BC}. If initial data $f_{0}$ satisfies 
	\Be \label{C1 cond}
		\Big[ \nabla_x f_0( x,v) + \nabla_v f_0(x, v) \frac{ (Qv)\otimes (Qv) }{v\cdot n}  \Big] R_{x} 
		=
		\nabla_x f_0(x, R_{x}v) 
		+ \nabla_v f_0(x, R_{x}v) \frac{ (QR_{x}v)\otimes (QR_{x}v) }{R_{x}v\cdot n},\quad (x,v)\in \gamma_{-},
	\Ee
then $f(t,x,v)$ defined in \eqref{solution} is a unique $C^{1}_{t,x,v}(\R_{+}\times  \mathcal{I})$ solution of \eqref{eq}. We also note that if \eqref{C1 cond} holds, then it also holds for $(x,v)\in \gamma_{+}$. Here, $Q$ is counterclockwise rotation by $\frac{\pi}{2}$ in $\R^{2}$. Moreover, if the initial condition \eqref{C1 cond} does not hold, then $f(t,x,v)$ is not of class $C^1_{t,x,v}$ at time $t$ such that $t^k(t,x,v)=0$ for some $k$. 

\begin{remark}[Example of initial data satisfying \eqref{C1 cond}] \label{example}
In \eqref{C1 cond}, we consider the following special case
\begin{equation} \label{specialcase}
 \nabla_x f_0(x,v)R_x = \nabla_x f_0(x,R_xv) \quad \textrm{and} \quad \nabla_vf_0(x,v)\frac{(Qv)\otimes (Qv)}{v\cdot n}R_x = \nabla_v f_0(x,R_xv) \frac{(QR_xv)\otimes (QR_xv)}{R_xv\cdot n},
 \end{equation}
for $(x,v)\in \gamma_-$. Since $Q^TR_xQ=-R_x$ and $v\cdot n=-R_xv\cdot n$, we derive 
 \begin{equation*}
 \nabla_v f_0(x,v) \cdot (Qv) = \nabla_v f_0(x,R_xv)\cdot (QR_xv),
 \end{equation*}
 from the second condition above. Here, $A^T$ means transpose of a matrix $A$. From \eqref{C1_v trivial}, we get
 \begin{equation*}
 \nabla_v f_0(x,R_xv)\cdot (R_xQv) = \nabla_v f_0(x,R_xv) \cdot (QR_xv),
 \end{equation*}
 which implies that $\nabla_vf_0(x,v)\cdot (Qv) = \nabla_v f_0(x,R_xv)\cdot (R_xQv) =0$ because $R_xQ= -QR_x$. It means that $\nabla_v f_0(x,v)$ is parallel to $v$. Then, $ f_0(x,v)$ is a radial function with respect to $v$. Since the second condition in \eqref{specialcase} also holds for $(x,v)\in \gamma_+$, we deduce that a direction of $\nabla_v f_0(x,v)$ is $v^T$ for $v\in \gamma_- \cup\gamma_+$. In other words, 
 \begin{equation*}
 f_0(x,v)=G(x,\vert v \vert), \quad (x,v)\in \gamma_-\cup \gamma_+,
 \end{equation*}
 where $G$ is a real-valued $C^1_{x,v}$ function. Moreover, $f_0$ can be continuously extended to $\gamma_0$ to satisfy $f_0 \in C^1_{x,v}(\partial \O\times \R^2)$. From the first condition $\nabla_x f_0 (x,v)R_x = \nabla_x f_0(x,R_xv)$ in \eqref{specialcase}, we have 
 \begin{equation*}
\nabla_x G(x,\vert v \vert) R_x = \nabla_x G(x,\vert v \vert).
 \end{equation*}
Thus, $\nabla_x G(x,\vert v \vert)$ is orthogonal to $n(x)=x$, which means that the directional derivative $\nabla_x f_0(x,v) \cdot n(x)$ be 0 for $x\in \partial \O$. In conclusion, $f_0(x,v)=G(x,\vert v \vert)$ such that $\nabla_x f_0(x,v) \cdot n(x)=0$ for all $(x,v)\in \partial \O\times \R^2$  whenever $f_0$ satisfies \eqref{specialcase} for $(x,v)\in \gamma_-$.
\end{remark}
	\hide
	\[
	 	Q =
	 	\begin{pmatrix}
	 	 \vert & \vert & \vert \\
	 	 \hat{x}\times \widehat{v\times x} & \hat{x} & \widehat{v\times x} \\
	 	 \vert & \vert & \vert  
	 	\end{pmatrix}
	 	\begin{pmatrix}
	 		0 & -1 & 0 \\
	 		1 & 0 & 0 \\
	 		0 &  0 & 1 
	 	\end{pmatrix}
	 	\begin{pmatrix}
	 	\vert & \vert & \vert \\
	 	\hat{x}\times \widehat{v\times x} & \hat{x} & \widehat{v\times x} \\
	 	\vert & \vert & \vert  
	 	\end{pmatrix}^{-1}.
	\]
	\unhide
\end{theorem}

\begin{theorem} [$C^{2}$ regularity] \label{thm 2}
	Let $f_{0}$ be $C^{2}_{x,v}( \overline{\O}\times\R^{2})$ which satisfies \eqref{BC} and \eqref{C1 cond}. (The condition \eqref{C1 cond} was necessary to satisfy $f(t,x,v)\in C^1_{t,x,v}$ in Theorem \ref{thm 1}). If we assume 
	\Be \label{C2 cond34}
		\nabla_{x}f_0(x, R_{x}v) \parallel (R_{x}v)^{T},\quad \nabla_{v}f_0(x, R_{x}v) \parallel (R_{x}v)^{T},
	\Ee
	and
	\begin{eqnarray}  
	&&R_{x} \Big[ \nabla_{xv}f_{0}(x,v) + \nabla_{vv}f_{0}(x,v) \frac{ (Qv)\otimes (Qv)}{v\cdot n} \Big] R_{x}
	= \nabla_{xv}f_{0}(x, R_xv)  + \nabla_{vv}f_{0}(x, R_xv) \frac{(QR_xv)\otimes (QR_xv)}{R_xv\cdot n}  \notag \\
	&&\quad\hspace{7.5cm} + 
	R_{x}
	\begin{bmatrix}
	\nabla_{v}f_{0}(x , R_xv) \mathcal{J}_1
	\\
	\nabla_{v}f_{0}(x , R_xv) \mathcal{J}_2
	\end{bmatrix} 
	R_{x},
	\label{C2 cond 1} \\
	&&R_{x}\Big[ \nabla_{xx}f_{0}(x,v)  + \nabla_{vx}f_{0}(x, v) \frac{ (Qv)\otimes (Qv)}{v\cdot n}  + \frac{ (Qv)\otimes (Qv)}{v\cdot n}  \nabla_{xv}f_{0}(x, v) \Big]  R_{x}  \notag \\
	&&\quad = \nabla_{xx}f_{0}(x, R_xv)  +  \nabla_{vx}f_{0}(x, R_xv)\frac{(QR_xv)\otimes (QR_xv)}{R_xv\cdot n}
	+ \frac{(QR_xv)\otimes (QR_xv)}{R_xv\cdot n} \nabla_{xv}f_{0}(x, R_xv)   \notag \\
	&&\quad \quad
		-2R_x
		\begin{bmatrix}
		\nabla_{v}f_{0}(x, R_xv) \nabla_{v}A^{1}_{v,x}
		\\
		\nabla_{v}f_{0}(x, R_xv) \nabla_{v}A^{2}_{v,x} 
		\end{bmatrix} R_xA_{v,x}R_x 
		+ 
		A_{v,x}\begin{bmatrix}
		\nabla_{v}f_{0}(x, R_xv) \mathcal{J}_1  \\
		\nabla_{v}f_{0}(x, R_xv) \mathcal{J}_2
		\end{bmatrix}R_x   
	\notag \\  
	&&\quad \quad
		- 2 R_x
		\begin{bmatrix}
		\nabla_{v}f_{0}(x, R_xv) \mathcal{K}_1
		\\
		\nabla_{v}f_{0}(x, R_xv) \mathcal{K}_2
		\end{bmatrix} R_x,
	\label{C2 cond 2}
	\end{eqnarray}
	where $x=(x_1,x_2), \; v=(v_1,v_2)$, and 
	\begin{align*}
		&\mathcal{J}_1:=\frac{1}{v\cdot x} \begin{bmatrix} 
	-4v_2x_1x_2  & 4v_1x_1x_2 \\ 
	-2v_2(x_2^2-x_1^2) & 2v_1(x_2^2-x_1^2)
	\end{bmatrix}, \quad 
	\mathcal{J}_2:= \frac{1}{v\cdot x}\begin{bmatrix} 
	-2v_2(x_2^2-x_1^2) & 2v_1(x_2^2-x_1^2)\\
	4v_2x_1x_2  & -4v_1x_1x_2 
	\end{bmatrix},\\
	&\mathcal{K}_1:=\begin{bmatrix}
	\dfrac{4v_1^2v_2^2x_1^3 +2v_1v_2^3(3x_1^2x_2-x_2^3)+ 2v_2^4(3x_1x_2^2+x_1^3)}{(v\cdot x)^3} & \dfrac{-4v_1^3v_2x_1^3-2v_1^2v_2^2(3x_1^2x_2-x_2^3)-2v_1v_2^3(3x_1x_2^2+x_1^3)}{(v\cdot x)^3}\\
	\dfrac{4v_2^4x_2^3+2v_1v_2^3(3x_1x_2^2-x_1^3)+2v_1^2v_2^2(3x_1^2x_2+x_2^3)}{(v\cdot x)^3} & \dfrac{-4v_1v_2^3x_2^3-2v_1^2v_2^2(3x_1x_2^2-x_1^3)-2v_1^3v_2(3x_1^2x_2+x_2^3)}{(v\cdot x)^3}
	\end{bmatrix},\\
	&\mathcal{K}_2 := \begin{bmatrix}
	\dfrac{-4v_1^3v_2x_1^3-2v_1v_2^3(3x_1x_2^2+x_1^3) -2v_1^2v_2^2(3x_1^2x_2-x_2^3)}{(v\cdot x)^3} & \dfrac{4v_1^4x_1^3 +2v_1^2v_2^2(3x_1x_2^2+x_1^3)+2v_1^3v_2 (3x_1^2x_2-x_2^3)}{(v \cdot x)^3}\\
	\dfrac{-4v_1v_2^3x_2^3 -2v_1^3v_2(3x_1^2x_2+x_2^3) -2v_1^2v_2^2(3x_1x_2^2-x_1^3)}{(v\cdot x)^3} & \dfrac{4v_1^2 v_2^2 x_2^3 +2v_1^4(3x_1^2x_2+x_2^3)+2v_1^3v_2(3x_1x_2^2-x_1^3)}{(v \cdot x)^3}
	\end{bmatrix},
	\end{align*}
for all $(x,v)\in \gamma_-$, then $f(t,x,v)$ defined in \eqref{solution} is a unique $C^{2}_{t,x,v}(\R_{+}\times  \mathcal{I})$ solution of \eqref{eq}.
	 In this case, $f_0(x,v)=G(x,\vert v \vert)$ satisfying $\nabla_x f_0(x, v)=0$ for $x \in \partial \O$, where $G$ is a real-valued $C^2_{x,v}$ function. Additionally, $f(t,x,v)$ is not of class $C^2_{t,x,v}$ at time $t$ such that $t^k(t,x,v)=0$ for some $k$ if one of the initial conditions \eqref{C2 cond34}, \eqref{C2 cond 1}, and \eqref{C2 cond 2} for $(x,v)\in \gamma_-$ is not satisfied.	  
\end{theorem}
\begin{remark} (Higher regularity)
	If we want higher regularity such as $C^{3}$ and $C^{4}$, we should assume additional initial-boundary compatibility conditions for those regularities as we assumed \eqref{C2 cond34}-\eqref{C2 cond 2} in Theorem \ref{thm 2} for $C^{2}$ as well as \eqref{C1 cond}. Although the computation for higher regularity is available in principle, we should carefully check whether the additional conditions for higher regularity make lower regularity conditions trivial or not. Here, the trivial condition for \eqref{C2 cond 2} means
	\[
		\nabla_{x,v}f_0(x,v) = 0,\quad \forall (x,v)\in\gamma_{-}.
	\]
	In fact, the answer is given in Section 1.2. Because of very nontrivial null structure of \eqref{1st order}, imposing \eqref{C2 cond34}-\eqref{C2 cond 2} does not make \eqref{C1 cond} trivial, fortunately. Once we find a new initial-boundary compatibility condition for $C^{3}$, for example, we also have to check 
	\Be \notag
	\begin{split}
		&\text{Do additional compatibility conditions for $C^{3}$ regularity  make } \\
		&\quad \text{\eqref{C1 cond} or \eqref{C2 cond34}-\eqref{C2 cond 2} trivial, e.g. $\nabla f_0 = \nabla^{2}f_0=0$ on $\gamma_{-}$?}
	\end{split}
	\Ee
	Whenever we gain conditions for higher order regularity, initial-boundary compatibility conditions are stacked and they might make lower order compatibility conditions just trivial ones. It is a very interesting question, but they require very complicated geometric considerations and  obtaining higher order condition itself will be also very painful. But, if we impose very strong (trivial) high order initial-boundary compatibility conditions
	\[
		\nabla_{x,v}^{i}f_0(x,v) = 0,\quad \forall(x,v)\in\gamma_{-},\quad 1\leq i \leq k,
	\]
	then we will get $C^{k}$ regularity of the solution.
\end{remark}	
\begin{remark} (Necessary conditions for $C^{2}$ regularity)
In Theorem \ref{thm 2}, initial conditions \eqref{C2 cond 1} and \eqref{C2 cond 2} are sufficient conditions for $f \in C^2_{t,x,v}$. Although these contain non-symmetric complicated first-order terms, we can obtain simpler necessary conditions. 
Observe that the null space of $\mathcal{J}_i, \mathcal{K}_i$ is spanned by $v$, i.e.,
\begin{equation} \label{null J,K}
	\mathcal{J}_i v =0, \quad \mathcal{K}_i v =0, \quad i=1,2.
\end{equation}
Multiplying the reflection matrix $R_x$ on both sides in \eqref{C2 cond 1} and \eqref{C2 cond 2},  we get necessary conditions for $C^{2}$ solution,
\hide
\begin{align*}
	&\nabla_{xv} f_0(x,v) +\nabla_{vv} f_0(x,v)\frac{ (Qv)\otimes (Qv)}{v\cdot n} = R_x \left[\nabla_{xv}f_{0}(x, R_xv)  + \nabla_{vv}f_{0}(x, R_xv) \frac{(QR_xv)\otimes (QR_xv)}{R_xv\cdot n} \right]R_x \\
	&\hspace{6.5cm} + 
	\begin{bmatrix}
	\nabla_{v}f_{0}(x , R_xv) \mathcal{J}_1
	\\
	\nabla_{v}f_{0}(x , R_xv) \mathcal{J}_2
	\end{bmatrix},\\
&\nabla_{xx}f_{0}(x,v)  + \nabla_{vx}f_{0}(x, v) \frac{ (Qv)\otimes (Qv)}{v\cdot n}  + \frac{ (Qv)\otimes (Qv)}{v\cdot n}  \nabla_{xv}f_{0}(x, v)\\
&\quad = R_x \left[ \nabla_{xx}f_{0}(x, R_xv)  +  \nabla_{vx}f_{0}(x, R_xv)\frac{(QR_xv)\otimes (QR_xv)}{R_xv\cdot n}
	+ \frac{(QR_xv)\otimes (QR_xv)}{R_xv\cdot n} \nabla_{xv}f_{0}(x, R_xv)\right]R_x \\ 
	&\quad \quad -2
		\begin{bmatrix}
		\nabla_{v}f_{0}(x, R_xv) \nabla_{v}A^{1}_{v,x}
		\\
		\nabla_{v}f_{0}(x, R_xv) \nabla_{v}A^{2}_{v,x} 
		\end{bmatrix} R_xA_{v,x}
		+ 
		R_xA_{v,x}\begin{bmatrix}
		\nabla_{v}f_{0}(x, R_xv) \mathcal{J}_1  \\
		\nabla_{v}f_{0}(x, R_xv) \mathcal{J}_2
		\end{bmatrix}  
		- 2
		\begin{bmatrix}
		\nabla_{v}f_{0}(x, R_xv) \mathcal{K}_1
		\\
		\nabla_{v}f_{0}(x, R_xv) \mathcal{K}_2
		\end{bmatrix} ,
\end{align*}
\unhide

\begin{equation} \label{C2 nec cond}
	\begin{split}
	&v^T \left[ \nabla_{xv} f_0(x,v) +\nabla_{vv} f_0(x,v)\frac{ (Qv)\otimes (Qv)}{v\cdot n}\right] v=(R_xv)^T \left[\nabla_{xv}f_{0}(x, R_xv)  + \nabla_{vv}f_{0}(x, R_xv) \frac{(QR_xv)\otimes (QR_xv)}{R_xv\cdot n} \right](R_xv),\\
	&v^T \left[ \nabla_{xx}f_{0}(x,v)  + \nabla_{vx}f_{0}(x, v) \frac{ (Qv)\otimes (Qv)}{v\cdot n}  + \frac{ (Qv)\otimes (Qv)}{v\cdot n}  \nabla_{xv}f_{0}(x, v)\right]v \\
	&=(R_xv)^T \left[\nabla_{xx}f_{0}(x, R_xv)  +  \nabla_{vx}f_{0}(x, R_xv)\frac{(QR_xv)\otimes (QR_xv)}{R_xv\cdot n}
	+ \frac{(QR_xv)\otimes (QR_xv)}{R_xv\cdot n} \nabla_{xv}f_{0}(x, R_xv)\right](R_xv),
	\end{split}
\end{equation}
for all $(x,v) \in \gamma_-$, where we used $R_x^2 =I$, \eqref{null J,K}, and $A_{v,x}v=0$ in Lemma \ref{lem_RA}.
\end{remark}
\begin{remark}\label{extension C2 cond34}
Using \eqref{C1_v trivial} and \eqref{C2 cond34} yields that 
\begin{equation} \label{f0 gamma+}
	\nabla_v f_0(x,v) \parallel v^T,
\end{equation}
for all $(x,v) \in \gamma_-\cup \gamma_+$. From \eqref{f0 gamma+}, we have 
\begin{equation*}
	\nabla_v f_0(x,v) \frac{(Qv) \otimes (Qv)}{v\cdot n} R_x = \nabla_v f_0(x,R_xv)\frac{(QR_xv)\otimes(QR_xv)}{R_x v\cdot n}.
\end{equation*}
Thus, the condition \eqref{C1 cond} in Theorem \ref{thm 1} becomes
\begin{equation*}
	\nabla_x f_0(x,v) R_x = \nabla_x f_0(x,R_xv).
\end{equation*}
Similarly, by \eqref{C2 cond34} and the above result, we have 
\begin{equation*}
	\nabla_x f_0(x,v) \parallel v^T,
\end{equation*}
for all $(x,v) \in \gamma_-\cup \gamma_+$. Hence, we conclude that \eqref{C2 cond34} can be extended to $\gamma_-\cup\gamma_+$ under conditions \eqref{C1 cond} and \eqref{C2 cond34} for $(x,v)\in\gamma_-$. 
\end{remark}
\begin{remark}[Extension to 3D sphere]
	By symmetry, Theorem \ref{thm 1} and \ref{thm 2} also hold for three dimensional sphere if the rotation operator $Q$ is properly redefined in the plane spanned  by $\{x, v\}$ for $x\in \p\O$, $x\nparallel v\neq 0$. \\
\end{remark}

\begin{theorem} [Regularity estimates] \label{thm 3}
	The $C^{1}(\R_{+}\times \mathcal{I})$ and $C^{2}(\R_{+}\times \mathcal{I})$ solutions of Theorem \ref{thm 1} and \ref{thm 2} enjoy the following regularity estimates :
	\Be \label{C1 bound}
	\|f\|_{C^{1}_{t,x,v}} \lesssim \|f_0\|_{C^{1}} \frac{|v|}{|v\cdot n(\xb)|^{2}} \langle v \rangle^{2}(1 + |v|t),
	\Ee
	\Be \label{C2 bound}
	\|f\|_{C^{2}_{t,x,v}} \lesssim \|f_0\|_{C^{2}} \frac{|v|^{2}}{|v\cdot n(\xb)|^{4}} \langle v \rangle^{4}(1 + |v|t)^{2},
	\Ee
	where $\xb = \xb(x,v)$ and $\langle v \rangle := 1 + |v|$. 
\end{theorem}

\subsection{Brief sketch of proofs and some important remarks}
In this paper, our aim is to analyze regularity of mild form \eqref{solution} where characteristic $(X(0;t, x, v), V(0;t, x,v))$ is well-defined (by excluding $\gamma_0$). If backward in time position $X(0;t,x,v) \notin \p\O$, the characteristic is also a smooth function and we expect that the regularity of \eqref{solution} will be the same as initial data $f_0$ by the chain rule. When $X(0;t,x,v) \in \p\O$, however, the derivative via the chain rule does not work anymore because of discontinuous behavior of velocity $V(0;t,x,v)$. Depending on perturbed  directions, we obtain different directional derivatives. In fact, we can split directions into two pieces: one gives bouncing and the other does not. See \eqref{R12_v} and \eqref{set R_vel} for $C^{1}_{v}$ for example. By matching these directional derivatives and performing some symmetrization, we obtain  symmetrized initial-boundary compatibility condition \eqref{C1 cond}. Of course, \eqref{C1_v trivial} also holds, but \eqref{C1_v trivial} is gained by taking the $v$-derivative of \eqref{BC} directly. We note that both $C^{1}_{x}$ and $C^{1}_{v}$ conditions yield identical initial compatibility condition \eqref{C1 cond}, and the condition for $C^{1}_{t}$ is just a necessary condition for \eqref{C1 cond}. \\
 
The analysis becomes much more complicated when we study $C^{2}$ conditions. Nearly all of our analysis consist of precise equalities, instead of estimates. This makes our business much harder. First, let us consider four cases: $\nabla_{xx}, \nabla_{xv}, \nabla_{vx}, \nabla_{vv}$. These yield very complicated initial-boundary compatibility conditions and in particular they contain derivatives of each column of reflection operator $R_x$ or $\nabla_{x,v}((n(x)\otimes n(x))v)$. It is nearly impossible to give proper geometric interpretation for each term. See \eqref{xv star1} and \eqref{xv star2} for example. \\

Nevertheless,  it is quite interesting that the four conditions from $\nabla_{xx}, \nabla_{xv}, \nabla_{vx}, \nabla_{vv}$ can be rearranged with respect to the order of time $t$. By matching all directional derivatives, we obtain \eqref{Cond2 1}--\eqref{Cond2 4} which contain both second-order terms and first-order terms. However, the conditions from $\nabla_{xx}, \nabla_{xv}, \nabla_{vx}, \nabla_{vv}$ must satisfy transpose compatibility condition
\Be \label{trans comp}
	 \nabla_{xv}^{T}=\nabla_{vx} \ \ \text{and} \ \ \nabla_{xx}^{T} = \nabla_{xx},
\Ee
since we hope the solution to be $C^{2}$. However, it is extremely hard to find any good geometric meaning or properties of some terms like
\Be \label{1st order}
	\nabla_{x}(R^{i}_{x^{1}(x,v)}),\quad \nabla_{x}(A^{i}_{v, x^{1}(x,v)}),\quad \text{for}\quad i=1,2,
\Ee
in \eqref{Cond2 1}--\eqref{Cond2 4}. If they do not have any special structures, the only way to get compatibility \eqref{trans comp} is to impose $\nabla_{x,v}f_0(x, Rv) = 0$ for all $(x,v)\in \gamma_{-}$. Then $C^{1}$ compatibility condition \eqref{C1 cond} becomes just trivial. Fortunately, however, the matrices of \eqref{1st order} have a rank $1$ structure. {\bf More surprisingly, all the null spaces are spanned by velocity $v$!} That is, from Lemma \ref{d_RA} and Lemma \ref{dx_A}, 
\[
	\nabla_{x}(R_{x^1(x,v)}^1)v =0, \quad \nabla_{x}(R_{x^1(x,v)}^2) v =0,\quad \nabla_x(-2A_{v,x^1(x,v)}^1)v =0, \quad \nabla_x (-2A_{v,x^1(x,v)}^2)v=0.
\]
From these interesting results, we can derive necessary conditions \eqref{C2 cond34} for transpose compatibility \eqref{trans comp}. By imposing \eqref{C2 cond34}, we derive $C^{2}$ conditions as in Theorem \ref{thm 2}, while keeping $C^{1}$ condition \eqref{C1 cond} nontrivial. We note that all the conditions that include $\p_{t}$ are repetitions of \eqref{Cond2 1}--\eqref{Cond2 4}.  \\

In the last section, we study $C^{1}$ and $C^{2}$  regularity estimates of the solution \eqref{solution}. Essentially the regularity estimates of  the solution come from the regularity estimates of characteristic $(X(0;t,x,v), V(0;t,x,v))$. For $C^{1}$ of $(X(0;t,x,v), V(0;t,x,v))$, we obtain Lemma \ref{est der X,V}. Note that we can find some cancellation that gives no singular bound for $\nabla_{v}X(0;t,x,v)$ which was found in \cite{GKTT2017} for general 3D convex domains. Growth in time need not to be exponential, but it is  just linear in time $t$. The second derivative of characteristic is much more complicated and nearly impossible to try to find any cancellation, because of too many terms and combinations that appear. Instead, by studying the most singular terms only, we obtain rough bounds in Lemma \ref{2nd est der X,V}.

\section{Preliminaries}
Now, let us recall standard matrix notations which will be used in this paper.  \\

\begin{definition} 
	When we perform matrix multiplications throughout this paper, we basically treat a n-dimensional vector $v$ as a {\it column} vector
	\[
	v = \begin{pmatrix}
	v_{1} \\ \vdots \\ v_{n}
	\end{pmatrix}.
	\]
	For about gradient of a smooth scalar function $a(x)$, however, we treat n-dimensional vector $\nabla a$ as a {\it row} vector,
	\[
	\nabla a(x) := (\p_{x_{1}} a, \p_{x_{2}} a, \cdots, \p_{x_{n}} a).
	\]
	For a smooth vector function $v : \R^{n}\rightarrow \R^{m}$ with $v(x)= \begin{pmatrix}
	v_{1}(x) \\ \vdots \\ v_{m}(x)
	\end{pmatrix}$, we define $\nabla_{x}v(x)$ as $m\times n$ matrix,
	\[
	\nabla_{x}v := \begin{pmatrix}
	\p_{1} v_{1} & \cdots & \p_{n} v_{1} \\
	\p_{1} v_{2} & \cdots & \p_{n} v_{2} \\
	\vdots & \vdots & \vdots \\
	\p_{1} v_{m} & \cdots & \p_{n} v_{m} \\
	\end{pmatrix}_{m\times n} 
	=
	\begin{pmatrix}
	& \nabla_{x} v_{1} &\\
	&\vdots&  \\
	&\nabla_{x} v_{m}& \\
	\end{pmatrix}_{m\times n} .
	\]
	We use $\otimes$ to denote tensor product 
	\begin{equation*}
	a\otimes b := \begin{pmatrix}
		a_{1} \\ \vdots \\ a_{m}
	\end{pmatrix}
	\begin{pmatrix}
		b_{1} & \cdots & b_{n}
	\end{pmatrix}.  \\
	\end{equation*}
\end{definition}

\begin{lemma}\label{matrix notation}
	(1) (Product rule) For scalar function $a(x)$ and vector function $v(x)$,
	\[
	 \nabla (a(x)v(x)) = a(x)\nabla v(x) + v\otimes \nabla a(x).
	\]
	(2) (Chain rule) For vector functions $v(x)$ and $w(x)$,
	\[
	\nabla (v(w(x))) = \nabla v(w(x)) \nabla w(x).
	\]
	(3) (Product rule) For vector functions,
	\[
		\nabla(v(x)\cdot w(x)) = v(x)\nabla w(x) + w(x)\nabla v(x).
	\]
	(4) For matrix $d\times d$ matrix $A(x)$ and $d\times 1$ vector $v(x)$,
	\begin{equation} \label{d_matrix}
	\begin{split}
		\nabla_{x} (A(x)v(x)) 
		&= A(x)\nabla v(x) 
		+ 
		\begin{pmatrix}
			v(x)\nabla A^{1}(x) \\
			\vdots \\
			v(x)\nabla A^{d}(x) 
		\end{pmatrix}  \\
		&= A(x)\nabla v(x) 
		+ \sum_{k=1}^{d}
		\p_{k}A(x) E_{k},
	\end{split}
	\end{equation}
	where $A^{i}(x)$ is $i$-th row of $A(x)$ and $E_{k}$ is $d\times d$ matrix whose $k$th column is $v$ and others are zero. (Here $\p_{k}A(x)$ means elementwise $x_{k}$-derivative of $A(x)$.) Moreover, if  $A = A(\theta(x))$ for some smooth $\theta:\O\rightarrow \mathbb{R}$, \\
	\Be \label{d_matrix_theta}
		\nabla_{x} (A(\theta)v(x)) 
		= A(\theta)\nabla v(x) 
		+ \p_{\theta}A(\theta)v \otimes \nabla_{x}\theta.
	\Ee
\end{lemma}
\begin{proof}
	Only \eqref{d_matrix_theta} needs some explanation. When $A=A(\theta(x))$, 
	\begin{equation*}
	\begin{split}
	\nabla_{x} (A(\theta)v(x)) 
	&= A(\theta)\nabla v(x) 
	+ \sum_{k=1}^{d}
	\p_{k}A(\theta) E_{k} 
	=
	A(\theta)\nabla v(x) 
	+ \sum_{k=1}^{d}
	\p_{\theta}A(\theta) \p_{k}\theta(x)E_{k} \\
	&= A(\theta)\nabla v(x) 
	+ \p_{\theta}A(\theta)
	\begin{pmatrix}
	&  &  \\
	\p_{1}\theta(x) v & \cdots & \p_{d}\theta(x)v  \\
	&  &  \\
	\end{pmatrix} \\
	&= A(\theta)\nabla v(x) 
	+ \p_{\theta}A(\theta)v \otimes \nabla_{x}\theta(x).
	\end{split}
	\end{equation*}
\end{proof}

\begin{lemma} \label{nabla xv b}
	We have the following computations where $\xb = \xb(x,v)$ and $\tb=\tb(x,v)$.  \\
	\begin{equation*}
	\begin{split}
		\nabla_{x}\tb &= \frac{n(\xb)}{v\cdot n(\xb)} ,  \\
		\nabla_{v}\tb &= -\tb\nabla_{x}\tb = -\tb\frac{n(\xb)}{v\cdot n(\xb)} ,  \\
		\nabla_{x}\xb &= I - \frac{v\otimes n(\xb)}{v\cdot n(\xb)}, \\
		\nabla_{v}\xb &= -\tb\Big(I - \frac{v\otimes n(\xb)}{v\cdot n(\xb)} \Big). \\
	\end{split}
	\end{equation*}
\end{lemma}

\begin{proof}
Remind the definition of $\xb$ and $\tb$ 
\begin{equation*}
	\xb=x-\tb v, \quad \tb=\sup\{ s \; \vert \; x-sv \in \O\}. 
\end{equation*}
Since $\xi(x) =0$ for $x \in \p\O$, we have $\xi(\xb) = \xi(x-\tb v)=0$. Taking the $x\mbox{-}$derivative $\nabla_x$, we get 
\begin{align*}
	\nabla_x(\xi(\xb))&= (\nabla\xi)(\xb) -[ (\nabla \xi)(\xb) \cdot v]\nabla_x \tb \\ 
	&= 0,
\end{align*}
where the first equality comes from product rule in Lemma \ref{matrix notation}. Thus, we can derive 
\begin{equation*}
	\nabla_x \tb = \frac{( \nabla \xi)(\xb)}{[ (\nabla \xi)(\xb) \cdot v]} = \frac{ n(\xb)}{v \cdot  n(\xb) }. 
\end{equation*}
Similarly, taking the $v\mbox{-}$derivative $\nabla_v$ and product rule in Lemma \ref{matrix notation} yields 
\begin{equation*}
	\nabla_v(\xi(\xb)) = (\nabla \xi)(\xb)(-\tb I - v \otimes \nabla_v \tb)= 0,
\end{equation*}
which implies  $\nabla_v \tb = - \tb \frac{n(\xb)}{ v\cdot n(\xb)}.$ It follows from the calculation of $\nabla_x \tb$ and $\nabla_v \tb$ above that 
\begin{align*}
	\nabla_x \xb &= \nabla_x ( x- \tb v) = I - v \otimes \nabla_x \tb = I - \frac{v \otimes n(\xb)}{v \cdot n (\xb)} \\
	\nabla_v \xb &= \nabla_v (x-\tb v) = -\tb I - v \otimes \nabla_v \tb = -\tb \left(I - \frac{ v \otimes n(\xb)}{ v \cdot n(\xb)}\right).
\end{align*}
\end{proof}

\begin{lemma} \label{d_n} For $n(\xb(x,v))$, we have the following derivative rules,
	\begin{equation} \label{normal}
		\nabla_x [n(\xb)] =   I - \frac{v \otimes n(\xb)}{v \cdot n (\xb)},  \quad \nabla_v [n(\xb)] =  -\tb  \Big( I - \frac{v \otimes n(\xb)}{v \cdot n (\xb)} \Big),
	\end{equation}
	where $\xb=\xb(x,v)$. 
\end{lemma}
\begin{proof}
For $\nabla_x n(\xb)$, we apply the chain rule in Lemma \ref{matrix notation} to $(\nabla \xi)(\xb)$ and $\frac{1}{\vert (\nabla \xi)(\xb)\vert }$ respectively. Because $\nabla \xi (x) \neq 0 $ at the boundary $x \in \p\O$ in a circle, it is possible to apply the chain rule to  $\frac{1}{\vert (\nabla \xi)(\xb)\vert }$. Taking $x\mbox{-}$derivative $\nabla_x$, one obtains 
\begin{align*}
	\nabla_x [(\nabla \xi)(\xb)]  &= (\nabla ^2 \xi)(\xb)\nabla_x \xb, \\ 
	\nabla_x \left[ \frac{1}{ \vert (\nabla \xi)(\xb) \vert} \right] & =- \frac{(\nabla\xi)(\xb) (\nabla^2\xi)(\xb) \nabla_x\xb}{\vert (\nabla \xi)(\xb) \vert^3}. 
\end{align*}
Hence, 
\begin{align*}
	\nabla_x [n(\xb)] = \nabla_x \left[ \frac{ (\nabla \xi)(\xb)}{ \vert (\nabla \xi)(\xb) \vert } \right] &= \frac{1}{ \vert (\nabla \xi)(\xb) \vert } \nabla_x [ (\nabla \xi)(\xb) ] + (\nabla \xi)(\xb) \otimes \nabla_x \left [ \frac{1}{\vert (\nabla \xi)(\xb) \vert} \right] \\ 
	& = \frac{1}{ \vert (\nabla \xi)(\xb) \vert }(\nabla ^2 \xi)(\xb)\nabla_x \xb - \nabla \xi(\xb) \otimes \frac{(\nabla\xi)(\xb) (\nabla^2\xi)(\xb) \nabla_x\xb}{\vert (\nabla \xi)(\xb) \vert^3}\\	
	&= \frac{1}{ \vert (\nabla \xi)(\xb) \vert }\Big( I - n(\xb) \otimes n(\xb)\Big) (\nabla^2 \xi)(\xb) \nabla_x \xb.
\end{align*}
Since $|\nabla\xi(\xb)| =1 $ and $\nabla^{2}\xi = I_{2}$, we deduce
\begin{align*}
\nabla_x [n(\xb)] &= \Big( I - n(\xb)\otimes n(\xb) \Big) \Big( I - \frac{v \otimes n(\xb)}{v \cdot n (\xb)} \Big) \\
&= I - \frac{v \otimes n(\xb)}{v \cdot n (\xb)} - n(\xb)\otimes n(\xb) +  n(\xb)\otimes n(\xb) \\
&= I - \frac{v \otimes n(\xb)}{v \cdot n (\xb)}.  
\end{align*}
The case for $\nabla_v [n(\xb)]$ is nearly same with extra term $-\tb$ which comes from Lemma \ref{nabla xv b}. \\
\end{proof} 

\hide
\begin{lemma}
	For fixed $x\in\O$, we can classify direction $\S^{2}$ into three parts,
	\begin{equation*}
	\begin{split}
		R_{0} &:= \{ \hat{r}\in \S^{2} : \nabla_{x}\tb(x,v)\}  \\
		R_{1} &:= \{ \hat{r}\in \S^{2} : \}  \\
		R_{2} &:= \{ \hat{r}\in \S^{2} : \}
	\end{split}
	\end{equation*}
\end{lemma}
\begin{proof}
	(i) From Proposition \eqref{nabla xv b} 
	\[
		\frac{\p}{\p\varepsilon}\tb(x+\varepsilon\hat{r}, v)\vert_{\varepsilon=0} = \nabla_{x}\tb(x,v)\cdot\hat{r} = \frac{\hat{r}\cdot n(\xb(x,v))}{v\cdot n(\xb(x,v))}
	\]	
\end{proof}
\unhide

\section{Initial-boundary compatibility condition for $C^{1}_{t,x,v}$}


\begin{lemma} \label{lem_RA}
	Recall definition \eqref{def A} of the matrix $A_{v,x}$. We have the following identities, for $(x,v) \in \{\p\O \times \R^d\} \backslash \gamma_0$,
	\begin{equation} \label{RA}
		\begin{split}
			R_xA_{v,x} &= \frac{1}{v\cdot n(x)} Q(v\otimes v)Q^{T} = \frac{1}{v\cdot n(x)} (Qv)\otimes (Qv),  \\
			A_{v,x} R_x &=  \frac{1}{v\cdot n(x)} R_xQ(v\otimes v)Q^{T}R_x = -\frac{1}{R_xv\cdot n(x)} (QR_xv)\otimes (QR_xv),  \\
		\end{split}
	\end{equation} 
	\begin{equation} \label{A2}
		\begin{split}
			A^{2}_{v,x} &=  \frac{1}{(v\cdot n(x))^{2}} (QR_xv\otimes QR_xv)(Qv\otimes Qv),
		\end{split}
	\end{equation}
	\begin{equation} \label{Av=0}
		A_{v,x}v =0,
	\end{equation} 
where $Q := Q_{\frac{\pi}{2}}$ is counterclockwise rotation by angle $\frac{\pi}{2}$. 
\end{lemma}
\begin{proof}
	We compute
	\begin{equation*}
		\begin{split}
			R_xA_{v,x}R_x &:= \left[\left((v\cdot n(x))I - (n(x) \otimes v) \right)\left(I + \frac{v\otimes n(x)}{v\cdot n(x)}\right)\right] \\
			&= \big(Qv \otimes Qn(x)\big)\left(I + \frac{v\otimes n(x)}{v\cdot n(x)}\right).
		\end{split}
	\end{equation*}
	Now let us define $\tau(x)= Q_{-\frac{\pi}{2}}n(x)$ as tangential vector at $x\in\p\O$. ($n$ as y-axis and $\tau$ as x-axis) Then,
	\begin{equation*}
		\begin{split}
			R_xA_{v,x}R_x &:=  Qv \otimes \Big( -\tau - \frac{v\cdot\tau}{v\cdot n(x)}n(x) \Big) \\
			&= -\frac{1}{v\cdot n(x)} Qv\otimes \Big( (v\cdot n(x))\tau + (v\cdot\tau)n(x) \Big)  \\
			&= -\frac{1}{v\cdot n(x)} Qv\otimes  \big( R_xQ^{T}v \big) \\
			&= \frac{1}{v\cdot n(x)} Qv\otimes  \big( R_xQv \big) \\
			&= \frac{1}{v\cdot n(x)} Q(v\otimes v)Q^{T}R_x, \\
		\end{split}
	\end{equation*}
	and we get \eqref{RA} using $R_xQ=-R_xQ^T$, because
	\[
	Q^{T}R_xQ = I - 2Q^{T}(n(x)\otimes n(x))Q = I - 2\tau\otimes\tau = -R_x.  \\
	\]
	\eqref{A2} is simply obatined by \eqref{RA}. By definition of $A_{v,x}$ in \eqref{def A}, one obtains that 
	\begin{align*}
		A_{v,x}v =  \left[\left((v\cdot n(x))I+(n(x) \otimes v) \right)\left(I-\frac{v\otimes n(x)}{v\cdot n(x)}\right)\right]v=\left((v\cdot n(x))I+(n(x)\otimes v))\right)(v-v)=0.
	\end{align*}
\end{proof}
Now, throughout this section, we study $C^{1}_{t,x,v}(\R_{+}\times \O\times \R^{2})$ of $f(t,x,v)$ of \eqref{solution} when
\Be \label{t1 zero}
	0 = t^{1}(t,x,v) \ \text{or equivalently} \ t = \tb(x,v).
\Ee

\subsection{$C^{1}_{v}$ condition of $f$}
Since we assume \eqref{t1 zero}, $X(0;t,x,v) = x^{1}(x,v) = \xb(x,v) \in \partial \O$. To derive compatibility condition for $C^{1}_{v}$ of $f(t,x,v)$, we consider $v$-perturbation and use the following notation for perturbed trajectory: 
\begin{equation} \label{XV epsilon v}
	X^{\epsilon}(0) := X(0;t,x,v+\epsilon \hat{r}) , \quad V^{\epsilon}(0):=V(0;t,x,v+\epsilon \hat{r} ), 
\end{equation}
where $\hat{r}\in\R^{2}$ is a unit-vector. As $\epsilon \rightarrow 0$, we simply get
\begin{equation*} 
	\lim_{\epsilon \rightarrow 0} X(0;t,x,v+\epsilon \hat{r}) = x^{1}(x,v) = \xb(x,v),
\end{equation*}
from continuity of $X(0;t,x,v)$ in $v$. However, $V(0;t,x,v)$ is not continuous in $v$ because of \eqref{BC}. Explicitly, from Lemma \ref{nabla xv b},
\Be \label{R12_v}
\frac{\p}{\p\varepsilon}\tb(x, v+\varepsilon\hat{r})\vert_{\varepsilon=0} = \nabla_{v}\tb(x,v)\cdot\hat{r} = -\tb\frac{\hat{r}\cdot n(\xb(x,v))}{v\cdot n(\xb(x,v))},\quad \text{where}\quad v\cdot n(\xb(x,v)) < 0.
\Ee	
So we define, for fixed $(x,v)$, $v\neq 0$,
\begin{equation} \label{set R_vel}
\begin{split}
	R_{vel, 1} &:= \{ \hat{r}\in \S^{2} :  \hat{r}\cdot n(\xb(x,v)) < 0  \},  \\
	R_{vel, 2} &:= \{ \hat{r}\in \S^{2} :  \hat{r}\cdot n(\xb(x,v)) \geq 0  \}.  \\
\end{split}
\end{equation}
Then from \eqref{R12_v}, $\nabla_{v}\tb(x,v)\cdot\hat{r} > 0$ when $\hat{r}\in R_{vel, 1}$ and $\nabla_{v}\tb(x,v)\cdot\hat{r} \leq 0$ when $\hat{r}\in R_{vel, 2}$. Therefore, for two unit vectors $\hat{r}_1\in R_{vel, 1}$ and $\hat{r}_2\in R_{vel, 2}$, by continuity argument,
\begin{equation*} 
	\lim_{\epsilon \rightarrow 0+} V(0;t,x,v+\epsilon \hat{r}_1) = v, \quad \lim_{\epsilon \rightarrow 0+} V (0;t,x,v+\epsilon \hat{r}_2) = v^1=R_{x^{1}}v.  \\
\end{equation*}

We consider directional derivatives with respect to $\hat{r}_1$ and $\hat{r}_2$. If $f$ belongs to the $C^1_v$ class, $\nabla_v f(t,x,v)$ exists and directional derivatives of $f$ with respect to $\hat{r}_1,\hat{r}_2$ will be $\nabla_v f(t,x,v) \hat{r}_1,\;\nabla_v f(t,x,v) \hat{r}_2$. Using \eqref{BC}, we have $f_{0}(x^{1}, v) = f_{0}(x^{1}, v^{1})$ and hence
\begin{align*}
	\nabla_v f(t,x,v) \hat{r}_1 &= \lim _{\epsilon\rightarrow 0+} \frac{1}{\epsilon}\left ( f(t,x,v+\epsilon \hat{r}_1) - f(t,x,v) \right )\\ 
	&=\lim_{\epsilon \rightarrow 0+} \frac{1}{\epsilon}\left( f_0(X(0;t,x,v+\epsilon \hat{r}_1),V(0;t,x,v+\epsilon \hat{r}_1)) - f_0(X(0;t,x,v),V(0;t,x,v)) \right)\\
	&=\lim_{\epsilon \rightarrow 0+} \frac{1}{\epsilon} \left( f_0(X^{\epsilon}(0), V^{\epsilon}(0))- f_0 (X^{\epsilon}(0),v)+f_0(X^{\epsilon}(0),v) -f_0(X(0),v) \right) \\
	&=\nabla_x f_0(X(0),v) \cdot \lim_{s\rightarrow 0+} \nabla_v X(s) \hat{r}_1+ \nabla_v f_0(X(0),v) \lim_{s \rightarrow 0+} \nabla_v V(s)\hat{r}_1,  \\
	\nabla_v f(t,x,v) \hat{r}_2 &= \lim _{\epsilon\rightarrow 0+} \frac{1}{\epsilon}\left ( f(t,x,v+\epsilon \hat{r}_2) - f(t,x,v) \right )\\ 
	&=\lim_{\epsilon \rightarrow 0+} \frac{1}{\epsilon}\left( f_0(X(0;t,x,v+\epsilon \hat{r}_2),V(0;t,x,v+\epsilon \hat{r}_2)) - f_0(X(0;t,x,v),V(0;t,x,v)) \right)\\
	&=\lim_{\epsilon \rightarrow 0+} \frac{1}{\epsilon} \left( f_0(X^{\epsilon}(0), V^{\epsilon}(0))- f_0 (X^{\epsilon}(0),v^{1})+f_0(X^{\epsilon}(0),v^{1}) -f_0(X(0),v^{1}) \right) \\
	&=\nabla_x f_0(X(0),v^{1}) \cdot \lim_{s\rightarrow 0-} \nabla_v X(s) \hat{r}_2+ \nabla_v f_0(X(0),v^{1}) \lim_{s \rightarrow 0-} \nabla_v V(s)\hat{r}_2,
\end{align*}
which implies
\begin{eqnarray}  
		&& \nabla_v f(t,x,v) =\nabla_x f_0(X(0),v) \lim_{s\rightarrow 0+} \nabla_v X(s)+ \nabla_v f_0(X(0),v) \lim_{s \rightarrow 0+} \nabla_v V(s), \label{case12 r1}\\
		&& \nabla_v f(t,x,v) =\nabla_x f_0(X(0),v^{1})  \lim_{s\rightarrow 0-} \nabla_v X(s)+ \nabla_v f_0(X(0),v^{1}) \lim_{s \rightarrow 0-} \nabla_v V(s). \label{case12 r2}
\end{eqnarray}

\noindent Since
\Be \label{nabla XV_v+}
	\lim_{s\rightarrow 0+} \nabla_v X(s)= \lim_{s\rightarrow 0+}\nabla_{v}(x-v(t-s)) = -t I_{2\times 2}, \quad \lim_{s \rightarrow 0+} \nabla_v V(s) = \lim_{s \rightarrow 0+} \nabla_v v =  I_{2\times 2},
\Ee
$\nabla_v f(t,x,v)$ of \eqref{case12 r1} becomes
\begin{equation} \label{c_1}
	\nabla_v f(t,x,v) = -t \nabla_x f_0(X(0),v) + \nabla_v f_0(X(0),v). \\
\end{equation} 
For \eqref{case12 r2}, using the product rule in Lemma \ref{matrix notation} and \eqref{normal} in Lemma \ref{d_n}, we have 
\begin{equation} \label{nabla XV_v-}
\begin{split}
	\lim_{s\rightarrow 0-} \nabla_v X(s)& = \lim_{s\rightarrow 0-} \nabla_v (x^1 - (t^1+s)v^1)= \lim_{s\rightarrow 0-} \nabla_v x^1 +  v^{1}\otimes\nabla_{v}\tb \\
	&= -t\left(I-\frac{v\otimes n(x^1)}{v\cdot n(x^1)}\right) -  t\frac{v^1 \otimes n(x^1)}{ v \cdot n(x^1)} \\
	&= -t \Big( I -\frac{1}{v\cdot n(x^1)} \big( 2(v\cdot n(x^{1}))n(x^{1}) \big)\otimes n(x^{1}) \Big) = -tR_{x^{1}}, \\
	\lim_{s \rightarrow 0-} \nabla_v V(s) &= \lim_{s \rightarrow 0-} \nabla_v (R_{x^{1}}v) \\
	&=\lim_{s \rightarrow 0-} \left( I- 2(v \cdot n(x^1) ) \nabla_v n(x^1) -2 n(x^1)\otimes n(x^1) -2 (n(x^1)\otimes v ) \nabla_v n(x^1)\right)\\
	&=R_{x^{1}} +2t(v\cdot n(x^1)) \left( I - \frac{ v \otimes n(x^1)}{ v \cdot n(x^1)} \right) +2t (n(x^1) \otimes v)\left( I - \frac{ v \otimes n(x^1)}{ v \cdot n(x^1)} \right) \\
	&= R_{x^{1}} + 2t A_{v, x^{1}},
\end{split}
\end{equation}
where $A_{v, x^{1}}$ is defined in \eqref{def A}. Hence, using \eqref{nabla XV_v-}, $\nabla_v f(t,x,v)$ in \eqref{case12 r2} becomes
\begin{align} \label{c_2} 
	\begin{split}
	\nabla_v f(t,x,v) &= -t \nabla_x f_0(X(0),R_{x^{1}}v) R_{x^{1}} \\
	&\quad +\nabla_v f_0(X(0),R_{x^{1}}v)R_{x^{1}} + t\nabla_v f_0(X(0),R_{x^{1}}v)\left[2\left((v\cdot n(x^1))I+(n(x^1) \otimes v) \right)\left(I-\frac{v\otimes n(x^1)}{v\cdot n(x^1)}\right)\right] \\
	&= -t \nabla_x f_0(x^{1},R_{x^1}v) R_{x^{1}} + \nabla_v f_0(x^{1},R_{x^{1}}v) (R_{x^{1}} + 2tA_{v, x^{1}}),
	\end{split}
\end{align}
where we used $v\cdot n (x^1) = - v^1 \cdot n(x^1)$. Meanwhile, taking $\nabla_{v}$ to specular reflection \eqref{BC} directly, we get
\Be \label{comp_v}
\nabla_{v}f_{0}(x,v) =\nabla_{v}f_{0}(x,R_{x}v)R_{x}, \quad \forall x\in\p\O. 
\Ee
Comparing \eqref{c_1}, \eqref{c_2}, and \eqref{comp_v}, we deduce  
\begin{align} \label{c_v}
	\begin{split}
	\nabla_x f_0( x^{1},v) &= \nabla_x f_0(x^{1},R_{x^1}v) R_{x^{1}}   - 2\nabla_v f_0(x^{1}, R_{x^{1}}v)A_{v,x^{1}},\quad (x^{1}, v)\in \gamma_{-}.
	\end{split}
\end{align}

\subsection{$C^{1}_{x}$ condition of $f$}
Recall we assumed \eqref{t1 zero}. Similar to previous subsection, we define $x$-perturbed trajectory,  
\begin{equation} \label{XV epsilon x}
	X^{\epsilon}(0) :=X(0;t,x+\epsilon \hat{r}, v ), \quad V^{\epsilon}(0) := V(0;t,x+\epsilon \hat{r}, v),
\end{equation}
where $\hat{r}\in\R^{2}$ is a unit-vector. As $\epsilon \rightarrow 0$, we simply get
\begin{equation*} 
	\lim_{\epsilon \rightarrow 0} X(0;t,x+\epsilon\hat{r},v) = x^{1}(x,v).
\end{equation*}
Similar to previous subsection, using Lemma \ref{nabla xv b},
\Be \label{R12_x}
\frac{\p}{\p\varepsilon}\tb(x+\varepsilon\hat{r}, v)\big\vert_{\varepsilon=0} = \nabla_{x}\tb(x,v)\cdot\hat{r} = \frac{\hat{r}\cdot n(\xb(x,v))}{v\cdot n(\xb(x,v))},\quad \text{where}\quad v\cdot n(\xb(x,v)) < 0.
\Ee	
So we define, for fixed $(x,v)$, $v\neq 0$,
\begin{equation} \label{set R_sp}
\begin{split}
R_{sp, 1} &:= \{ \hat{r}\in \S^{2} :  \hat{r}\cdot n(\xb(x,v)) > 0  \},  \\
R_{sp, 2} &:= \{ \hat{r}\in \S^{2} :  \hat{r}\cdot n(\xb(x,v)) \leq 0  \}.  \\
\end{split}
\end{equation}
Then from \eqref{R12_x}, $\nabla_{x}\tb(x,v)\cdot\hat{r} > 0$ when $\hat{r}\in R_{sp, 1}$ and $\nabla_{x}\tb(x,v)\cdot\hat{r} \leq 0$ when $\hat{r}\in R_{sp, 2}$. Therefore, for two unit vectors $\hat{r}_1\in R_{sp, 1}$ and $\hat{r}_2\in R_{sp, 2}$, by continuity argument,
\begin{equation*}
\lim_{\epsilon \rightarrow 0+} V(0;t,x,v+\epsilon \hat{r}_1) = v, \quad \lim_{\epsilon \rightarrow 0+} V (0;t,x,v+\epsilon \hat{r}_2) = v^1=R_{x^{1}}v. 
\end{equation*}

Using similar arguments in previous subsection, we obtain 
\begin{eqnarray}
&& \nabla_x f(t,x,v) \hat{r}_{1} =\nabla_x f_0(X(0),v) \lim_{s\rightarrow 0+} \nabla_x X(s)\hat{r}_{1} + \nabla_v f_0(X(0),v) \lim_{s \rightarrow 0+} \nabla_x V(s)\hat{r}_{1},  \label{case12 r1 x}\\ 
		&& \nabla_x f(t,x,v)\hat{r}_{2} =\nabla_x f_0(X(0),Rv)  \lim_{s\rightarrow 0-} \nabla_x X(s) \hat{r}_{2} + \nabla_v f_0(X(0),Rv) \lim_{s \rightarrow 0-} \nabla_x V(s) \hat{r}_{2}. \label{case12 r2 x} 
\end{eqnarray}
Since 
\Be \label{nabla XV_x+}
	\lim_{s\rightarrow 0+} \nabla_x X(s) = I_{2\times 2},\quad \lim _{s \rightarrow 0+} \nabla_x V(s)=0_{2 \times 2},
\Ee
$\nabla_{x}f(t,x,v)$ of \eqref{case12 r1 x} becomes
\begin{equation} \label{c_3}
	\nabla_x f(t,x,v) = \nabla_x f_0(X(0),v). 
\end{equation}
For $\nabla_{x}f(t,x,v)$ of \eqref{case12 r2 x}, we apply the product rule in Lemma \ref{matrix notation} and \eqref{normal} in Lemma \ref{d_n} to get
\begin{equation} \label{nabla XV_x-}
\begin{split}
\lim_{s\rightarrow 0-} \nabla_x X(s)& = \lim_{s\rightarrow 0-} \nabla_x (x^1 - (t^1+s)v^1)=\left(I-\frac{v\otimes n(x^1)}{v\cdot n(x^1)}\right)  + \frac{v^1 \otimes n(x^1)}{ v \cdot n(x^1)} = R_{x^{1}},\\ 
	\lim_{s \rightarrow 0-} \nabla_x V(s) &= \lim_{s \rightarrow 0-} \nabla_x (R_{x^{1}}v) \\
	&=\lim_{s \rightarrow 0-} \left( - 2(v \cdot n(x^1) ) \nabla_x n(x^1) -2 (n(x^1)\otimes v ) \nabla_x  n(x^1)\right)\\
	&=-2(v\cdot n(x^1)) \left( I - \frac{ v \otimes n(x^1)}{ v \cdot n(x^1)} \right) -2 (n(x^1) \otimes v)\left( I - \frac{ v \otimes n(x^1)}{ v \cdot n(x^1)} \right) \\
	&= -2 A_{v,x^{1}}.
\end{split}
\end{equation}
Hence, using \eqref{nabla XV_x-}, $\nabla_x f(t,x,v)$ in \eqref{case12 r2 x} becomes
\begin{align} \label{c_4}
	\begin{split}
		\nabla_x f(t,x,v) &= \nabla_x f_0(X(0), R_{x^{1}}v) R_{x^{1}}  - 2\nabla_v f_0(X(0),R_{x^{1}}v) A_{v,x^{1}}.
	\end{split}
\end{align}
Combining \eqref{c_3} and \eqref{c_4}, 
\begin{align}\label{c_x}
\begin{split}
	\nabla_x f_0( x^{1},v) &= \nabla_x f_0(x^{1},R_{x^{1}}v) R_{x^{1}}  - 2\nabla_v f_0(x^{1}, R_{x^{1}}v) A_{v,x^{1}},\quad (x^{1}, v)\in \gamma_{-},
	\end{split}
\end{align}	
which is identical to \eqref{c_v}. \\

\hide
which exactly coincides with \eqref{c_v}. We rewrite compatibility condition as 
\Be
	\begin{split}
		\nabla_x f_0( x^{1},v) &= \nabla_x f_0(x^{1},R_{x^{1}}v) R_{x^{1}} 
		- \nabla_v f_0(x^{1}, R_{x^{1}}v)\left[2\left((v\cdot n(x^1))I+(n(x^1) \otimes v) \right)\left(I-\frac{v\otimes n(x^1)}{v\cdot n(x^1)}\right)\right],\quad (x^{1},v) \in \gamma_{-}.
	\end{split}
\Ee
\unhide

\subsection{$C^{1}_{t}$ condition of $f$} 
To check the $C^1_t$ condition, we define
\begin{align}\label{Perb_t}
	X^\epsilon(0):=X(0;t+\epsilon, x,v), \quad V^\epsilon(0):= V(0;t+\epsilon,x,v). 
\end{align}
More specifically, 
\begin{align*}
	X^\epsilon(0)=x^1-(t^1+\epsilon) R_{x^{1}}v, \quad V^\epsilon(0)= R_{x^{1}}v,\quad \epsilon > 0,
\end{align*}
and 
\begin{align*}
	X^\epsilon(0)=x-(t+\epsilon)v,\quad  V^\epsilon(0)=v,\quad \epsilon < 0. 
\end{align*}
Thus, the case ($\epsilon>0$) describes the situation after bounce (backward in time) and the case ($\epsilon<0$) describes the situation just before bounce (backward in time). Then, for $\epsilon>0$, 

\begin{align*}
	f_t(t,x,v)&= \lim_{\epsilon\rightarrow0+}\frac{f(t+\epsilon,x,v)-f(t,x,v)}{\epsilon} \\
	&=\lim_{\epsilon\rightarrow 0+} \frac{f_0(X^\epsilon(0),V^\epsilon(0))-f_0(X(0),V(0))}{\epsilon}\\
	&=\lim_{\epsilon\rightarrow 0+} \frac{f_0(X^\epsilon(0),R_{x^{1}}v)-f_0(X(0),R_{x^{1}}v)}{\epsilon}\\
	&=\nabla_x f_0(x^1,R_{x^{1}}v) \lim_{\epsilon \rightarrow 0+} \frac{X^\epsilon(0)-X(0)}{\epsilon}\\
	&=-\nabla_x f_0(x^1,R_{x^{1}}v)R_{x^{1}}v. 
\end{align*}
We only consider the situation just before collision and then  
\begin{align*}
	f_t(t,x,v)&= \lim_{\epsilon\rightarrow0-}\frac{f(t+\epsilon,x,v)-f(t,x,v)}{\epsilon} \\
	&= \lim_{\epsilon\rightarrow0-}\frac{f_0(X^\epsilon(0),v)-f_0(X(0),v)}{\epsilon}\\
	&= \nabla_xf_0(x^1,v)  \lim_{\epsilon\rightarrow0-} \frac{X^\epsilon(0)-X(0)}{\epsilon}\\
	&=-\nabla_x f_0(x^1,v)v.
\end{align*}
Thus, we derive a $C^1_t$ condition 
\begin{equation} \label{c_t}
	 \nabla_x f_0(x^1,v)v = \nabla_x f_0(x^1,R_{x^{1}}v)R_{x^{1}}v,\quad (x^{1}, v)\in \gamma_{-}.
\end{equation}
Actually, \eqref{c_t} is just particular case of \eqref{c_v}, because of \eqref{Av=0}.

\hide
{\color{blue}
\begin{remark} \label{trivial case}
Let us consider trivial case : $f(t,x,v) = f_{0}(v)$, spatially independent case. Since specular reflection  holds for all $x\in \p\O$, $f_{0}$ should be radial function, $f_{0}(v) = f_{0}(|v|)$. \eqref{Cond} also holds for this case, because vector $\nabla_{v}f_{0}(x^{1},Rv)$ has $Rv$ direction and
\Be 
\begin{split}
	&\underbrace{(Rv)}_{\text{row vector}}\left((v\cdot n(x^1))I+(n(x^1) \otimes v) \right)\left(I-\frac{v\otimes n(x^1)}{v\cdot n(x^1)}\right)  \\
	&= \left((v\cdot n(x^1)) (Rv) + (Rv\cdot n(x^1)) v \right)\left(I-\frac{v\otimes n(x^1)}{v\cdot n(x^1)}\right)  \\
	&= (v\cdot n(x^1)) Rv - (v\cdot n(x^{1}))v - (Rv\cdot v)n(x^{1}) + |v|^{2}n(x^{1})  \\
	&= (v\cdot n(x^1)) \big(v - 2(v\cdot n(x^{1}))n(x^{1}) \big)- (v\cdot n(x^{1}))v - \big(|v|^{2} - 2|v\cdot n(x^{1})|^{2}\big) n(x^{1}) + |v|^{2}n(x^{1})  \\
	&= 0.
\end{split}
\Ee
\end{remark}
} 
\unhide



\subsection{Proof of  Theorem \ref{thm 1}}

\begin{proof} [Proof of Theorem \ref{thm 1}]
		If $0 \neq t^{k}$ for any $k\in \mathbb{N}$, then $X(0;t,x,v)$ and $V(0;t,x,v)$ are both smooth function of $(t,x,v)$. By chain rule and $f_0\in C^{1}_{x,v}$, $f(t,x,v)$ of \eqref{solution} is also $C^{1}_{t,x,v}$. \\
		Now let us assume $ 0 = t^{k}(t,x,v)$ for some $k\in \mathbb{N}$. From discontinuous property of $V(0;t,x,v)$, we consider the following two cases:
		\begin{align*}
			& \quad \lim_{s\rightarrow 0+} \nabla_v V(s) \textcolor{blue}{ ( \text{or} \ \nabla_vX(s))} = \underbrace{ \lim_{s\rightarrow 0+} \frac{\partial V(s)\textcolor{blue}{( \text{or} \  \partial X(s))}}{\partial(t^{k-1},x^{k-1}, v^{k-1})} } \frac{\partial(t^{k-1},x^{k-1}, v^{k-1})}{\partial v},\\
			& \quad \lim_{s\rightarrow 0-} \nabla_v V(s) \textcolor{blue}{( \text{or} \ \nabla_vX(s))}= \underbrace{  \lim_{s\rightarrow 0-} \frac{\partial V(s)\textcolor{blue}{( \text{or} \ \partial X(s))}}{\partial(t^{k},x^{k}, v^{k})}\frac{\partial(t^{k},x^k,v^k)}{\partial(t^{k-1},x^{k-1},v^{k-1})} } \frac{\partial(t^{k-1},x^{k-1}, v^{k-1})}{\partial v}.
		\end{align*}
		First, we note that the factor $\displaystyle \frac{\partial(t^{k-1},x^{k-1}, v^{k-1})}{\partial v}$ which is common for both of above is smooth. From Lemma \ref{nabla xv b}, $t^{1}(t,x,v) = t - \tb(x,v)$, $x^{1}(x,v) = x - \tb(x,v) v$,  and $v^{1}(x,v) = R_{\xb(x,v)}v$ are all smooth functions of $(x,v)$ if $(t^{1}, x^{1}, v^{1})$ is nongrazing at $x^{1}$. Now, let us consider the mapping 
		\[
			(t^{1}, x^{1}, v^{1}) \mapsto 	(t^{2}, x^{2}, v^{2}) 
		\]
		which is smooth by 
		\[
			t^{2} = t^{1} - \tb(x^{1}, v^{1}),\quad x^{2} = x^{1} - v^{1}\tb(x^{1}, v^{1}),\quad v^{2} = R_{x^{2}}v^{1}.
		\]
		(Note that the derivative of $\tb$ on $\p\O \times \R^{3}_{v}$ can be performed by its local parametrization.)
		By the chain rule, we easily derive that $(t^{k}, x^{k}, v^{k})$ is smooth in $(x, v)$. For explicit computation and their Jacobian, we refer to \cite{KimLee}. 
		Now, it suffices to compare above two underbraced terms only. {\bf It means that no generality is lost by setting $k=1$.}  \\ 

	Initial-boundary compatibility conditions for $C^{1}_{t,x,v}$ were obtained in \eqref{c_v}, \eqref{c_x}, and \eqref{c_t}. Since compatibility conditions \eqref{c_x} and \eqref{c_t} are covered by \eqref{c_v}, $f(t,x,v)\in C^{1}_{t,x,v}$ once \eqref{c_v} holds. To change \eqref{c_v} into more symmetric presentation \eqref{C1 cond}, we apply \eqref{comp_v} and multiply invertible matrix $R_{x^{1}}$ on both sides from the right to obtain
	\begin{equation*}
	\begin{split}
		&\big( \nabla_x f_0(x^{1},v) + \nabla_v f_0(x^{1}, v) R_{x^1}A_{v,x^{1}} \big) R_{x^1} = \nabla_x f_0(x^1, R_{x^1}v) 
		- \nabla_v f_0(x^1, R_{x^1}v) A_{v,x^{1}}R_{x^1}.  \\
	\end{split}
	\end{equation*}
	This yields
	\begin{equation*}
	\begin{split}
		\Big[ \nabla_x f_0( x,v) + \nabla_v f_0(x, v) \frac{ (Qv)\otimes (Qv) }{v\cdot n(x)}  \Big]R_x 
		&=
		\nabla_x f_0(x, R_xv) 
		+ \nabla_v f_0(x, R_xv) \frac{ (QR_xv)\otimes (QR_xv) }{R_xv\cdot n(x)},
	\end{split}
	\end{equation*}
	by \eqref{RA}. \\
	
	Now we claim that compatibility condition \eqref{c_v} also holds for $(x^{1}, v)\in \gamma_{+}$. 
	By multiplying $R_{x^{1}}$ both sides and using $R^2_{x^1} = I$, $R_{x^1}n(x^1) = -n(x^1)$, and \eqref{comp_v}, we obtain
	\Be \label{C1 gamma+}
	\begin{split}
		\nabla_x f_0( x^{1}, R_{x^1}v) &= \nabla_x f_0(x^{1}, v) R_{x^{1}} 
		+ 2 \nabla_v f_0(x^{1}, v) \underbrace{ R_{x^1}\left[\left((v\cdot n(x^1))I+(n(x^1) \otimes v) \right)\left(I-\frac{v\otimes n(x^1)}{v\cdot n(x^1)}\right)\right] R_{x^1} }.
	\end{split}
	\Ee
	Since $R_{x^1}=R^{T}_{x^1}$ (transpose), the underbraced term is written as
	\Be 
	\begin{split}
		&R_{x^{1}} A_{v,x^{1}} R_{x^{1}} \\
		&= R_{x^{1}} \left[\left((v\cdot n(x^1))I+(n(x^1) \otimes v) \right)\left(I-\frac{v\otimes n(x^1)}{v\cdot n(x^1)}\right)\right] R_{x^{1}}  \\
		&= -(R_{x^1}v\cdot n(x^{1}))I - R_{x^1}v\otimes R_{x^1}n(x^{1}) + R_{x^1}n(x^{1})\otimes R_{x^1}v - \frac{R_{x^1}n(x^{1})\otimes R_{x^1}n(x^{1})}{v\cdot n(x^{1})}|R_{x^1}v|^{2}  \\
		&= - \left[\left((R_{x^1}v\cdot n(x^1))I+(n(x^1) \otimes R_{x^1}v) \right)\left(I-\frac{R_{x^1}v\otimes n(x^1)}{R_{x^1}v\cdot n(x^1)}\right)\right] \\
		&= -A_{R_{x^1}v,x^{1}},
	\end{split}
	\Ee
	and hence \eqref{C1 gamma+} is identical to \eqref{c_v} when $(x^{1},v)\in \gamma_{+}$. 
	
	Finally, we will prove that $f(t,x,v)$ is not of class $C^1_{t,x,v}$ at time $t$ such that $t^k(t,x,v)=0$ for some $k$ if \eqref{C1 cond} does not hold. As we used the chain rule, we set $t^1(t,x,v)=0$. Thus, it suffices to prove that $f(t,x,v)$ is not of class $C^1_{t,x,v}$ at time $t$ which satisfies $t^1(t,x,v)=0$ if \eqref{C1 cond} is not satisfied for $(X(0;t,x,v),v)\in \gamma_-$. Remind directional derivatives with respect to $\hat{r}_1$ and $\hat{r}_2$ to get $f \in C^1_{t,x,v}(\R_+\times \mathcal{I})$. In $C^1_v$ case, we deduced two conditions \eqref{case12 r1} and \eqref{case12 r2} from directional derivatives. However, if initial data $f_0$ does not satisfy the condition \eqref{C1 cond} at $(X(0;t,x,v),v)\in \gamma_-$, two conditions cannot coincide. It means that $f(t,x,v)$ is not $C^1_v$ at $t$ such that $t^1(t,x,v)=0$. Similar to $C^1_{t,x}$ cases, we get the same result. 
\end{proof}

\section{Initial-boundary compatibility condition for $C^{2}_{t,x,v}$}
As mentioned in the beginning of the previous section, we treat the problem \eqref{eq} as 2D problem in a unit disk $\{x\in\R^{2} : |x| < 1 \}$. 
And, throughout this section, we use the following notation to interchange column and row for notational convenience,
\[
	\begin{pmatrix}
	a \\ b
	\end{pmatrix}
	\stackrel{c\leftrightarrow r}{=}	
	\begin{pmatrix}
	a & b
	\end{pmatrix}
	,\quad 
	\begin{pmatrix}
	a & b
	\end{pmatrix}
	\stackrel{r\leftrightarrow c}{=}	
	\begin{pmatrix}
	a \\ b
	\end{pmatrix}.  \\
\]

Similar to previous section, we assume \eqref{t1 zero}, i.e., $0=t^{1}(t,x,v)$. We also assume $f_0$ satisfies specular reflection \eqref{BC} and $C^{1}_{t,x,v}$ compatibility condition \eqref{c_v} (or \eqref{C1 cond}) in this section. \\

\subsection{Condition for $\nabla_{xv}$}
Similar to previous section, we split perturbed direction into \eqref{set R_sp}. We also note that $\nabla_{v}f(t,x,v)$ can be written as \eqref{c_1} or \eqref{c_2}, which are identical by assuming \eqref{c_v}. First, using \eqref{c_1}, $\hat{r}_{1}$ of \eqref{set R_sp}, and notation \eqref{XV epsilon x}

\begin{equation} \label{nabla_xv f case1}
\begin{split}
	&\nabla_{xv} f(t,x,v) \hat{r}_1 \stackrel{c\leftrightarrow r}{=}	 \lim _{\epsilon\rightarrow 0+} \frac{1}{\epsilon}\left ( \nabla_{v}f(t,x+\epsilon \hat{r}_1,v) - \nabla_{v}f(t,x,v) \right ) \\ 
	&=\lim_{\epsilon \rightarrow 0+} \frac{1}{\epsilon}\Big( \nabla_{v}\big[ f_0(X(0;t,x+\epsilon \hat{r}_1,v),V(0;t,x+\epsilon \hat{r}_1,v)) \big]  - \big(   -t \nabla_x f_0(X(0),v) + \nabla_v f_0(X(0),v) \big)\Big)  \\
	&=\lim_{\epsilon \rightarrow 0+} \frac{1}{\epsilon} \Big\{ \nabla_{x}f_{0}(X^{\varepsilon}(0), V^{\varepsilon}(0))  \nabla_{v}X^{\varepsilon}(0) + \nabla_{v}f_{0}(X^{\varepsilon}(0), V^{\varepsilon}(0)) \nabla_{v}V^{\varepsilon}(0) \\
	&\quad\quad\quad\quad  - \big(   -t \nabla_x f_0(X(0),v) + \nabla_v f_0(X(0),v) \big) \Big\}   \\
	&=\lim_{\epsilon \rightarrow 0+} \frac{1}{\epsilon} \Big\{ -t\big[ \nabla_{x}f_{0}(X^{\varepsilon}(0), V^{\varepsilon}(0)) - \nabla_{x}f_{0}(X(0), v) \big]
	+
	\big[ \nabla_{v}f_{0}(X^{\varepsilon}(0), V^{\varepsilon}(0)) - \nabla_{v}f_{0}(X(0), v) \big] \Big\}   \\
	&=\lim_{\epsilon \rightarrow 0+} \frac{1}{\epsilon} \Big\{ -t\big[ \nabla_{x}f_{0}(X^{\varepsilon}(0), v) - \nabla_{x}f_{0}(X(0), v) \big]
	+
	\big[ \nabla_{v}f_{0}(X^{\varepsilon}(0), v) - \nabla_{v}f_{0}(X(0), v) \big] \Big\}   \\
	&\stackrel{r\leftrightarrow c}{=}	 \nabla_{xx}f_{0}(X(0),v)\lim_{s\rightarrow 0+} \nabla_{x}X(s) (-t )\hat{r}_1 + \nabla_{xv}f_{0}(X(0),v)\lim_{s\rightarrow 0+}\nabla_{x}X(s) \hat{r}_1 \\
	&= \Big( \nabla_{xx}f_{0}(x^{1},v) (-t ) + \nabla_{xv}f_{0}(x^{1},v) \Big) \hat{r}_1,
\end{split}
\end{equation}
where we have used \eqref{nabla XV_v+}, \eqref{nabla XV_x+}, $\nabla_{v}X^{\varepsilon}(0)=-t I_{2}$., and $\nabla_{v}V^{\varepsilon}(0)= I_{2}$. Similarly, using \eqref{c_2} and $\hat{r}_{2}$ of \eqref{set R_sp},
\begin{equation} \notag
\begin{split}
&\nabla_{xv} f(t,x,v) \hat{r}_2 \stackrel{c\leftrightarrow r}{=}	 \lim _{\epsilon\rightarrow 0+} \frac{1}{\epsilon}\left ( \nabla_{v}f(t,x+\epsilon \hat{r}_2,v) - \nabla_{v}f(t,x,v) \right )\\ 
&=\lim_{\epsilon \rightarrow 0+} \frac{1}{\epsilon} \Big\{ \nabla_{x}f_{0}(X^{\varepsilon}(0), V^{\varepsilon}(0))  \nabla_{v}X^{\varepsilon}(0) + \nabla_{v}f_{0}(X^{\varepsilon}(0), V^{\varepsilon}(0))   \nabla_{v}V^{\varepsilon}(0) \\
&\quad\quad\quad   - \big(  -t \nabla_x f_0(X(0),R_{x^1}v) R_{x^{1}} + \nabla_v f_0(X(0),R_{x^1}v) (R_{x^1} + 2tA_{v, x^{1}})\big) \Big\}    \\
&=\lim_{\epsilon \rightarrow 0+} \frac{1}{\epsilon} \Big\{\nabla_{x}f_{0}(X^{\varepsilon}(0), V^{\varepsilon}(0))  \nabla_{v}X^{\varepsilon}(0) + t \nabla_x f_0(X(0),R_{x^1}v) R_{x^{1}} \Big\}  \\
&\quad + \lim_{\epsilon \rightarrow 0+} \frac{1}{\epsilon} \Big\{ \nabla_{v}f_{0}(X^{\varepsilon}(0), V^{\varepsilon}(0)) \nabla_{v}V^{\varepsilon}(0) - \nabla_v f_0(X(0),R_{x^1}v) (R_{x^1} + 2tA_{v, x^{1}}) \Big\}    \\
&:=   I_{xv,1} + I_{xv,2} .
\end{split}
\end{equation}
Using \eqref{nabla XV_v-} and \eqref{nabla XV_x-},
\begin{equation*}
\begin{split}
	I_{xv,1} &:= \lim_{\epsilon \rightarrow 0+} \frac{1}{\epsilon} \Big\{\nabla_{x}f_{0}(X^{\varepsilon}(0), V^{\varepsilon}(0))  \nabla_{v}X^{\varepsilon}(0) - \nabla_{x}f_{0}(X^{\varepsilon}(0), V^{\varepsilon}(0))  \lim_{s\rightarrow 0-}\nabla_{v}X(s) \\
	&\quad\quad \quad\quad\quad  + \nabla_{x}f_{0}(X^{\varepsilon}(0), V^{\varepsilon}(0)) \lim_{s\rightarrow 0-} \nabla_{v}X(s) + t \nabla_x f_0(X(0),R_{x^1}v) R_{x^{1}} \Big\} ,\quad\quad \lim_{s\rightarrow 0-} \nabla_{v}X(s) = -tR_{x^{1}}, \\
	&=	 \nabla_{x}f_{0}(x^{1}, R_{x^1}v) \lim_{\epsilon \rightarrow 0} \frac{1}{\epsilon} \Big( \nabla_{v}X^{\varepsilon}(0) - \lim_{s\rightarrow 0-} \nabla_{v}X(s) \Big) 
	+ \lim_{\epsilon \rightarrow 0+} \frac{1}{\epsilon}\Big( \nabla_{x}f_{0}(X^{\varepsilon}(0),V^{\varepsilon}(0)) - \nabla_{x}f_{0}(X(0), R_{x^1}v) \Big) (-tR_{x^1})  \\
	&\stackrel{r\leftrightarrow c}{=} \Big[ \nabla_{x}f_{0}(x^{1}, R_{x^1}v) \lim_{\epsilon \rightarrow 0+} \frac{1}{\epsilon} \Big( \nabla_{v}X^{\varepsilon}(0) - \lim_{s\rightarrow 0-} \nabla_{v}X(s) \Big) \Big]^{T}  \\
	&\quad + (-tR_{x^1}) \Big( \nabla_{xx}f_{0}(x^{1}, R_{x^1}v) \lim_{s\rightarrow 0-}\nabla_{x}X(s) + \nabla_{vx}f_{0}(x^{1}, R_{x^1}v)\lim_{s\rightarrow 0-}\nabla_{x}V(s)   \Big) \hat{r}_{2}   \\
	&= 
		\underbrace{ \Big[ \nabla_{x}f_{0}(x^{1}, R_{x^1}v) \lim_{\epsilon \rightarrow 0+} \frac{1}{\epsilon} \Big( \nabla_{v}X(0;t, x+\epsilon \hat{r}_{2}, v) - \lim_{s\rightarrow 0-} \nabla_{v}X(s) \Big) \Big]^{T} }_{:=(*)_{xv,1}\hat{r}_{2} }   \\
	&\quad
	 + (-tR_{x^1})\big[  \nabla_{xx}f_{0}(x^{1}, R_{x^1}v) R_{x^1}  +  \nabla_{vx}f_{0}(x^{1}, R_{x^1}v)  (-2A_{v,x^{1}}) \big] \hat{r}_{2},  \\
\end{split}
\end{equation*}
and
\begin{equation*}
\begin{split}
I_{xv,2} &:= \lim_{\epsilon \rightarrow 0+} \frac{1}{\epsilon} \Big\{ \nabla_{v}f_{0}(X^{\varepsilon}(0), V^{\varepsilon}(0)) \nabla_{v}V^{\varepsilon}(0)  - \nabla_{v}f_{0}(X^{\varepsilon}(0), V^{\varepsilon}(0)) \lim_{s\rightarrow 0-}\nabla_{v}V(s) \\
&\quad\quad\quad\quad + \nabla_{v}f_{0}(X^{\varepsilon}(0), V^{\varepsilon}(0)) (R_{x^{1}} + 2tA_{v,x^{1}})  - \nabla_v f_0(X(0),R_{x^1}v) (R_{x^1} + 2tA_{v, x^{1}})\Big\} \\
&= \nabla_{v}f_{0}(x^{1}, R_{x^1}v)\lim_{\epsilon \rightarrow 0+} \frac{1}{\epsilon} \Big( \nabla_{v}V^{\varepsilon}(0)  - \lim_{s\rightarrow 0-}\nabla_{v}V(s) \Big)  \\
&\quad\quad\quad\quad + \lim_{\epsilon \rightarrow 0+} \frac{1}{\epsilon} \Big( \nabla_{v}f_{0}(X^{\varepsilon}(0), V^{\varepsilon}(0))  -  \nabla_{v}f_{0}(x^{1}, R_{x^1}v)  \Big) (R_{x^{1}} + 2tA_{v,x^{1}})  \\
&\stackrel{r\leftrightarrow c}{=} \Big[ \nabla_{v}f_{0}(x^{1}, R_{x^1}v)\lim_{\epsilon \rightarrow 0+} \frac{1}{\epsilon} \Big( \nabla_{v}V^{\varepsilon}(0)  - \lim_{s\rightarrow 0-}\nabla_{v}V(s) \Big)  \Big]^{T}  \\
&\quad\quad\quad\quad +  (R_{x^{1}} + 2tA^{T}_{v,x^{1}})\Big( \nabla_{xv}f_{0}(x^{1}, R_{x^1}v) \lim_{s\rightarrow 0-} \nabla_{x}X(s) + \nabla_{vv}f_{0}(x^{1}, R_{x^1}v)\lim_{s\rightarrow 0-}\nabla_{x}V(s)  \Big) \hat{r}_{2} \\
&= 
\underbrace{ \Big[ \nabla_{v}f_{0}(x^{1}, R_{x^1}v)\lim_{\epsilon \rightarrow 0+} \frac{1}{\epsilon} \Big( \nabla_{v}V(0;t,x+\epsilon \hat{r}_{2}, v)  - \lim_{s\rightarrow 0-}\nabla_{v}V(s) \Big) \Big]^{T} }_{:=(*)_{xv,2} \hat{r}_{2}}  \\
&\quad\quad\quad\quad + (R_{x^{1}} + 2tA^{T}_{v,x^{1}}) \big[ \nabla_{xv}f_{0}(x^{1}, R_{x^1}v) R_{x^{1}}  + \nabla_{vv}f_{0}(x^{1}, R_{x^1}v) (-2A_{v, x^{1}}) \big] \hat{r}_{2}.   \\
\end{split}
\end{equation*}
Now we compute two underbraced $(*)_{xv,1}$ and $(*)_{xv,2}$  \\
\begin{equation} \label{xv star1}
\begin{split}
(*)_{xv,1} \hat{r}_{2} &= \Big[ \nabla_{x}f_{0}(x^{1}, R_{x^1}v) \lim_{\epsilon \rightarrow 0+} \frac{1}{\epsilon} \Big( \nabla_{v}X^{\varepsilon}(0) - \lim_{s\rightarrow 0-} \nabla_{v}X(s) \Big) \Big]^{T} \\
&= 
\Big[  \nabla_{x}f_{0}(x^{1}, R_{x^1}v) \lim_{s \rightarrow 0-}\nabla_{x}(\p_{v_{1}}X(s)) \hat{r}_{2},  \nabla_{x}f_{0}(x^{1}, R_{x^1}v) \lim_{s \rightarrow 0-}\nabla_{x}(\p_{v_{2}}X(s)) \hat{r}_{2} \Big]^{T}  \\
&= \begin{bmatrix}
\nabla_{x}f_{0}(x^{1}, R_{x^1}v) \lim_{s \rightarrow 0-}\nabla_{x}(\p_{v_{1}}X(s)) 
\\
\nabla_{x}f_{0}(x^{1}, R_{x^1}v) \lim_{s \rightarrow 0-}\nabla_{x}(\p_{v_{2}}X(s)) 
\end{bmatrix}
\hat{r}_{2}  \\
&= \begin{bmatrix}
\nabla_{x}f_{0}(x^{1}, R_{x^1}v) \nabla_{x}(-t R_{x^{1}(x,v)}^1)
\\
\nabla_{x}f_{0}(x^{1}, R_{x^1}v) \nabla_{x}(-t R_{x^{1}(x,v)}^2)
\end{bmatrix}
\hat{r}_{2}.  \\
\end{split}
\end{equation}
Similarly,
\begin{equation} \label{xv star2}
\begin{split}
(*)_{xv,2} \hat{r}_{2}
&= \begin{bmatrix}
\nabla_{v}f_{0}(x^{1}, R_{x^1}v) \lim_{s \rightarrow 0-}\nabla_{x}(\p_{v_{1}}V(s)) 
\\
\nabla_{v}f_{0}(x^{1}, R_{x^1}v) \lim_{s \rightarrow 0-}\nabla_{x}(\p_{v_{2}}V(s)) 
\end{bmatrix}
\hat{r}_{2}  \\
&= \begin{bmatrix}
\nabla_{v}f_{0}(x^{1}, R_{x^1}v) \nabla_{x}(R_{x^{1}(x,v)}^1 + 2t A_{v,x^{1}(x,v)}^1)
\\
\nabla_{v}f_{0}(x^{1}, R_{x^1}v) \nabla_{x}(R_{x^{1}(x,v)}^2 + 2t A_{v,x^{1}(x,v)}^2)
\end{bmatrix}
\hat{r}_{2},  \\
\end{split}
\end{equation}
where $A^i$ means $i$th column of matrix $A$. Therefore,  
\begin{equation} \label{nabla_xv f case2}
\begin{split}
	\nabla_{xv}f(t,x,v)   
	&= \underline{(*)_{xv,1}}_{\eqref{xv star1}} + \underline{(*)_{xv,2}}_{\eqref{xv star2}}  \\
	&\quad + (-tR_{x^1})\big[  \nabla_{xx}f_{0}(x^{1}, R_{x^1}v) R_{x^1}  +  \nabla_{vx}f_{0}(x^{1}, R_{x^1}v)  (-2A_{v,x^{1}}) \big]   \\
	&\quad + (R_{x^{1}} + 2tA^{T}_{v,x^{1}}) \big[ \nabla_{xv}f_{0}(x^{1}, R_{x^1}v) R_{x^{1}}  + \nabla_{vv}f_{0}(x^{1}, R_{x^1}v) (-2A_{v, x^{1}}) \big].  \\
\end{split}
\end{equation}
 From \eqref{nabla_xv f case1} and \eqref{nabla_xv f case2}, we get the following compatibility condition 
\begin{equation} \label{xv comp}
\begin{split}	
	 &(-t)\nabla_{xx}f_{0}(x^{1},v)   + \nabla_{xv}f_{0}(x^{1},v)  \\
	 &=  \underline{(*)_{xv,1}}_{\eqref{xv star1}} + \underline{(*)_{xv,2}}_{\eqref{xv star2}}  \\
	 &\quad + (-tR_{x^1}) \nabla_{xx}f_{0}(x^{1}, R_{x^1}v) R_{x^1} + (-tR_{x^1})\nabla_{vx}f_{0}(x^{1}, R_{x^1}v) (-2A_{v,x^{1}}) \\
	 &\quad + (R_{x^{1}} + 2tA^{T}_{v,x^{1}}) \nabla_{xv}f_{0}(x^{1}, R_{x^1}v) R_{x^{1}} 
      + (R_{x^{1}} + 2tA^{T}_{v,x^{1}})  \nabla_{vv}f_{0}(x^{1}, R_{x^1}v) (-2A_{v, x^{1}}) . 
\end{split}
\end{equation}

\subsection{Condition for $\nabla_{vv}$} 
We split perturbed direction into \eqref{set R_vel}. $\nabla_{v}f(t,x,v)$ can be written as \eqref{c_1} or \eqref{c_2}. Using \eqref{c_1}, $\hat{r}_{1}$ of \eqref{set R_vel}, and notation \eqref{XV epsilon v}, \\
\begin{equation} \label{nabla_vv f case1}
\begin{split}
&\nabla_{vv} f(t,x,v) \hat{r}_1 \\
&\stackrel{c\leftrightarrow r}{=} \lim_{\epsilon \rightarrow 0+} \frac{1}{\epsilon} \Big\{ -t\big[ \nabla_{x}f_{0}(X^{\varepsilon}(0), V^{\varepsilon}(0)) - \nabla_{x}f_{0}(X(0), v) \big]
+
\big[ \nabla_{v}f_{0}(X^{\varepsilon}(0), V^{\varepsilon}(0)) - \nabla_{v}f_{0}(X(0), v) \big] \Big\}   \\
&=\lim_{\epsilon \rightarrow 0+} \frac{1}{\epsilon} \Big\{ -t\big[ \nabla_{x}f_{0}(X^{\varepsilon}(0), v+\epsilon \hat{r}_{1} ) - \nabla_{x}f_{0}(X(0), v) \big]
+
\big[ \nabla_{v}f_{0}(X^{\varepsilon}(0), v+\epsilon \hat{r}_{1}  ) - \nabla_{v}f_{0}(X(0), v) \big] \Big\} \\
&\stackrel{r\leftrightarrow c}{=} \Big[ -t \nabla_{xx}f_{0}(X(0),v)\lim_{s\rightarrow 0+} \nabla_{v}X(s)  -t \nabla_{vx}f_{0}(X(0),v)\lim_{s\rightarrow 0+} \nabla_{v}V(s)     \\
&\quad + \nabla_{xv}f_{0}(X(0),v)\lim_{s\rightarrow 0+}\nabla_{v}X(s)  + \nabla_{vv}f_{0}(X(0),v)\lim_{s\rightarrow 0+}\nabla_{v}V(s) \Big] \hat{r}_1  \\
&= \Big[ (-t ) \nabla_{xx}f_{0}(x^{1},v) (-t) + (-t )\nabla_{vx}f_{0}(x^{1},v)  + \nabla_{xv}f_{0}(x^{1},v) (-t ) + \nabla_{vv}f_{0}(x^{1},v) \Big] \hat{r}_1,
\end{split}
\end{equation}
where we have used \eqref{nabla XV_v+} and note that we have $\nabla_{v}X^{\varepsilon}(0)=-t I_{2}$ and $\nabla_{v}V^{\varepsilon}(0)= I_{2}$ for $v+\varepsilon\hat{r}_1$ case also. \\
Similarly, using \eqref{c_2}, $\hat{r}_{2}$ of \eqref{set R_vel}, and notation \eqref{XV epsilon v}, 
\begin{equation*}
\begin{split}
&\nabla_{vv} f(t,x,v) \hat{r}_{2} \\
&\stackrel{c\leftrightarrow r}{=} \lim_{\epsilon \rightarrow 0+} \frac{1}{\epsilon} \Big\{\nabla_{x}f_{0}(X^{\varepsilon}(0), V^{\varepsilon}(0))  \nabla_{v}X^{\varepsilon}(0) + t \nabla_x f_0(X(0),R_{x^1}v) R_{x^{1}} \Big\}  \\
&\quad + \lim_{\epsilon \rightarrow 0+} \frac{1}{\epsilon} \Big\{ \nabla_{v}f_{0}(X^{\varepsilon}(0), V^{\varepsilon}(0)) \nabla_{v}V^{\varepsilon}(0) - \nabla_v f_0(X(0),R_{x^1}v) (R_{x^1} + 2tA_{v, x^{1}}) \Big\}  \\
&:=   I_{vv,1} + I_{vv,2},
\end{split}
\end{equation*}
and each $I_{vv,1},I_{vv,2}$ are estimated by 
\begin{equation*}
\begin{split}
I_{vv,1} &:= \lim_{\epsilon \rightarrow 0+} \frac{1}{\epsilon} \Big\{\nabla_{x}f_{0}(X^{\varepsilon}(0), V^{\varepsilon}(0))  \nabla_{v}X^{\varepsilon}(0) - \nabla_{x}f_{0}(X^{\varepsilon}(0), V^{\varepsilon}(0))  \lim_{s\rightarrow 0-}\nabla_{v}X(s) \\
&\quad\quad \quad\quad\quad  + \nabla_{x}f_{0}(X^{\varepsilon}(0), V^{\varepsilon}(0)) \lim_{s\rightarrow 0-} \nabla_{v}X(s) + t \nabla_x f_0(X(0),R_{x^1}v) R_{x^{1}} \Big\} ,\quad \lim_{s\rightarrow 0-} \nabla_{v}X(s) = -tR_{x^{1}}, \\
&\stackrel{r \leftrightarrow c}{= } \Big[ \nabla_{x}f_{0}(x^{1}, R_{x^1}v) \lim_{\epsilon \rightarrow 0+} \frac{1}{\epsilon} \Big( \nabla_{v}X^{\varepsilon}(0) - \lim_{s\rightarrow 0-} \nabla_{v}X(s) \Big) \Big]^{T} \\
&\quad + (-tR_{x^1}) \Big( \nabla_{xx}f_{0}(x^{1}, R_{x^1}v) \lim_{s\rightarrow 0-}\nabla_{v}X(s)  + \nabla_{vx}f_{0}(x^{1}, R_{x^1}v)\lim_{s\rightarrow 0-}\nabla_{v}V(s)  \Big) \hat{r}_{2}  \\
&= 
\underbrace{ \Big[ \nabla_{x}f_{0}(x^{1}, R_{x^1}v) \lim_{\epsilon \rightarrow 0+} \frac{1}{\epsilon} \Big( \nabla_{v}X(0; t, x, v+\epsilon \hat{r}_{2}) - \lim_{s\rightarrow 0-} \nabla_{v}X(s) \Big) \Big]^{T} }_{(*)_{vv,1}\hat{r}_{2} }  \\
&\quad + (-tR_{x^1}) \big[ \nabla_{xx}f_{0}(x^{1}, R_{x^1}v)  (-tR_{x^1})  + \nabla_{vx}f_{0}(x^{1}, R_{x^1}v) (R_{x^1} + 2tA_{v,x^{1}}) \big] \hat{r}_{2},   \\
\end{split}
\end{equation*}
\begin{equation*}
\begin{split}
I_{vv,2} &:= \lim_{\epsilon \rightarrow 0+} \frac{1}{\epsilon} \Big\{ \nabla_{v}f_{0}(X^{\varepsilon}(0), V^{\varepsilon}(0)) \nabla_{v}V^{\varepsilon}(0)  - \nabla_{v}f_{0}(X^{\varepsilon}(0), V^{\varepsilon}(0)) \lim_{s\rightarrow 0-}\nabla_{v}V(s) \\
&\quad\quad\quad\quad + \nabla_{v}f_{0}(X^{\varepsilon}(0), V^{\varepsilon}(0)) (R_{x^{1}} + 2tA_{v,x^{1}})  - \nabla_v f_0(X(0),R_{x^1}v) (R_{x^1} + 2tA_{v, x^{1}}) \Big\} \\
&\stackrel{r\leftrightarrow c}{=}  \Big[ \nabla_{v}f_{0}(x^{1}, R_{x^1}v)\lim_{\epsilon \rightarrow 0+} \frac{1}{\epsilon} \Big( \nabla_{v}V^{\varepsilon}(0)  - \lim_{s\rightarrow 0-}\nabla_{v}V(s) \Big) \Big]^{T}    \\
&\quad\quad\quad\quad + (R_{x^{1}} + 2tA^{T}_{v,x^{1}}) \Big( \nabla_{xv}f_{0}(x^{1}, R_{x^1}v) \lim_{s\rightarrow 0-} \nabla_{v}X(s) + \nabla_{vv}f_{0}(x^{1}, R_{x^1}v)\lim_{s\rightarrow 0-}\nabla_{v}V(s) \Big) \hat{r}_{2}  \\
&= 
\underbrace{ \Big[ \nabla_{v}f_{0}(x^{1}, R_{x^1}v)\lim_{\epsilon \rightarrow 0+} \frac{1}{\epsilon} \Big( \nabla_{v}V(0; t, x, v+\epsilon \hat{r}_{2})  - \lim_{s\rightarrow 0-}\nabla_{v}V(s) \Big) \Big]^{T} }_{(*)_{vv,2} \hat{r}_{2} }   \\
&\quad\quad\quad\quad + (R_{x^{1}} + 2tA^{T}_{v,x^{1}}) \big[  \nabla_{xv}f_{0}(x^{1}, R_{x^1}v) (-tR_{x^1}) + \nabla_{vv}f_{0}(x^{1}, R_{x^1}v) (R_{x^{1}} + 2tA_{v,x^{1}}) \big] \hat{r}_{2}.  \\
\end{split}
\end{equation*}
Similar to \eqref{xv star1} and \eqref{xv star2}, using Lemma \ref{nabla xv b},
\begin{equation} \label{vv star1}  
\begin{split}
	(*)_{vv,1} \hat{r}_{2} 
	&= \begin{bmatrix}
	\nabla_{x}f_{0}(x^{1}, R_{x^1}v) \nabla_{v}(-t R_{x^{1}(x,v)}^1)
	\\
	\nabla_{x}f_{0}(x^{1}, R_{x^1}v) \nabla_{v}(-t R_{x^{1}(x,v)}^2)
	\end{bmatrix}
	\hat{r}_{2}  
	= 
	t^{2} 
	\begin{bmatrix}
	\nabla_{x}f_{0}(x^{1}, R_{x^1}v) \nabla_{x}(R_{x^{1}(x,v)}^1)
	\\
	\nabla_{x}f_{0}(x^{1}, R_{x^1}v) \nabla_{x}(R_{x^{1}(x,v)}^2)
	\end{bmatrix}
	\hat{r}_{2},  \\
\end{split}
\end{equation}
\begin{equation}  \label{vv star2}  
\begin{split}
(*)_{vv,2} \hat{r}_{2}
&= \begin{bmatrix}
\nabla_{v}f_{0}(x^{1}, R_{x^1}v) \nabla_{v}(R_{x^{1}(x,v)}^1 + 2t A_{v,x^{1}(x,v)}^1)
\\
\nabla_{v}f_{0}(x^{1}, R_{x^1}v) \nabla_{v}(R_{x^{1}(x,v)}^2 + 2t A_{v,x^{1}(x,v)}^2)
\end{bmatrix}
\hat{r}_{2}  \\
&=
-t
\begin{bmatrix}
\nabla_{v}f_{0}(x^{1}, R_{x^1}v) \nabla_{x}(R_{x^{1}(x,v)}^1)
\\
\nabla_{v}f_{0}(x^{1}, R_{x^1}v) \nabla_{x}(R_{x^{1}(x,v)}^2)
\end{bmatrix}
\hat{r}_{2}  
-
2t^{2}
\begin{bmatrix}
\nabla_{v}f_{0}(x^{1}, R_{x^1}v) \nabla_{x}(A_{v,x^{1}(x,v)}^1)
\\
\nabla_{v}f_{0}(x^{1}, R_{x^1}v) \nabla_{x}(A_{v,x^{1}(x,v)}^2)
\end{bmatrix}
\hat{r}_{2} 
\\
&\quad +
2t
\begin{bmatrix}
\nabla_{v}f_{0}(x^{1}, R_{x^1}v) \nabla_{v}A^{1}_{v,x^{1}}
\\
\nabla_{v}f_{0}(x^{1}, R_{x^1}v) \nabla_{v}A^{2}_{v,x^{1}} 
\end{bmatrix}
\hat{r}_{2}. 
\\
\end{split}
\end{equation}
Hence, we get
\begin{equation} \label{nabla_vv f case2}
\begin{split}
&\nabla_{vv} f(t,x,v) \\
&= \underline{(*)_{vv,1}}_{\eqref{vv star1}} + \underline{(*)_{vv,2}}_{\eqref{vv star2}}  \\
&\quad + (-tR_{x^1}) \big[ \nabla_{xx}f_{0}(x^{1}, R_{x^1}v)  (-tR_{x^1})  + \nabla_{vx}f_{0}(x^{1}, R_{x^1}v) (R_{x^1} + 2tA_{v,x^{1}}) \big]    \\
&\quad + (R_{x^{1}} + 2tA^{T}_{v,x^{1}}) \big[  \nabla_{xv}f_{0}(x^{1}, R_{x^1}v) (-tR_{x^1}) + \nabla_{vv}f_{0}(x^{1}, R_{x^1}v) (R_{x^{1}} + 2tA_{v,x^{1}}) \big].   \\
\end{split}
\end{equation}
 Then from \eqref{nabla_vv f case1} and \eqref{nabla_vv f case2} we get the following compatibility condition.
\begin{equation} \label{vv comp}
\begin{split}	
	&(-t ) \nabla_{xx}f_{0}(x^{1},v) (-t) + (-t )\nabla_{vx}f_{0}(x^{1},v)  + \nabla_{xv}f_{0}(x^{1},v) (-t ) + \nabla_{vv}f_{0}(x^{1},v)  \\	
	&= \underline{(*)_{vv,1}}_{\eqref{vv star1}} + \underline{(*)_{vv,2}}_{\eqref{vv star2}}  \\
	&\quad + (-tR_{x^1}) \nabla_{xx}f_{0}(x^{1}, R_{x^1}v)  (-tR_{x^1})  + (-tR_{x^1})\nabla_{vx}f_{0}(x^{1}, R_{x^1}v) (R_{x^1} + 2tA_{v,x^{1}})    \\
	&\quad + (R_{x^{1}} + 2tA^{T}_{v,x^{1}}) \nabla_{xv}f_{0}(x^{1}, R_{x^1}v) (-tR_{x^1}) + (R_{x^{1}} + 2tA^{T}_{v,x^{1}}) \nabla_{vv}f_{0}(x^{1}, R_{x^1}v) (R_{x^{1}} + 2tA_{v,x^{1}}).     \\
\end{split}
\end{equation}

\subsection{Condition for $\nabla_{xx}$} 
We split perturbed direction into \eqref{set R_sp}. $\nabla_{x}f(t,x,v)$ can be written as \eqref{c_3} or \eqref{c_4}, which are identical due to \eqref{c_v}. Using \eqref{c_3}, $\hat{r}_{1}$ of \eqref{set R_sp}, and notation \eqref{XV epsilon x}, \\

\begin{equation} \label{nabla_xx f case1}
\begin{split}
&\nabla_{xx} f(t,x,v) \hat{r}_1 \stackrel{c \leftrightarrow r}{=} \lim _{\epsilon\rightarrow 0+} \frac{1}{\epsilon}\left ( \nabla_{x}f(t,x+\epsilon \hat{r}_1,v) - \nabla_{x}f(t,x,v) \right ) \\ 
&= \lim_{\epsilon \rightarrow 0+} \frac{1}{\epsilon}\Big( \nabla_{x}\big[ f_0(X(0;t,x+\epsilon \hat{r}_1,v),V(0;t,x+\epsilon \hat{r}_1,v)) \big]  -  \nabla_{x}f_{0}(X(0),v) \Big)   \\
&=\lim_{\epsilon \rightarrow 0+} \frac{1}{\epsilon} \Big\{ \nabla_{x}f_{0}(X^{\varepsilon}(0), V^{\varepsilon}(0))  \nabla_{x}X^{\varepsilon}(0) + \nabla_{v}f_{0}(X^{\varepsilon}(0), V^{\varepsilon}(0)) \underbrace{\nabla_{x}V^{\varepsilon}(0)}_{=0}   - \nabla_{x}f_{0}(X(0),v) \Big\}   \\
&\stackrel{r \leftrightarrow c}{=} \nabla_{xx}f_{0}(x^{1},v) \lim_{s \rightarrow 0+}\nabla_{x}X(s)   \hat{r}_1  
= \nabla_{xx}f_{0}(x^{1},v)  \hat{r}_1,  
\end{split}
\end{equation}
where we have used \eqref{nabla XV_v+}, \eqref{nabla XV_x+}, $\nabla_{x}X^{\varepsilon}(0) = I_{2}$, and $\nabla_{x}V^{\varepsilon}(0)= 0$. Similarly, using \eqref{c_4}, $\hat{r}_{2}$ of \eqref{set R_sp}, and notation \eqref{XV epsilon x}, \\
\begin{equation} \notag
\begin{split}
&\nabla_{xx} f(t,x,v) \hat{r}_2 \stackrel{c \leftrightarrow r}{=} \lim _{\epsilon\rightarrow 0+} \frac{1}{\epsilon}\left ( \nabla_{x}f(t,x+\epsilon \hat{r}_2,v) - \nabla_{x}f(t,x,v) \right )\\ 
&=\lim_{\epsilon \rightarrow 0+} \frac{1}{\epsilon} \Big\{ \nabla_{x}f_{0}(X^{\varepsilon}(0), V^{\varepsilon}(0))  \nabla_{x}X^{\varepsilon}(0) + \nabla_{v}f_{0}(X^{\varepsilon}(0), V^{\varepsilon}(0))   \nabla_{x}V^{\varepsilon}(0) \\
&\quad\quad\quad   - \big( \nabla_x f_0(X(0), R_{x^1}v) R_{x^{1}}  - 2\nabla_v f_0(X(0),R_{x^1}v) A_{v,x^{1}} \big) \Big\}    \\
&= \lim_{\epsilon \rightarrow 0+} \frac{1}{\epsilon} \Big\{ \nabla_{x}f_{0}(X^{\varepsilon}(0), V^{\varepsilon}(0))  \nabla_{x}X^{\varepsilon}(0)   - \nabla_x f_0(X(0), R_{x^1}v) R_{x^{1}} \Big\} \\
&\quad + \lim_{\epsilon \rightarrow 0+} \frac{1}{\epsilon} \Big\{ \nabla_{v}f_{0}(X^{\varepsilon}(0), V^{\varepsilon}(0))  \nabla_{x}V^{\varepsilon}(0) + 2\nabla_v f_0(X(0),R_{x^1}v) A_{v,x^{1}} \Big\}   \\
&:=   I_{xx,1} + I_{xx,2}  ,
\end{split}
\end{equation}
where
\begin{equation*}
\begin{split}
I_{xx,1} &:= \lim_{\epsilon \rightarrow 0+} \frac{1}{\epsilon} \Big\{ \nabla_{x}f_{0}(X^{\varepsilon}(0), V^{\varepsilon}(0))  \nabla_{x}X^{\varepsilon}(0)  - \nabla_{x}f_{0}(X^{\varepsilon}(0), V^{\varepsilon}(0)) \lim_{s \rightarrow 0-}\nabla_{x}X(s)   \\
&\quad  + \Big( \nabla_{x}f_{0}(X^{\varepsilon}(0), V^{\varepsilon}(0)) R_{x^{1}}  - \nabla_x f_0(X(0), R_{x^1}v) R_{x^{1}} \Big) \Big\}  \\
&\stackrel{r \leftrightarrow c}{=} 
\underbrace{ \Big[  \nabla_{x}f_{0}(x^{1}, R_{x^1}v) \lim_{\epsilon \rightarrow 0+} \frac{1}{\epsilon} \Big(  \nabla_{x}X(0; t, x+\epsilon\hat{r}_{2}, v) - \lim_{s \rightarrow 0-}\nabla_{x}X(s)  \Big) \Big]^{T} }_{(*)_{xx,1}\hat{r}_{2}}  \\
&\quad + R_{x^{1}} \big[ \nabla_{xx}f_{0}(x^{1}, R_{x^1}v) R_{x^{1}} + \nabla_{vx}f_{0}(x^{1}, R_{x^1}v) (-2A_{v,x^{1}}) \big] \hat{r}_{2},
\end{split}
\end{equation*}
\begin{equation*}
\begin{split}
I_{xx,2} &:= \lim_{\epsilon \rightarrow 0+} \frac{1}{\epsilon} \Big\{ \nabla_{v}f_{0}(X^{\varepsilon}(0), V^{\varepsilon}(0))  \nabla_{x}V^{\varepsilon}(0) - \nabla_{v}f_{0}(X^{\varepsilon}(0), V^{\varepsilon}(0))\lim_{s \rightarrow 0-}\nabla_{x}V(s)   \\
&\quad - 2\nabla_{v}f_{0}(X^{\varepsilon}(0), V^{\varepsilon}(0)) A_{v,x^{1}} + 2\nabla_v f_0(X(0),R_{x^1}v) A_{v,x^{1}} \Big\} \\
&\stackrel{r\leftrightarrow c}{=}  \Big[ \nabla_{v}f_{0}(x^{1}, R_{x^1}v) \lim_{\epsilon \rightarrow 0+} \frac{1}{\epsilon} \Big(   \nabla_{x}V^{\varepsilon}(0) -  \lim_{s \rightarrow 0-}\nabla_{x}V(s) \Big) \Big]^{T}  \\
&\quad + (- 2A^{T}_{v,x^{1}}) \Big\{ \nabla_{xv}f_{0}(x^{1}, R_{x^1}v)\lim_{s \rightarrow 0-}\nabla_{x}X(s) + \nabla_{vv}f_{0}(x^{1}, R_{x^1}v)\lim_{s \rightarrow 0-}\nabla_{x}V(s) \Big\} \hat{r}_{2}  \\
&=  \underbrace{ \Big[ \nabla_{v}f_{0}(x^{1}, R_{x^1}v) \lim_{\epsilon \rightarrow 0+} \frac{1}{\epsilon} \Big(   \nabla_{x}V(0; t, x+\epsilon\hat{r}_{2}, v) -  \lim_{s \rightarrow 0-}\nabla_{x}V(s) \Big) \Big]^{T} }_{(*)_{xx,2}\hat{r}_{2}}   \\
&\quad + (- 2A^{T}_{v,x^{1}}) \big[ \nabla_{xv}f_{0}(x^{1}, R_{x^1}v) R_{x^{1}} + \nabla_{vv}f_{0}(x^{1}, R_{x^1}v)(-2 A_{v,x^{1}}) \big] \hat{r}_{2}.  \\
\end{split}
\end{equation*}
Similar to \eqref{xv star1} and \eqref{xv star2},
\begin{equation} \label{xx star1}  
\begin{split}
(*)_{xx,1} \hat{r}_{2} 
&= \begin{bmatrix}
\nabla_{x}f_{0}(x^{1}, R_{x^1}v) \nabla_{x}(R_{x^{1}(x,v)}^1)
\\
\nabla_{x}f_{0}(x^{1}, R_{x^1}v) \nabla_{x}(R_{x^{1}(x,v)}^2)
\end{bmatrix}
\hat{r}_{2},  \\
\end{split}
\end{equation}
\begin{equation}  \label{xx star2}  
\begin{split}
(*)_{xx,2} \hat{r}_{2}
&= \begin{bmatrix}
\nabla_{v}f_{0}(x^{1}, R_{x^1}v) \nabla_{x}(- 2 A_{v,x^{1}(x,v)}^1)
\\
\nabla_{v}f_{0}(x^{1}, R_{x^1}v) \nabla_{x}(- 2 A_{v,x^{1}(x,v)}^2)
\end{bmatrix}
\hat{r}_{2}.  \\
\end{split}
\end{equation}
Hence,
\begin{equation} \label{nabla_xx f case2}
\begin{split}
&\nabla_{xx} f(t,x,v) \\
&= \underline{(*)_{xx,1}}_{\eqref{xx star1}} + \underline{(*)_{xx,2}}_{\eqref{xx star2}}  \\
&\quad + R_{x^{1}} \big[ \nabla_{xx}f_{0}(x^{1}, R_{x^1}v) R_{x^{1}} + \nabla_{vx}f_{0}(x^{1}, R_{x^1}v) (-2A_{v,x^{1}}) \big]  \\
&\quad + (- 2A^{T}_{v,x^{1}}) \big[ \nabla_{xv}f_{0}(x^{1}, R_{x^1}v) R_{x^{1}} + \nabla_{vv}f_{0}(x^{1}, R_{x^1}v)(-2 A_{v,x^{1}}) \big]. 
\end{split}
\end{equation}
Then, from \eqref{nabla_xx f case1} and \eqref{nabla_xx f case2}, we get the following compatibility condition  
\begin{equation} \label{xx comp}
\begin{split}	
	&\nabla_{xx}f_{0}(x^{1},v)  \\
	&= \underline{(*)_{xx,1}}_{\eqref{xx star1}} + \underline{(*)_{xx,2}}_{\eqref{xx star2}}  \\
	&\quad + R_{x^{1}} \nabla_{xx}f_{0}(x^{1}, R_{x^1}v) R_{x^{1}} + R_{x^{1}} \nabla_{vx}f_{0}(x^{1}, R_{x^1}v) (-2A_{v,x^{1}})  \\
	&\quad + (- 2A^{T}_{v,x^{1}}) \nabla_{xv}f_{0}(x^{1}, R_{x^1}v) R_{x^{1}} + (- 2A^{T}_{v,x^{1}}) \nabla_{vv}f_{0}(x^{1}, R_{x^1}v)(-2 A_{v,x^{1}}). \\
\end{split}
\end{equation}

\subsection{Condition for $\nabla_{vx}$}  We split perturbed direction into \eqref{set R_vel}. $\nabla_{x}f(t,x,v)$ can be written as \eqref{c_3} or \eqref{c_4}. Using \eqref{c_3}, $\hat{r}_{1}$ of \eqref{set R_vel}, and notation \eqref{XV epsilon v}, \\
\begin{equation} \label{nabla_vx f case1}
\begin{split}
&\nabla_{vx} f(t,x,v) \hat{r}_1 \stackrel{c\leftrightarrow r}{=} \lim _{\epsilon\rightarrow 0+} \frac{1}{\epsilon}\left ( \nabla_{x}f(t,x,v+\epsilon \hat{r}_1) - \nabla_{x}f(t,x,v) \right ) \\ 
&=\lim_{\epsilon \rightarrow 0+} \frac{1}{\epsilon}\Big( \nabla_{x}\big[ f_0(X(0;t,x ,v+\epsilon \hat{r}_1),V(0;t,x ,v+\epsilon \hat{r}_1)) \big]  -  \nabla_{x}f_{0}(X(0),v) \Big)   \\
&=\lim_{\epsilon \rightarrow 0+} \frac{1}{\epsilon} \Big\{ \nabla_{x}f_{0}(X^{\varepsilon}(0), V^{\varepsilon}(0))  \nabla_{x}X^{\varepsilon}(0) + \nabla_{v}f_{0}(X^{\varepsilon}(0), V^{\varepsilon}(0))  \nabla_{x}V^{\varepsilon}(0)   - \nabla_{x}f_{0}(X(0),v) \Big\}   \\
&=\lim_{\epsilon \rightarrow 0+} \frac{1}{\epsilon} \Big\{ \nabla_{x}f_{0}(X^{\varepsilon}(0), v+\epsilon \hat{r}_{1}  )  \nabla_{x}X^{\varepsilon}(0) + \nabla_{v}f_{0}(X^{\varepsilon}(0), v+\epsilon \hat{r}_{1} ) \underbrace{ \nabla_{x}V^{\varepsilon}(0) }_{=0}  - \nabla_{x}f_{0}(X(0),v) \Big\} \\
&\stackrel{r\leftrightarrow c}{=} \nabla_{xx}f_{0}(x^{1}, v)\lim_{s \rightarrow 0+}\nabla_{v}X(s)\hat{r}_{1} + \nabla_{vx}f_{0}(x^{1}, v) \lim_{s \rightarrow 0+}\nabla_{v}V(s)\hat{r}_{1}   \\
&= \big( \nabla_{xx}f_{0}(x^{1}, v)(-t) + \nabla_{vx}f_{0}(x^{1}, v)  \big) \hat{r}_{1} ,  \\
\end{split}
\end{equation}
where we have used \eqref{nabla XV_v+}, \eqref{nabla XV_x+}, $\nabla_{x}X^{\varepsilon}(0) = I_{2}$, and $\nabla_{x}V^{\varepsilon}(0)= 0$. Similalry, using \eqref{c_4}, $\hat{r}_{2}$ of \eqref{set R_vel}, and notation \eqref{XV epsilon v}, \\
\begin{equation} \notag
\begin{split}
&\nabla_{vx} f(t,x,v) \hat{r}_2 \stackrel{c\leftrightarrow r}{=} \lim _{\epsilon\rightarrow 0+} \frac{1}{\epsilon}\left ( \nabla_{x}f(t,x,v+\epsilon \hat{r}_2) - \nabla_{x}f(t,x,v) \right )\\ 
&=\lim_{\epsilon \rightarrow 0+} \frac{1}{\epsilon} \Big\{ \nabla_{x}f_{0}(X^{\varepsilon}(0), V^{\varepsilon}(0))  \nabla_{x}X^{\varepsilon}(0) + \nabla_{v}f_{0}(X^{\varepsilon}(0), V^{\varepsilon}(0))   \nabla_{x}V^{\varepsilon}(0) \\
&\quad\quad\quad   - \big( \nabla_x f_0(X(0), R_{x^1}v) R_{x^{1}}  - 2\nabla_v f_0(X(0),R_{x^1}v) A_{v,x^{1}} \big) \Big\}  \\
&= \lim_{\epsilon \rightarrow 0+} \frac{1}{\epsilon} \Big\{ \nabla_{x}f_{0}(X^{\varepsilon}(0), V^{\varepsilon}(0))  \nabla_{x}X^{\varepsilon}(0)   - \nabla_x f_0(X(0), R_{x^1}v) R_{x^{1}} \Big\} \\
&\quad + \lim_{\epsilon \rightarrow 0+} \frac{1}{\epsilon} \Big\{ \nabla_{v}f_{0}(X^{\varepsilon}(0), V^{\varepsilon}(0))  \nabla_{x}V^{\varepsilon}(0) + 2\nabla_v f_0(X(0),R_{x^1}v) A_{v,x^{1}} \Big\}     \\
&:=   I_{vx,1} + I_{vx,2},  
\end{split}
\end{equation}
where
\begin{equation*}
\begin{split}
I_{vx,1} &:= \lim_{\epsilon \rightarrow 0+} \frac{1}{\epsilon} \Big\{ \nabla_{x}f_{0}(X^{\varepsilon}(0), V^{\varepsilon}(0))  \nabla_{x}X^{\varepsilon}(0)  - \nabla_{x}f_{0}(X^{\varepsilon}(0), V^{\varepsilon}(0)) \lim_{s \rightarrow 0-}\nabla_{x}X(s)   \\
&\quad  + \Big( \nabla_{x}f_{0}(X^{\varepsilon}(0), V^{\varepsilon}(0)) R_{x^{1}}  - \nabla_x f_0(X(0), R_{x^1}v) R_{x^{1}} \Big) \Big\}  \\
&\stackrel{r\leftrightarrow c}{=} \Big[ \nabla_{x}f_{0}(x^{1}, R_{x^1}v) \lim_{\epsilon \rightarrow 0+} \frac{1}{\epsilon} \Big(  \nabla_{x}X^{\varepsilon}(0) - \lim_{s \rightarrow 0-}\nabla_{x}X(s)  \Big) \Big]^{T}   \\
&\quad + R_{x^{1}}  \Big\{ \nabla_{xx}f_{0}(x^{1}, R_{x^1}v)\lim_{s \rightarrow 0-}\nabla_{v}X(s) + \nabla_{vx}f_{0}(x^{1}, R_{x^1}v)\lim_{s \rightarrow 0-}\nabla_{v}V(s) \Big\} \hat{r}_{2}   \\
&=  \underbrace{ \Big[ \nabla_{x}f_{0}(x^{1}, R_{x^1}v) \lim_{\epsilon \rightarrow 0+} \frac{1}{\epsilon} \Big(  \nabla_{x}X(0; t, x, v+\epsilon \hat{r}_{2}) - \lim_{s \rightarrow 0-}\nabla_{x}X(s)  \Big) \Big]^{T} }_{(*)_{vx,1}
\hat{r}_{2} }   \\
&\quad + R_{x^{1}}  \big[ \nabla_{xx}f_{0}(x^{1}, R_{x^1}v)(-tR_{x^1})  +  \nabla_{vx}f_{0}(x^{1}, R_{x^1}v) (R_{x^{1}} + 2tA_{v,x^{1}}) \big] \hat{r}_{2},
\end{split}
\end{equation*}
\begin{equation*}
\begin{split}
I_{vx,2} &:= \lim_{\epsilon \rightarrow 0+} \frac{1}{\epsilon} \Big\{ \nabla_{v}f_{0}(X^{\varepsilon}(0), V^{\varepsilon}(0))  \nabla_{x}V^{\varepsilon}(0) - \nabla_{v}f_{0}(X^{\varepsilon}(0), V^{\varepsilon}(0))\lim_{s \rightarrow 0-}\nabla_{x}V(s)   \\
&\quad - 2\nabla_{v}f_{0}(X^{\varepsilon}(0), V^{\varepsilon}(0)) A_{v,x^{1}} + 2\nabla_v f_0(X(0),R_{x^1}v) A_{v,x^{1}} \Big\} \\
&\stackrel{r\leftrightarrow c}{=} \Big[ \nabla_{v}f_{0}(x^{1}, R_{x^1}v) \lim_{\epsilon \rightarrow 0+} \frac{1}{\epsilon} \Big(   \nabla_{x}V^{\varepsilon}(0) -  \lim_{s \rightarrow 0-}\nabla_{x}V(s) \Big) \Big]^{T}  \\
&\quad + (-2A^{T}_{v,x^{1}}) \Big\{ \nabla_{xv}f_{0}(x^{1}, R_{x^1}v)\lim_{s \rightarrow 0-}\nabla_{v}X(s) + \nabla_{vv}f_{0}(x^{1}, R_{x^1}v)\lim_{s \rightarrow 0-}\nabla_{v}V(s) \Big\} \hat{r}_{2}   \\
&=  \underbrace{ \Big[ \nabla_{v}f_{0}(x^{1}, R_{x^1}v) \lim_{\epsilon \rightarrow 0+} \frac{1}{\epsilon} \Big(   \nabla_{x}V(0; t, x, v+\epsilon \hat{r}_{2}) -  \lim_{s \rightarrow 0-}\nabla_{x}V(s) \Big) \Big]^{T} }_{(*)_{vx,2}\hat{r}_{2}}   \\
&\quad + (-2A^{T}_{v,x^{1}}) \big[ \nabla_{xv}f_{0}(x^{1}, R_{x^1}v) (-tR_{x^{1}}) + \nabla_{vv}f_{0}(x^{1}, R_{x^1}v) ( R_{x^{1}} + 2tA_{v,x^{1}}) \big] \hat{r}_{2}.
\end{split}
\end{equation*}
Similar to \eqref{xv star1} and \eqref{xv star2},
\begin{equation} \label{vx star1}  
\begin{split}
(*)_{vx,1} \hat{r}_{2} 
&= \begin{bmatrix}
\nabla_{x}f_{0}(x^{1}, R_{x^1}v) \nabla_{v}(R_{x^{1}(x,v)}^1)
\\
\nabla_{x}f_{0}(x^{1}, R_{x^1}v) \nabla_{v}(R_{x^{1}(x,v)}^2)
\end{bmatrix}
\hat{r}_{2}  
= -t
\begin{bmatrix}
\nabla_{x}f_{0}(x^{1}, R_{x^1}v) \nabla_{x}(R_{x^{1}(x,v)}^1)
\\
\nabla_{x}f_{0}(x^{1}, R_{x^1}v) \nabla_{x}(R_{x^{1}(x,v)}^2)
\end{bmatrix}
\hat{r}_{2}, 
\\
\end{split}
\end{equation}
\begin{equation}  \label{vx star2}  
\begin{split}
(*)_{vx,2} \hat{r}_{2}
&= \begin{bmatrix}
\nabla_{v}f_{0}(x^{1}, R_{x^1}v) \nabla_{v}(- 2 A_{v,x^{1}(x,v)}^1)
\\
\nabla_{v}f_{0}(x^{1}, R_{x^1}v) \nabla_{v}(- 2 A_{v,x^{1}(x,v)}^2)
\end{bmatrix}
\hat{r}_{2}  \\
&=
2t
\begin{bmatrix}
\nabla_{v}f_{0}(x^{1}, R_{x^1}v) \nabla_{x}(A_{v,x^{1}(x,v)}^1)
\\
\nabla_{v}f_{0}(x^{1}, R_{x^1}v) \nabla_{x}(A_{v,x^{1}(x,v)}^2)
\end{bmatrix}
\hat{r}_{2} 
-2 
\begin{bmatrix}
\nabla_{v}f_{0}(x^{1}, R_{x^1}v) \nabla_{v}A^{1}_{v,x^{1}}
\\
\nabla_{v}f_{0}(x^{1}, R_{x^1}v) \nabla_{v}A^{2}_{v,x^{1}} 
\end{bmatrix}
\hat{r}_{2}. 
\\
\end{split}
\end{equation}
Hence,
\begin{equation} \label{nabla_vx f case2}
\begin{split}
&\nabla_{vx} f(t,x,v) \\
&= \underline{(*)_{vx,1}}_{\eqref{vx star1}} + \underline{(*)_{vx,2}}_{\eqref{vx star2}}  \\
&\quad + R_{x^{1}}  \big[ \nabla_{xx}f_{0}(x^{1}, R_{x^1}v)(-tR_{x^1})  +  \nabla_{vx}f_{0}(x^{1}, R_{x^1}v) (R_{x^{1}} + 2tA_{v,x^{1}}) \big]   \\
&\quad + (-2A^{T}_{v,x^{1}}) \big[ \nabla_{xv}f_{0}(x^{1}, R_{x^1}v) (-tR_{x^{1}}) + \nabla_{vv}f_{0}(x^{1}, R_{x^1}v) ( R_{x^{1}} + 2tA_{v,x^{1}}) \big]. 
\end{split}
\end{equation}
Then from \eqref{nabla_vx f case1} and \eqref{nabla_vx f case2} we get the following compatibility condition
\begin{equation} \label{vx comp}
\begin{split}	
	&\nabla_{xx}f_{0}(x^{1}, v)(-t) + \nabla_{vx}f_{0}(x^{1}, v)  \\
	&= \underline{(*)_{vx,1}}_{\eqref{vx star1}} + \underline{(*)_{vx,2}}_{\eqref{vx star2}}  \\
	&\quad + R_{x^{1}} \nabla_{xx}f_{0}(x^{1}, R_{x^1}v)(-tR_{x^1})  +  R_{x^{1}} \nabla_{vx}f_{0}(x^{1}, R_{x^1}v) (R_{x^{1}} + 2tA_{v,x^{1}})   \\
	&\quad + (-2A^{T}_{v,x^{1}}) \nabla_{xv}f_{0}(x^{1}, R_{x^1}v) (-tR_{x^{1}}) + (-2A^{T}_{v,x^{1}}) \nabla_{vv}f_{0}(x^{1}, R_{x^1}v) ( R_{x^{1}} + 2tA_{v,x^{1}}).  \\
\end{split}
\end{equation}

\subsection{Compatibility conditions for transpose : $\nabla_{xv}^{T} = \nabla_{vx}$ and $\nabla_{xx}^{T} = \nabla_{xx}$}

First, we claim that \eqref{xv comp}, \eqref{vv comp}, \eqref{xx comp}, and \eqref{vx comp} imply the following four conditions for $(x^{1}, v)\in \gamma_{-}$
\begin{eqnarray}  
\nabla_{xv}f_{0}(x^{1},v)  
&=&  R_{x^{1}}\nabla_{xv}f_{0}(x^{1}, R_{x^1}v) R_{x^{1}}  + R_{x^{1}}\nabla_{vv}f_{0}(x^{1}, R_{x^1}v) (-2A_{v, x^{1}})  \notag \\
&&\quad + \begin{bmatrix}
\nabla_{v}f_{0}(x^{1}, R_{x^1}v) \nabla_{x}(R_{x^{1}(x,v)}^1)  \\
\nabla_{v}f_{0}(x^{1}, R_{x^1}v) \nabla_{x}(R_{x^{1}(x,v)}^2)
\end{bmatrix} ,\quad x^{1}=x^{1}(x,v),  \label{Cond2 1}  \\
\nabla_{xx}f_{0}(x^{1},v)  
&=& R_{x^{1}} \nabla_{xx}f_{0}(x^{1}, R_{x^1}v) R_{x^{1}} + R_{x^{1}} \nabla_{vx}f_{0}(x^{1}, R_{x^1}v)(-2A_{v,x^{1}}) \notag \\
&&\quad + (-2A^{T}_{v,x^{1}}) \nabla_{xv}f_{0}(x^{1}, R_{x^1}v) R_{x^{1}} + (-2A^{T}_{v,x^{1}}) \nabla_{vv}f_{0}(x^{1}, R_{x^1}v)  (-2A_{v,x^{1}})  \notag \\
&&\quad + \begin{bmatrix}
\nabla_{x}f_{0}(x^{1}, R_{x^1}v) \nabla_{x}(R_{x^{1}(x,v)}^1)
\\
\nabla_{x}f_{0}(x^{1}, R_{x^1}v) \nabla_{x}(R_{x^{1}(x,v)}^2)
\end{bmatrix} 
- 2
\begin{bmatrix}
\nabla_{v}f_{0}(x^{1}, R_{x^1}v) \nabla_{x}(A_{v,x^{1}(x,v)}^1)
\\
\nabla_{v}f_{0}(x^{1}, R_{x^1}v) \nabla_{x}(A_{v,x^{1}(x,v)}^2)
\end{bmatrix},
\label{Cond2 2}  \\
\nabla_{vv}f_{0}(x^{1},v)  
&=& R_{x^{1}}\nabla_{vv}f_{0}(x^{1}, R_{x^1}v) R_{x^{1}},   \label{Cond2 3}  \\
\nabla_{vx}f_{0}(x^{1},v)
&=& R_{x^{1}}\nabla_{vx}f_{0}(x^{1}, R_{x^1}v) R_{x^{1}} +  (-2A^{T}_{v,x^{1}})\nabla_{vv}f_{0}(x^{1}, R_{x^1}v)R_{x^{1}} \notag \\
&&\quad -2 
\begin{bmatrix}
\nabla_{v}f_{0}(x^{1}, R_{x^1}v) \nabla_{v}A^{1}_{v,x^{1}}
\\
\nabla_{v}f_{0}(x^{1}, R_{x^1}v) \nabla_{v}A^{2}_{v,x^{1}} 
\end{bmatrix}. \label{Cond2 4} 
\end{eqnarray}
\eqref{xx comp} is just identical to \eqref{Cond2 2}. Then applying \eqref{xx comp} to \eqref{vx comp} and \eqref{xv comp}, we obtain \eqref{Cond2 1} and \eqref{Cond2 4}, respectively. Finally, applying \eqref{Cond2 1}, \eqref{Cond2 2}, and \eqref{Cond2 4} to \eqref{vv comp}, we obtain \eqref{Cond2 3} which is true by taking $\nabla_{v}^{2}$ to \eqref{BC} directly.  \\

From \eqref{Cond2 1}--\eqref{Cond2 4}, we must check conditions to guarantee necessary conditions, $\nabla_{xv}^{T} = \nabla_{vx}$ and $\nabla_{xx}^{T} = \nabla_{xx}$. \\

\subsubsection{$\nabla_{xv}^T=\nabla_{vx}$}
From \eqref{Cond2 1} and \eqref{Cond2 4}, we need 
\Be \label{T invariant}
\begin{bmatrix}
	\nabla_{v}f_{0}(x^{1}, R_{x^1}v) \nabla_{x}(R_{x^{1}(x,v)}^1)
	\\
	\nabla_{v}f_{0}(x^{1}, R_{x^1}v) \nabla_{x}(R_{x^{1}(x,v)}^2)
\end{bmatrix}^{T}
=
-2 
\begin{bmatrix}
	\nabla_{v}f_{0}(x^{1}, R_{x^1}v) \nabla_{v}A^{1}_{v,x^{1}}
	\\
	\nabla_{v}f_{0}(x^{1}, R_{x^1}v) \nabla_{v}A^{2}_{v,x^{1}}
\end{bmatrix}.
\Ee
To check \eqref{T invariant}, we explicitly compute $\nabla_x(R_{x^1(x,v)}^1),\nabla_x(R_{x^1(x,v)}^2),\nabla_v(-2A^1_{v,x^1}),$ and $\nabla_v(-2A_{v,x^1}^2)$ in the following Lemma. 
\begin{lemma} \label{d_RA}
	Recall reflection operator $R_{x^1}$ in \eqref{BC} and $A_{v,x^1}$ in \eqref{def A},
	\begin{equation*}
	A_{v,x^1} := \left[ \left((v\cdot n(x^1))I +(n(x^1)\otimes v)\right) \left(I-\frac{v\otimes n(x^1)}{v\cdot n(x^1)}\right)\right].
	\end{equation*}
	We write that $A^i$ is the $i$th column of matrix $A$ and $\nabla_vA^i_{v,y}$ be the $v$-derivative of $A_{v,y}^i$ for $1\leq i \leq 2$ and $(v,y) \in \R^2 \times \partial \O$. Then, 
	\begin{align*}
	&\nabla_x (R_{x^1(x,v)}^1) = \begin{bmatrix} 
	\dfrac{-4v_2n_1n_2}{v\cdot n(x^1)}  & \dfrac{4v_1n_1n_2}{v\cdot n(x^1)} \\ 
	\dfrac{-2v_2(n_2^2-n_1^2)}{v\cdot n(x^1)} & \dfrac{2v_1(n_2^2-n_1^2)}{v\cdot n(x^1)}
	\end{bmatrix}, \quad 
	\nabla_x (R_{x^1(x,v)}^2)= \begin{bmatrix} 
	\dfrac{-2v_2(n_2^2-n_1^2)}{v\cdot n(x^1)} & \dfrac{2v_1(n_2^2-n_1^2)}{v\cdot n(x^1)}\\
	\dfrac{4v_2n_1n_2}{v\cdot n(x^1)}  & \dfrac{-4v_1n_1n_2}{v\cdot n(x^1)} 
	\end{bmatrix},\\
	&\nabla_v(-2A_{v,x^1}^1)= \begin{bmatrix}
	-\dfrac{2v_2^2n_1}{(v\cdot n(x^1))^2} & -2n_2-\dfrac{2v_1^2n_1^2n_2}{(v\cdot n(x^1))^2} + \dfrac{4v_1v_2n_1^3}{(v\cdot n(x^1))^2} + \dfrac{2v_2^2n_1^2n_2}{(v\cdot n(x^1))^2} \\
	-\dfrac{2v_2^2n_2}{(v\cdot n(x^1))^2} & 2n_1 -\dfrac{2v_1^2 n_1n_2^2}{(v\cdot n(x^1))^2} + \dfrac{4v_1v_2 n_1^2n_2}{(v\cdot n(x^1))^2} + \dfrac{2v_2^2 n_1n_2^2}{(v\cdot n(x^1))^2}
	\end{bmatrix},\\
	&\nabla_v(-2A_{v,x^1}^2)= \begin{bmatrix}
	2n_2+\dfrac{2v_1^2n_1^2n_2}{(v\cdot n(x^1))^2} + \dfrac{4v_1v_2n_1n_2^2}{(v\cdot n(x^1))^2} - \dfrac{2v_2^2n_1^2n_2}{(v\cdot n(x^1))^2} & -\dfrac{2v_1^2n_1}{(v\cdot n(x^1))^2}  \\
	-2n_1 -\dfrac{2v_2^2 n_1n_2^2}{(v\cdot n(x^1))^2} + \dfrac{4v_1v_2 n_2^3}{(v\cdot n(x^1))^2} + \dfrac{2v_1^2 n_1n_2^2}{(v\cdot n(x^1))^2} & -\dfrac{2v_1^2n_2}{(v\cdot n(x^1))^2} 
	\end{bmatrix},
	\end{align*}
	where $v_i$ be the $i$th component of $v$. We denote the $i$th component $n_i(x,v)$ of $n(x^1)$ as $n_i$, that is, $n_i$ depends on $x,v$. Moreover, the following identity holds that 
	\begin{equation}\label{prop d_R}
	\nabla_{x}(R_{x^1(x,v)}^1)v =0, \quad \nabla_{x}(R_{x^1(x,v)}^2) v =0. 
	\end{equation}
\end{lemma}
\begin{proof}
	Recall the definition of the reflection matrix $R_{x^1}$ and $-2A_{v,x^1}$:
	\begin{align*}
	R_{x^1}&=I-2n(x^1)\otimes n(x^1) = \begin{bmatrix}
	1-2n_1^2 & -2n_1n_2 \\ 
	-2n_1n_2 & 1-2n_2^2
	\end{bmatrix},\\
	-2A_{v,x^1}&= -2 \left[ \left((v\cdot n(x^1))I +(n(x^1)\otimes v)\right) \left(I-\frac{v\otimes n(x^1)}{v\cdot n(x^1)}\right)\right]\\
	&=\begin{bmatrix}
	-2v_2n_2 -\dfrac{2v_1v_2n_1n_2}{v\cdot n(x^1)} +\dfrac{2v_2^2 n_1^2}{ v\cdot n(x^1)} & 2v_1n_2 + \dfrac{2v_1^2n_1n_2}{v\cdot n(x^1)} -\dfrac{2v_1v_2n_1^2}{ v\cdot n(x^1)} \\
	2v_2n_1 -\dfrac{2v_1v_2n_2^2}{v\cdot n(x^1)} +\dfrac{2v_2^2n_1n_2}{v\cdot n(x^1)} & -2v_1n_1 -\dfrac{2v_1v_2n_1n_2}{v\cdot n(x^1)} + \dfrac{2v_1^2n_2^2}{v \cdot n(x^1)}
	\end{bmatrix}.
	\end{align*}
	To find $\nabla_x (R_{x^1(x,v)}^1),\nabla_x (R_{x^1(x,v)}^2)$, we use \eqref{normal} in Lemma \ref{d_n}:
	\begin{equation} \label{comp_dn}
	\nabla_x [n(x^1(x,v))] = I-\frac{v\otimes n(x^1)}{ v\cdot n(x^1)}= \begin{bmatrix}
	\dfrac{v_2n_2}{v\cdot n(x^1)} & -\dfrac{v_1n_2}{v\cdot n(x^1)} \\
	-\dfrac{v_2n_1}{v \cdot n(x^1)} & \dfrac{v_1n_1}{v\cdot n(x^1)}
	\end{bmatrix}. 
	\end{equation}
	Firstly, we directly calculate $\nabla_x (R_{x^1(x,v)}^1)$ and $\nabla_x(R_{x^1(x,v)}^2)$ using \eqref{comp_dn}: 
	\begin{align*}
	\nabla_x (R_{x^1(x,v)}^1) = \nabla_x \begin{bmatrix}
	1-2n_1^2 \\ -2n_1n_2
	\end{bmatrix}= 
	\begin{bmatrix} 
	\dfrac{-4v_2n_1n_2}{v\cdot n(x^1)}  & \dfrac{4v_1n_1n_2}{v\cdot n(x^1)} \\ 
	\dfrac{-2v_2(n_2^2-n_1^2)}{v\cdot n(x^1)} & \dfrac{2v_1(n_2^2-n_1^2)}{v\cdot n(x^1)}
	\end{bmatrix},\\ 
	\nabla_x (R_{x^1(x,v)}^2) = \nabla_x \begin{bmatrix}
	-2n_1n_2 \\ 1-2n_2^2
	\end{bmatrix}= 
	\begin{bmatrix} 
	\dfrac{-2v_2(n_2^2-n_1^2)}{v\cdot n(x^1)} & \dfrac{2v_1(n_2^2-n_1^2)}{v\cdot n(x^1)}\\
	\dfrac{4v_2n_1n_2}{v\cdot n(x^1)}  & \dfrac{-4v_1n_1n_2}{v\cdot n(x^1)}
	\end{bmatrix}.
	\end{align*}
	Next, we calculate the $v$-derivative of $[-2A_{v,x^1}^1]$:
	\begin{align*}
	(\nabla_v (-2A_{v,x^1}^1))_{(1,1)} &= -\frac{2v_2n_1n_2 (v \cdot n(x^1))-2v_1v_2n_1^2n_2}{(v\cdot n(x^1))^2}-\frac{2v_2^2n_1^3}{(v\cdot n(x^1))^2}=-\frac{2v_2^2n_1}{(v\cdot n(x^1))^2}, \\
	(\nabla_v (-2A_{v,x^1}^1))_{(1,2)} &=-2n_2-\frac{2v_1n_1n_2(v\cdot n(x^1))-2v_1v_2n_1n_2^2}{(v\cdot n(x^1))^2} +\frac{4v_2n_1^2 (v\cdot n(x^1)) -2v_2^2n_1^2n_2}{(v\cdot n(x^1))^2}\\
	&= -2n_2-\dfrac{2v_1^2n_1^2n_2}{(v\cdot n(x^1))^2} + \dfrac{4v_1v_2n_1^3}{(v\cdot n(x^1))^2} + \dfrac{2v_2^2n_1^2n_2}{(v\cdot n(x^1))^2},\\
	(\nabla_v (-2A_{v,x^1}^1))_{(2,1)}&=-\frac{2v_2n_2^2 (v\cdot n(x^1)) -2v_1v_2n_1n_2^2}{(v\cdot n(x^1))^2}-\frac{2v_2^2n_1^2n_2}{(v\cdot n(x^1))^2}=-\frac{2v_2^2n_2}{(v\cdot n(x^1))^2},\\
	(\nabla_v (-2A_{v,x^1}^1))_{(2,2)}&=2n_1-\frac{2v_1n_2^2(v \cdot n(x^1))-2v_1v_2n_2^3}{(v\cdot n(x^1))^2}+\frac{4v_2n_1n_2 (v\cdot n(x^1)) -2v_2^2n_1n_2^2}{(v\cdot n(x^1))^2}\\
	&= 2n_1 -\dfrac{2v_1^2 n_1n_2^2}{(v\cdot n(x^1))^2} + \dfrac{4v_1v_2 n_1^2n_2}{(v\cdot n(x^1))^2} + \dfrac{2v_2^2 n_1n_2^2}{(v\cdot n(x^1))^2}.
	\end{align*}
	Similarly, we deduce the $v$-derivative of $[-2A_{v,x^1}^2]$. We derived $\nabla_x(R_{x^1(x,v)}^1)$ and $\nabla_x(R_{x^1(x,v)}^2)$, and then \eqref{prop d_R} follows from direct calculation that 
	\begin{align*}
	\nabla_{x}(R_{x^1(x,v)}^1) v =  \begin{bmatrix} 
	\dfrac{-4v_2n_1n_2}{v\cdot n(x^1)}  & \dfrac{4v_1n_1n_2}{v\cdot n(x^1)} \\ 
	\dfrac{-2v_2(n_2^2-n_1^2)}{v\cdot n(x^1)} & \dfrac{2v_1(n_2^2-n_1^2)}{v\cdot n(x^1)}
	\end{bmatrix}\begin{bmatrix} 
	v_1 \\ v_2
	\end{bmatrix} = 0, \\
	\nabla_x(R_{x^1(x,v)}^2) v =\begin{bmatrix} 
	\dfrac{-2v_2(n_2^2-n_1^2)}{v\cdot n(x^1)} & \dfrac{2v_1(n_2^2-n_1^2)}{v\cdot n(x^1)}\\
	\dfrac{4v_2n_1n_2}{v\cdot n(x^1)}  & \dfrac{-4v_1n_1n_2}{v\cdot n(x^1)} 
	\end{bmatrix} \begin{bmatrix}
	v_1 \\ v_2
	\end{bmatrix}=0. 
	\end{align*}
\end{proof}
Back to the point, we find the condition of $\nabla_v f_0(x^1,Rv)$ satisfying \eqref{T invariant}. Since 
\begin{align*}
\begin{bmatrix}
\nabla_{v}f_{0}(x^{1}, Rv) \nabla_{x}(R_{x^{1}(x,v)}^1)
\\
\nabla_{v}f_{0}(x^{1}, Rv) \nabla_{x}(R_{x^{1}(x,v)}^2)
\end{bmatrix}^{T}=\begin{bmatrix}
\nabla_vf_0(x^1,Rv) \dfrac{\partial}{\partial x_1} (R_{x^1(x,v)}^1) & \nabla_vf_0(x^1,Rv) \dfrac{\partial}{\partial x_1} (R_{x^1(x,v)}^2)\\
\nabla_vf_0(x^1,Rv) \dfrac{\partial}{\partial x_2} (R_{x^1(x,v)}^1) & \nabla_vf_0(x^1,Rv) \dfrac{\partial}{\partial x_2} (R_{x^1(x,v)}^2)
\end{bmatrix}, \\
-2 
\begin{bmatrix}
\nabla_{v}f_{0}(x^{1}, Rv) \nabla_{v}A^{1}_{v,x^{1}}
\\
\nabla_{v}f_{0}(x^{1}, Rv) \nabla_{v}A^{2}_{v,x^{1}} .
\end{bmatrix}=\begin{bmatrix}
\nabla_v f_0(x^1,R_{x^1}v) \dfrac{\partial}{\partial v_1}(-2A^1_{v,x^1}) & \nabla_v f_0(x^1,R_{x^1}v) \dfrac{\partial}{\partial v_2}(-2A^1_{v,x^1})\\
\nabla_v f_0(x^1,R_{x^1}v) \dfrac{\partial}{\partial v_1}(-2A^2_{v,x^1}) & \nabla_v f_0(x^1,R_{x^1}v) \dfrac{\partial}{\partial v_2}(-2A^2_{v,x^1})
\end{bmatrix},
\end{align*}
it suffices to find the condition of $\nabla_v f_0(x^1,R_{x^1}v)$ such that 
\begin{align*}
\nabla_vf_0(x^1,R_{x^1}v) \left( \dfrac{\partial}{\partial x_1} (R_{x^1(x,v)}^1) -\dfrac{\partial}{\partial v_1} (-2A_{v,x^1}^1)\right) =0, \quad 
\nabla_vf_0(x^1,R_{x^1}v) \left( \dfrac{\partial}{\partial x_2} (R_{x^1(x,v)}^1) -\dfrac{\partial}{\partial v_1} (-2A_{v,x^1}^2)\right) =0,\\
\nabla_vf_0(x^1,R_{x^1}v) \left( \dfrac{\partial}{\partial x_1} (R_{x^1(x,v)}^2) -\dfrac{\partial}{\partial v_2} (-2A_{v,x^1}^1)\right) =0,\quad
\nabla_vf_0(x^1,R_{x^1}v) \left( \dfrac{\partial}{\partial x_2} (R_{x^1(x,v)}^2) -\dfrac{\partial}{\partial v_2} (-2A_{v,x^1}^2)\right) =0.
\end{align*}
We denote column vectors 
\begin{align*}
K_1 := \dfrac{\partial}{\partial x_1} (R_{x^1(x,v)}^1) -\dfrac{\partial}{\partial v_1} (-2A_{v,x^1}^1), \quad K_2:= \dfrac{\partial}{\partial x_2} (R_{x^1(x,v)}^1) -\dfrac{\partial}{\partial v_1} (-2A_{v,x^1}^2),\\
K_3 := \dfrac{\partial}{\partial x_1} (R_{x^1(x,v)}^2) -\dfrac{\partial}{\partial v_2} (-2A_{v,x^1}^1), \quad K_4:=\dfrac{\partial}{\partial x_2} (R_{x^1(x,v)}^2) -\dfrac{\partial}{\partial v_2} (-2A_{v,x^1}^2).
\end{align*}
To determine whether $\nabla_v f_0(x^1,R_{x^1}v)$ is a nonzero vector or not for \eqref{T invariant}, we need to calculate the following determinant 
\begin{align*}
\det \begin{bmatrix} \vert & \vert \\ K_i & K_j \\ \vert & \vert  \end{bmatrix}, \quad 1\leq i < j \leq 4.
\end{align*} 
If every determinant has a value of zero, then $\nabla_v f_0(x^1,R_{x^1}v)$ satisfying \eqref{T invariant} is not the zero vector. We now show that every determinant is $0$ and $\nabla_v f_0(x^1,R_{x^1}v)$ is parallel to a particular direction to satisfy \eqref{T invariant}. Using Lemma \ref{d_RA} and $\vert n(x^1) \vert = n_1^2 +n_2^2=1$,\\

\textrm{(Case 1)} $(K_1 \leftrightarrow K_4) $
\begin{align*}
\det \begin{bmatrix} \vert & \vert \\ K_1 & K_4 \\ \vert & \vert  \end{bmatrix}&= \left(\dfrac{-2}{v\cdot n(x^1)}\right)^2 \det 
\begin{bmatrix}
2v_2n_1n_2 -\dfrac{v_2^2n_1}{v\cdot n(x^1)} & v_1 (n_1^2-n_2^2)-\dfrac{v_1^2n_1}{v\cdot n(x^1)}\\
v_2(n_2^2-n_1^2)-\dfrac{v_2^2n_2}{v\cdot n(x^1)} & 2v_1n_1n_2 -\dfrac{v_1^2n_2}{v\cdot n(x^1)}
\end{bmatrix}\\
&=  \left(\dfrac{-2}{v\cdot n(x^1)}\right)^2\left[\left( 4v_1v_2n_1^2n_2^2-\frac{2v_1^2v_2n_1n_2^2}{v\cdot n(x^1)} -\frac{2v_1v_2^2n_1^2n_2}{v\cdot n(x^1)} +\dfrac{v_1^2v_2^2n_1n_2}{(v\cdot n(x^1))^2}\right)  \right. \\ 
&\quad \left. -\left(-v_1v_2(n_2^2-n_1^2)^2-\frac{v_1v_2^2n_2(n_1^2-n_2^2)}{v\cdot n(x^1)} -\frac{v_1^2v_2n_1(n_2^2-n_1^2)}{v\cdot n(x^1)}+\frac{v_1^2v_2^2n_1n_2}{(v\cdot n(x^1))^2}\right)  \right]\\
\hide
&=\left(\dfrac{-2}{v\cdot n(x^1)}\right)^2\left[\left( 4v_1v_2n_1^2n_2^2 +v_1v_2(n_2^2-n_1^2)^2\right)+\left(-\frac{2v_1^2v_2n_1n_2^2}{v\cdot n(x^1)}+\frac{v_1^2v_2n_1(n_2^2-n_1^2)}{v \cdot n(x^1)}\right) \right. \\
&\quad \left.+\left(\frac{v_1v_2^2n_2(n_1^2-n_2^2)}{v\cdot n(x^1)} -\frac{2v_1v_2^2n_1^2n_2}{v\cdot n(x^1)} \right) \right]\\
\unhide
&=\left(\frac{-2}{v\cdot n(x^1)}\right)^2\left( v_1v_2 -\frac{v_1^2v_2 n_1}{ v\cdot n(x^1)} -\frac{v_1v_2^2n_2}{v\cdot n(x^1)}\right)\\
&=0,
\end{align*}

\textrm{(Case 2)} $(K_1 \leftrightarrow K_2)$ 
\begin{align*}
\det \begin{bmatrix} \vert & \vert \\ K_1 & K_2 \\ \vert & \vert  \end{bmatrix}&=\left(\frac{-2}{v\cdot n(x^1)}\right)^2 \det
\begin{bmatrix}
2v_2n_1n_2 -\dfrac{v_2^2n_1}{v\cdot n(x^1)} & -v_1n_1n_2+v_2n_2^2 -\dfrac{(v_2^2-v_1^2)n_1^2n_2}{v\cdot n(x^1)} +\dfrac{2v_1v_2n_1n_2^2}{v\cdot n(x^1)} \\ 
v_2(n_2^2-n_1^2)-\dfrac{v_2^2n_2}{v\cdot n(x^1)} & -v_1n_2^2 -v_2n_1n_2 -\dfrac{(v_2^2-v_1^2)n_1n_2^2}{v\cdot n(x^1)} +\dfrac{2v_1v_2n_2^3}{v \cdot n(x^1)}
\end{bmatrix}\\
&=  \left(\dfrac{-2}{v\cdot n(x^1)}\right)^2 \left[ \left(-2v_1v_2 n_1 n_2^3 -2v_2^2 n_1^2 n_2^2 -\dfrac{2(v_2^2-v_1^2)v_2n_1^2n_2^3}{v\cdot n(x^1)}+\dfrac{4v_1v_2^2n_1n_2^4}{ v \cdot n(x^1)} \right. \right. \\ 
&\quad+ \left. \left.\dfrac{v_1v_2^2n_1n_2^2}{v\cdot n(x^1)} +\dfrac{v_2^3n_1^2n_2}{v\cdot n(x^1)} +\dfrac{(v_2^2-v_1^2)v_2^2n_1^2n_2^2}{(v \cdot n(x^1))^2} -\dfrac{2v_1v_2^3n_1n_2^3}{(v\cdot n(x^1))^2}  \right) \right. \\ 
&\quad \left. -\left(-v_1v_2n_1n_2(n_2^2-n_1^2)+v_2^2n_2^2(n_2^2-n_1^2)-\dfrac{(v_2^2-v_1^2)v_2n_1^2n_2(n_2^2-n_1^2)}{v\cdot n(x^1)}+\dfrac{2v_1v_2^2n_1n_2^2(n_2^2-n_1^2)}{v\cdot n(x^1)} \right.  \right.\\
&\quad \left. \left. +\dfrac{v_1v_2^2n_1n_2^2}{v\cdot n(x^1)} -\dfrac{v_2^3n_2^3}{v\cdot n(x^1)} +\dfrac{(v_2^2-v_1^2)v_2^2n_1^2n_2^2}{(v\cdot n(x^1))^2}-\dfrac{2v_1v_2^3n_1n_2^3}{(v\cdot n(x^1))^2} \right) \right] \\
\hide
&= \left(\dfrac{-2}{v\cdot n(x^1)}\right)^2 \left[ \left(-v_1v_2n_1n_2-v_2^2n_2^2\right)+\left(-\frac{(v_2^2-v_1^2)v_2n_1^2n_2}{v\cdot n(x^1)}+\frac{2v_1v_2^2n_1n_2^2}{v\cdot n(x^1)}+\frac{v_2^3n_2}{v \cdot n(x^1)} \right) \right]\\
\unhide
&=\left(\dfrac{-2}{v\cdot n(x^1)}\right)^2\left[ -\frac{v_2^3n_2^3}{v\cdot n(x^1)} -\frac{v_2^3n_1^2n_2}{v\cdot n(x^1)} +\frac{v_2^3n_2}{v\cdot n(x^1)}\right]\\
&=0,
\end{align*}

\textrm{(Case 3)}  $(K_1 \leftrightarrow K_3)$ 
\begin{align*}
\det \begin{bmatrix} \vert & \vert \\ K_1 & K_3 \\ \vert & \vert  \end{bmatrix}&=\left(\frac{-2}{v\cdot n(x^1)}\right)^2 \det
\begin{bmatrix}
2v_2n_1n_2 -\dfrac{v_2^2n_1}{v\cdot n(x^1)} &-v_2n_1^2 -v_1n_1n_2-\dfrac{(v_1^2-v_2^2)n_1^2n_2}{v\cdot n(x^1)} +\dfrac{2v_1v_2n_1^3}{v\cdot n(x^1)}\\
v_2(n_2^2-n_1^2)-\dfrac{v_2^2n_2}{v\cdot n(x^1)} & v_1n_1^2 -v_2n_1n_2 -\dfrac{(v_1^2-v_2^2)n_1n_2^2}{v\cdot n(x^1)} +\dfrac{2v_1v_2n_1^2n_2}{v\cdot n(x^1)}
\end{bmatrix}\\
&= \left(\dfrac{-2}{v\cdot n(x^1)}\right)^2 \left[ \left(2v_1v_2 n_1^3 n_2 -2v_2^2 n_1^2 n_2^2 -\dfrac{2(v_1^2-v_2^2)v_2n_1^2n_2^3}{v\cdot n(x^1)}+\dfrac{4v_1v_2^2n_1^3n_2^2}{ v \cdot n(x^1)} \right. \right. \\ 
&\quad- \left. \left.\dfrac{v_1v_2^2n_1^3}{v\cdot n(x^1)} +\dfrac{v_2^3n_1^2n_2}{v\cdot n(x^1)} +\dfrac{(v_1^2-v_2^2)v_2^2n_1^2n_2^2}{(v \cdot n(x^1))^2} -\dfrac{2v_1v_2^3n_1^3n_2}{(v\cdot n(x^1))^2}  \right) \right. \\ 
&\quad \left. -\left(-v_2^2n_1^2(n_2^2-n_1^2)-v_1v_2n_1n_2(n_2^2-n_1^2)-\dfrac{(v_1^2-v_2^2)v_2n_1^2n_2(n_2^2-n_1^2)}{v\cdot n(x^1)}+\dfrac{2v_1v_2^2n_1^3(n_2^2-n_1^2)}{v\cdot n(x^1)} \right.  \right.\\
&\quad \left. \left. +\dfrac{v_2^3n_1^2n_2}{v\cdot n(x^1)} +\dfrac{v_1v_2^2n_1n_2^2}{v\cdot n(x^1)} +\dfrac{(v_1^2-v_2^2)v_2^2n_1^2n_2^2}{(v\cdot n(x^1))^2}-\dfrac{2v_1v_2^3n_1^3n_2}{(v\cdot n(x^1))^2} \right) \right]\\
\hide
&= \left(\dfrac{-2}{v\cdot n(x^1)}\right)^2 \left[ \left(v_1v_2n_1n_2-v_2^2n_1^2\right)+\left(-\frac{(v_1^2-v_2^2)v_2n_1^2n_2}{v\cdot n(x^1)} +\frac{2v_1v_2^2n_1^3}{v \cdot n(x^1)} -\frac{v_1v_2^2n_1}{v\cdot n(x^1)}\right)\right]\\
\unhide
&=  \left(\dfrac{-2}{v\cdot n(x^1)}\right)^2\left[ \frac{v_1v_2^2n_1^3}{v\cdot n(x^1)} +\frac{v_1v_2^2n_1n_2^2}{v \cdot n(x^1)} -\frac{v_1v_2^2n_1}{v\cdot n(x^1)}\right]\\
&=0.
\end{align*}
Moreover, from (Case 1) and (Case 2), we deduce 
\begin{align*}
\det \begin{bmatrix} \vert & \vert \\ K_2 & K_4 \\ \vert & \vert  \end{bmatrix}=0.
\end{align*}
Likewise, it holds that 
\begin{align*}
\det \begin{bmatrix} \vert & \vert \\ K_2 & K_3 \\ \vert & \vert  \end{bmatrix}=0, \quad \det \begin{bmatrix} \vert & \vert \\ K_3 & K_4 \\ \vert & \vert  \end{bmatrix}=0.
\end{align*}
Therefore, it means that we can find a nonzero vector $\nabla_v f_0(x^1,R_{x^1}v)$ satisfying \eqref{T invariant}. Since 
\begin{align*}
\nabla_v f_0(x^1,R_{x^1}v) \begin{bmatrix} \vert \\ K_1 \\ \vert \end{bmatrix} = 0,
\end{align*}
$\nabla _v f_0 (x^1,R_{x^1}v)$ is orthogonal to the column vector $K_1$. More specifically, $\nabla_v f_0(x^1,R_{x^1}v)^T$ has the following direction 
\begin{align*}
\frac{-2}{v\cdot n(x^1)} \begin{bmatrix} 
-v_2(n_2^2-n_1^2) + \dfrac{v_2^2 n_2}{v\cdot n(x^1)} \\
2v_2n_1n_2-\dfrac{v_2^2n_1}{v\cdot n(x^1)}
\end{bmatrix}&=\frac{-2}{(v\cdot n(x^1))^2}
\begin{bmatrix}
-v_1v_2n_1(n_2^2-n_1^2) +2v_2^2n_1^2n_2\\
2v_1v_2n_1^2n_2 +v_2^2n_1(n_2^2-n_1^2) 
\end{bmatrix}\\
&=\frac{2v_2n_1}{(v\cdot n(x^1))^2} \begin{bmatrix}
n_2^2-n_1^2 & -2n_1n_2\\
-2n_1n_2 & n_1^2 -n_2^2 
\end{bmatrix} \begin{bmatrix}
v_1 \\ v_2
\end{bmatrix}= \frac{2v_2n_1}{(v\cdot n(x^1))^2} R_{x^1}v.
\end{align*}
Consequently, for \eqref{T invariant}, we get the following condition 
\begin{align} \label{Cond3}
\nabla _v f_0(x,R_xv) \parallel (R_xv)^T,
\end{align}
for any $x \in \partial \O$. \\

\subsubsection{$\nabla_{xx}^T =\nabla_{xx}$} From \eqref{Cond2 2}, we need  
\begin{equation*}
	\begin{split}
&\left(\begin{bmatrix}
\nabla_{x}f_{0}(x^{1}, R_{x^1}v) \nabla_{x}(R_{x^{1}(x,v)}^1)
\\
\nabla_{x}f_{0}(x^{1}, R_{x^1}v) \nabla_{x}(R_{x^{1}(x,v)}^2)
\end{bmatrix} +
\begin{bmatrix}
\nabla_{v}f_{0}(x^{1}, R_{x^1}v) \nabla_{x}(-2A_{v,x^{1}(x,v)}^1)
\\
\nabla_{v}f_{0}(x^{1}, R_{x^1}v) \nabla_{x}(-2A_{v,x^{1}(x,v)}^2)
\end{bmatrix}\right)^T 
\\
&= \begin{bmatrix}
\nabla_{x}f_{0}(x^{1}, R_{x^1}v) \nabla_{x}(R_{x^{1}(x,v)}^1)
\\
\nabla_{x}f_{0}(x^{1}, R_{x^1}v) \nabla_{x}(R_{x^{1}(x,v)}^2)
\end{bmatrix} +
\begin{bmatrix}
\nabla_{v}f_{0}(x^{1}, R_{x^1}v) \nabla_{x}(-2A_{v,x^{1}(x,v)}^1)
\\
\nabla_{v}f_{0}(x^{1}, R_{x^1}v) \nabla_{x}(-2A_{v,x^{1}(x,v)}^2)
\end{bmatrix}.
	\end{split}
\end{equation*}
Thus, it suffices to check that
\begin{align*}
&\nabla_x f_0(x^1,R_{x^1}v) \frac{\partial}{\partial x_2}(R_{x^1(x,v)}^1) +\nabla_v f_0(x^1,R_{x^1}v) \frac{\partial}{\partial x_2} (-2A_{v,x^1(x,v)}^1)\\
&= \nabla_x f_0(x^1,R_{x^1}v) \frac{\partial}{\partial x_1}(R_{x^1(x,v)}^2) +\nabla_v f_0(x^1,R_{x^1}v) \frac{\partial}{\partial x_1} (-2A_{v,x^1(x,v)}^2).
\end{align*}
In other words, we have to find the condition of $\nabla_x f_0 (x^1,R_{x^1}v)$ to satisfy 
\begin{align}\label{xx_sym2}
\nabla_xf_0(x^1,R_{x^1}v) \left[\frac{\partial}{\partial x_2}(R_{x^1(x,v)}^1)-\frac{\partial}{\partial x_1}(R_{x^1(x,v)}^2) \right] = \nabla_v f_0(x^1,R_{x^1}v) \left[ \frac{\partial}{\partial x_1} (-2A_{v,x^1(x,v)}^2)-\frac{\partial}{\partial x_2} (-2A_{v,x^1(x,v)}^1)\right].
\end{align}
Since we computed $\nabla_x (R_{x^1(x,v)}^1), \nabla_x (R_{x^1(x,v)}^2)$ in Lemma \ref{d_RA}, we represent $\nabla_x (-2A_{v,x^1(x,v)}^1)$ and $\nabla_x (-2A_{v,x^1(x,v)}^2)$ by components. 
\begin{lemma} \label{dx_A} Recall the matrix $A_{v,x}$ defined in \eqref{def A}, and then 
	\begin{equation*}
	A_{v,x^1} = \left[ \left((v\cdot n(x^1))I +(n(x^1)\otimes v)\right) \left(I-\frac{v\otimes n(x^1)}{v\cdot n(x^1)}\right)\right].
	\end{equation*}
	If we write that $A^i$ is the $i$th column of matrix $A$, then
	\begin{align*}
	&\nabla_x(-2A_{v,x^1(x,v)}^1) \\
	&= \begin{bmatrix}
	\dfrac{4v_1^2v_2^2n_1^3 +2v_1v_2^3(3n_1^2n_2-n_2^3)+ 2v_2^4(3n_1n_2^2+n_1^3)}{(v\cdot n(x^1))^3} & \dfrac{-4v_1^3v_2n_1^3-2v_1^2v_2^2(3n_1^2n_2-n_2^3)-2v_1v_2^3(3n_1n_2^2+n_1^3)}{(v\cdot n(x^1))^3}\\
	\dfrac{4v_2^4n_2^3+2v_1v_2^3(3n_1n_2^2-n_1^3)+2v_1^2v_2^2(3n_1^2n_2+n_2^3)}{(v\cdot n(x^1))^3} & \dfrac{-4v_1v_2^3n_2^3-2v_1^2v_2^2(3n_1n_2^2-n_1^3)-2v_1^3v_2(3n_1^2n_2+n_2^3)}{(v\cdot n(x^1))^3}
	\end{bmatrix},\\
	&\nabla_x(-2A_{v,x^1(x,v)}^2)\\
	 &= \begin{bmatrix}
	\dfrac{-4v_1^3v_2n_1^3-2v_1v_2^3(3n_1n_2^2+n_1^3) -2v_1^2v_2^2(3n_1^2n_2-n_2^3)}{(v\cdot n(x^1))^3} & \dfrac{4v_1^4n_1^3 +2v_1^2v_2^2(3n_1n_2^2+n_1^3)+2v_1^3v_2 (3n_1^2n_2-n_2^3)}{(v \cdot n(x^1))^3}\\
	\dfrac{-4v_1v_2^3n_2^3 -2v_1^3v_2(3n_1^2n_2+n_2^3) -2v_1^2v_2^2(3n_1n_2^2-n_1^3)}{(v\cdot n(x^1))^3} & \dfrac{4v_1^2 v_2^2 n_2^3 +2v_1^4(3n_1^2n_2+n_2^3)+2v_1^3v_2(3n_1n_2^2-n_1^3)}{(v \cdot n(x^1))^3}
	\end{bmatrix},
	\end{align*}
	where $v_i$ be the $i$th component of $v$. We denote the $i$th component $n_i(x,v)$ of $n(x^1)$ as $n_i$, that is, $n_i$ depends on $x,v$. Furthermore, it holds that 
	\begin{equation} \label{prop d_A}
	\nabla_x(-2A_{v,x^1(x,v)}^1)v =0, \quad \nabla_x (-2A_{v,x^1(x,v)}^2)v=0.
	\end{equation}
\end{lemma}
\begin{proof}
	We write the matrix $-2A_{v,x^1}$ by components: 
	\begin{align*}
	-2A_{v,x^1}=\begin{bmatrix}
	-2v_2n_2 -\dfrac{2v_1v_2n_1n_2}{v\cdot n(x^1)} +\dfrac{2v_2^2 n_1^2}{ v\cdot n(x^1)} & 2v_1n_2 + \dfrac{2v_1^2n_1n_2}{v\cdot n(x^1)} -\dfrac{2v_1v_2n_1^2}{ v\cdot n(x^1)} \\
	2v_2n_1 -\dfrac{2v_1v_2n_2^2}{v\cdot n(x^1)} +\dfrac{2v_2^2n_1n_2}{v\cdot n(x^1)} & -2v_1n_1 -\dfrac{2v_1v_2n_1n_2}{v\cdot n(x^1)} + \dfrac{2v_1^2n_2^2}{v \cdot n(x^1)}
	\end{bmatrix}.
	\end{align*}
	For $\nabla_x(-2A_{v,x^1(x,v)}^1)$, we firstly take a derivative of $(1,1)$ component of $-2A_{v,x^1}$ with respect to $x_1$ 
	\begin{align*}
	&\frac{\partial}{\partial x_1} \left(-2v_2n_2 -\dfrac{2v_1v_2n_1n_2}{v\cdot n(x^1)} +\dfrac{2v_2^2 n_1^2}{ v\cdot n(x^1)} \right)\\
	&=-2v_2 \frac{\partial n_2}{\partial x_1} + \dfrac{\left (-2v_1v_2 n_2\frac{\partial n_1}{\partial x_1} - 2 v_1v_2n_1 \frac{\partial n_2}{\partial x_1}\right)(v\cdot n(x^1))+2v_1v_2n_1n_2\left(v_1\frac{\partial n_1}{\partial x_1}+v_2\frac{\partial n_2}{\partial x_1}\right)}{(v\cdot n(x^1))^2}\\
	& \quad+ \dfrac{\left(4v_2^2n_1\frac{\partial n_1}{\partial x_1}\right)(v\cdot n(x^1))-2v_2^2n_1^2\left(v_1\frac{\partial n_1}{\partial x_1}+v_2\frac{\partial n_2}{\partial x_1}\right)}{(v \cdot n(x^1))^2}\\
	\hide
	&=\frac{4v_1v_2^2n_1^2-2v_1v_2^2n_2^2+6v_2^3n_1n_2}{(v\cdot n(x^1))^2} +\frac{2v_1^2v_2^2n_1n_2^2-4v_1v_2^3n_1^2n_2+2v_2^4n_1^3}{(v\cdot n(x^1))^3}\\
	\unhide
	&=\dfrac{4v_1^2v_2^2n_1^3 +2v_1v_2^3(3n_1^2n_2-n_2^3)+ 2v_2^4(3n_1n_2^2+n_1^3)}{(v\cdot n(x^1))^3},
	\end{align*}
	where we used \eqref{normal} in Lemma \ref{d_n}. Similarly, using \eqref{normal} in Lemma \ref{d_n}, we get 
	\begin{align*}
	&\frac{\partial}{\partial x_2} \left(-2v_2n_2 -\dfrac{2v_1v_2n_1n_2}{v\cdot n(x^1)} +\dfrac{2v_2^2 n_1^2}{ v\cdot n(x^1)} \right)\\
	&=-2v_2 \frac{\partial n_2}{\partial x_2} + \dfrac{\left (-2v_1v_2 n_2\frac{\partial n_1}{\partial x_2} - 2 v_1v_2n_1 \frac{\partial n_2}{\partial x_2}\right)(v\cdot n(x^1))+2v_1v_2n_1n_2\left(v_1\frac{\partial n_1}{\partial x_2}+v_2\frac{\partial n_2}{\partial x_2}\right)}{(v\cdot n(x^1))^2}\\
	&  \quad+ \frac{\left(4v_2^2n_1\frac{\partial n_1}{\partial x_2}\right)(v\cdot n(x^1))-2v_2^2n_1^2\left(v_1\frac{\partial n_1}{\partial x_2}+v_2\frac{\partial n_2}{\partial x_2}\right)}{(v \cdot n(x^1))^2}\\
	\hide
	&=\frac{-4v_1^2v_2n_1^2-6v_1v_2^2n_1n_2+2v_1^2v_2n_2^2}{(v \cdot n(x^1))^2} + \frac{4v_1^2v_2^2n_1^2n_2-2v_1^3v_2n_1n_2^2-2v_1v_2^3n_1^3}{(v\cdot n(x^1))^3} \\
	\unhide
	&= \dfrac{-4v_1^3v_2n_1^3-2v_1^2v_2^2(3n_1^2n_2-n_2^3)-2v_1v_2^3(3n_1n_2^2+n_1^3)}{(v\cdot n(x^1))^3},\\
	&\frac{\partial}{\partial x_1} \left(2v_2n_1 -\dfrac{2v_1v_2n_2^2}{v\cdot n(x^1)} +\dfrac{2v_2^2n_1n_2}{v\cdot n(x^1)}\right)\\
	&=2v_2\frac{\partial n_1}{\partial x_1} - \frac{\left(4v_1v_2n_2\frac{\partial n_2}{\partial x_1}\right)(v\cdot n(x^1))-2v_1v_2n_2^2\left(v_1\frac{\partial n_1}{\partial x_1}+v_2\frac{\partial n_2}{\partial x_1}\right) }{(v \cdot n(x^1))^2}\\
	&\quad+ \frac{\left( 2v_2^2n_2\frac{\partial n_1}{\partial x_1} +2v_2^2n_1\frac{\partial n_2}{\partial x_1}\right) (v\cdot n(x^1))-2v_2^2n_1n_2\left( v_1 \frac{\partial n_1}{\partial x_1}+v_2 \frac{\partial n_2}{\partial x_1}\right)}{(v\cdot n(x^1))^2}\\
	\hide
	&=\frac{6v_1v_2^2n_1n_2+4v_2^3n_2^2-2v_2^3n_1^2}{(v\cdot n(x^1))^2} +\frac{2v_1^2v_2^2n_2^3-4v_1v_2^3n_1n_2^2+2v_2^4n_1^2n_2}{(v\cdot n(x^1))^3}\\
	\unhide
	&=\dfrac{4v_2^4n_2^3+2v_1v_2^3(3n_1n_2^2-n_1^3)+2v_1^2v_2^2(3n_1^2n_2+n_2^3)}{(v\cdot n(x^1))^3},\\
	&\frac{\partial}{\partial x_2} \left(2v_2n_1 -\dfrac{2v_1v_2n_2^2}{v\cdot n(x^1)} +\dfrac{2v_2^2n_1n_2}{v\cdot n(x^1)}\right)\\
	&=2v_2\frac{\partial n_1}{\partial x_2} - \frac{\left(4v_1v_2n_2\frac{\partial n_2}{\partial x_2}\right)(v\cdot n(x^1))-2v_1v_2n_2^2\left(v_1\frac{\partial n_1}{\partial x_2}+v_2\frac{\partial n_2}{\partial x_2}\right) }{(v \cdot n(x^1))^2}\\
	&\quad+ \frac{\left( 2v_2^2n_2\frac{\partial n_1}{\partial x_2} +2v_2^2n_1\frac{\partial n_2}{\partial x_2}\right) (v\cdot n(x^1))-2v_2^2n_1n_2\left( v_1 \frac{\partial n_1}{\partial x_2}+v_2 \frac{\partial n_2}{\partial x_2}\right)}{(v\cdot n(x^1))^2}\\
	\hide
	&\quad -\frac{2v_1v_2^2n_2^2}{(v\cdot n(x^1))^2} +\frac{2v_1v_2^2n_1^2}{(v\cdot n(x^1))^2} +\frac{2v_1^2v_2^2n_1n_2^2}{(v\cdot n(x^1))^3} -\frac{2v_1v_2^3n_1^2n_2}{(v\cdot n(x^1))^3}\\
	\unhide
	&=\dfrac{-4v_1v_2^3n_2^3-2v_1^2v_2^2(3n_1n_2^2-n_1^3)-2v_1^3v_2(3n_1^2n_2+n_2^3)}{(v\cdot n(x^1))^3}.
	\end{align*}
	Thus, we derived $\nabla_x(-2A_{v,x^1(x,v)}^1)$. Similar to $\nabla_x(-2A_{v,x^1(x,v)}^1)$, we can obtain $\nabla_x(-2A_{v,x^1(x,v)}^2)$, and the details are omitted. By the $\nabla_x(-2A_{v,x^1(x,v)}^1)$ and $\nabla_x(-2A_{v,x^1(x,v)}^2)$ formula above, direct calculation gives \eqref{prop d_A}: 
	\begin{footnotesize}
	\begin{align*} 
	&\nabla_x(-2A_{v,x^1(x,v)}^1) v \\
	&= \begin{bmatrix}
	\dfrac{4v_1^2v_2^2n_1^3 +2v_1v_2^3(3n_1^2n_2-n_2^3)+ 2v_2^4(3n_1n_2^2+n_1^3)}{(v\cdot n(x^1))^3} & \dfrac{-4v_1^3v_2n_1^3-2v_1^2v_2^2(3n_1^2n_2-n_2^3)-2v_1v_2^3(3n_1n_2^2+n_1^3)}{(v\cdot n(x^1))^3}\\
	\dfrac{4v_2^4n_2^3+2v_1v_2^3(3n_1n_2^2-n_1^3)+2v_1^2v_2^2(3n_1^2n_2+n_2^3)}{(v\cdot n(x^1))^3} & \dfrac{-4v_1v_2^3n_2^3-2v_1^2v_2^2(3n_1n_2^2-n_1^3)-2v_1^3v_2(3n_1^2n_2+n_2^3)}{(v\cdot n(x^1))^3}
	\end{bmatrix} \begin{bmatrix}
	v_1 \\ v_2 
	\end{bmatrix}=0,\\
	&\nabla_x(-2A_{v,x^1(x,v)}^2) v \\
	&=\begin{bmatrix} 
	\dfrac{-4v_1^3v_2n_1^3-2v_1v_2^3(3n_1n_2^2+n_1^3) -2v_1^2v_2^2(3n_1^2n_2-n_2^3)}{(v\cdot n(x^1))^3} & \dfrac{4v_1^4n_1^3 +2v_1^2v_2^2(3n_1n_2^2+n_1^3)+2v_1^3v_2 (3n_1^2n_2-n_2^3)}{(v \cdot n(x^1))^3}\\
	\dfrac{-4v_1v_2^3n_2^3 -2v_1^3v_2(3n_1^2n_2+n_2^3) -2v_1^2v_2^2(3n_1n_2^2-n_1^3)}{(v\cdot n(x^1))^3} & \dfrac{4v_1^2 v_2^2 n_2^3 +2v_1^4(3n_1^2n_2+n_2^3)+2v_1^3v_2(3n_1n_2^2-n_1^3)}{(v \cdot n(x^1))^3}
	\end{bmatrix} \begin{bmatrix}
	v_1 \\ v_2 
	\end{bmatrix}=0.
	\end{align*}
	\end{footnotesize}
\end{proof}
Now, back to our consideration \eqref{xx_sym2}. By Lemma \ref{dx_A}, we have 
\begin{align*}
\frac{\partial}{\partial x_2} (-2A_{v,x^1(x,v)}^1)= \frac{\partial}{\partial x_1}(-2A_{v,x^1(x,v)}^2),
\end{align*}
which implies that 
\begin{align*}
\nabla_xf_0(x^1,R_{x^1}v)  \left[\frac{\partial}{\partial x_2}(R_{x^1(x,v)}^1)-\frac{\partial}{\partial x_1}(R_{x^1(x,v)}^2) \right]=\frac{2}{v\cdot n(x^1)}\nabla_xf_0(x^1,R_{x^1}v) \begin{bmatrix}
2v_1n_1n_2 +v_2(n_2^2-n_1^2) \\
v_1(n_2^2-n_1^2)-2v_2n_1n_2
\end{bmatrix}=0.
\end{align*}
It means that $\nabla_x f_0(x^1,R_{x^1}v)$ is orthogonal to $\frac{\partial}{\partial x_2}(R_{x^1(x,v)}^1)-\frac{\partial}{\partial x_1}(R_{x^1(x,v)}^2)$ and $\nabla_xf_0(x^1,R_{x^1}v)^T$ has the following direction 
\begin{align*}
\begin{bmatrix}
-v_1(n_2^2-n_1^2)+2v_2n_1n_2 \\ 2v_1n_1n_2+v_2(n_2^2-n_1^2)
\end{bmatrix}=-\begin{bmatrix}
n_2^2-n_1^2 & -2n_1n_2 \\ -2n_1n_2 & n_1^2-n_2^2
\end{bmatrix}\begin{bmatrix}
v_1 \\ v_2 
\end{bmatrix}=-R_{x^1}v.
\end{align*}
To hold $\nabla_{xx} f_0(x^1,R_{x^1}v)^T = \nabla_{xx} f_0 (x^1,R_{x^1}v)$, the following condition 
\begin{align} \label{Cond4}
	\nabla_xf_0(x,R_xv) \parallel (R_xv)^T,
\end{align}
must be satisfied for $x \in \partial \O$. \\

\subsection{Conditions including $\p_{t}$}
In this subsection, we find conditions for $\p_{tt}, \p_{t}\nabla_{x}, \p_{t}\nabla_{v}, \nabla_{x}\p_{t}, \nabla_{v}\p_{t}$. In the last subsubsection, we show that all these $\p_{t}$ including compatibility conditions are covered by \eqref{Cond2 1}--\eqref{Cond2 4}, \eqref{Cond3}, and \eqref{Cond4}. \\

\subsubsection{$\partial_{tt}$} Using the same perturbation \eqref{Perb_t} in $C^1_t$ compatibility condition, we derive $C^2_t$ compatibility condition. For $\epsilon>0$, 
\begin{align*}
	\partial_t(f(t+\epsilon,x,v)-f(t,x,v))&= \partial_t (f_0(X^\epsilon(0),R_{x^1}v)-f_0(X(0),R_{x^1}v))\\
	&=\left( \nabla_x f_0(X^\epsilon(0),R_{x^1}v)-\nabla_xf_0(X(0),R_{x^1}v)\right) (-R_{x^1}v) \\
	&=(-R_{x^1}v)^T \left (\nabla_x f_0(X^\epsilon(0),R_{x^1}v) -\nabla_xf_0(X(0),R_{x^1}v) \right)^T,
\end{align*}
which implies 
\begin{align*}
	f_{tt}(t,x,v) &= \lim_{\epsilon \rightarrow 0+}\frac{ \partial_t f(t+\epsilon,x,v)-\partial_t f(t,x,v)}{\epsilon}\\ 
	&=(-R_{x^1}v)^T \nabla_{xx} f_0(x^1,R_{x^1}v) \lim_{\epsilon\rightarrow 0+}\frac{ X^\epsilon(0)-X(0)}{\epsilon}\\
	&=(-R_{x^1}v)^T \nabla_{xx} f_0(x^1,R_{x^1}v) (-R_{x^1}v).
\end{align*} 
On the other hand, for $\epsilon<0$, it holds that 
\begin{align*}
	\partial_t(f(t+\epsilon,x,v)-f(t,x,v))= \partial_t (f_0(X^\epsilon(0),v)-f_0(X(0),v))&=\left( \nabla_x f_0(X^\epsilon(0),v)-\nabla_xf_0(X(0),v)\right) (-v)\\
	&=(-v)^T \left( \nabla_x f_0(X^\epsilon(0),v)-\nabla_xf_0(X(0),v)\right)^T.
\end{align*}
Thus, we have 
\begin{align*}
	f_{tt}(t,x,v) &= \lim_{\epsilon \rightarrow 0-}\frac{ \partial_t f(t+\epsilon,x,v)-\partial_t f(t,x,v)}{\epsilon}\\ 
	&=(-v)^T \nabla_{xx} f_0(x^1,v) \lim_{\epsilon\rightarrow 0-}\frac{ X^\epsilon(0)-X(0)}{\epsilon}\\
	&=(-v)^T \nabla_{xx} f_0(x^1,v) (-v).
\end{align*}
To sum up, the condition 
\begin{align} \label{time cond}
	v^T \nabla_{xx}f_0(x^1,v)v = (R_{x^1}v)^T \nabla_{xx}f_0(x^1,R_{x^1}v)(R_{x^1}v),
\end{align}
must be satisfied to $f \in C^2_t$. \\

\subsubsection{$C^2_{t,x}$} We firstly use the perturbation \eqref{Perb_t} for $\epsilon <0$. From \eqref{c_3}, it holds that 
\begin{equation} \label{nabla_tx f case1}
	\begin{split}
	\partial_t [\nabla_xf(t,x,v)]&= \lim_{\epsilon \rightarrow 0-} \frac{ \nabla_x f(t+\epsilon,x,v) - \nabla_xf(t,x,v)}{\epsilon}\\
	&=\lim_{\epsilon \rightarrow 0-} \frac{1}{\epsilon} \left( \nabla_x \left[ f_0(X(0;t+\epsilon,x,v),V(0;t+\epsilon,x,v))\right]-\nabla_xf_0(X(0),v)\right)\\
	&=\lim_{\epsilon \rightarrow 0-} \frac{1}{\epsilon} \left( \nabla_x f_0(X^\epsilon(0),v)-\nabla_x f_0(X(0),v)\right)\\
	&=-v^T \nabla_{xx} f_0(x^1,v), 
	\end{split}
\end{equation}
where we used $\nabla_x X^{\epsilon}(0) = I_2$ and $\nabla_x V^{\epsilon}(0)=0$. On the other hand,  for $\epsilon>0$, 
\begin{align*}
	X^{\epsilon}(0):= X(0;t+\epsilon,x,v)=X(0;t,x-\epsilon v, v), \quad V^{\epsilon}(0):=V(0;t+\epsilon,x,v)=R_{x^1}v. 
\end{align*}
Similar to previous case $\nabla_{xx}$, using \eqref{nabla XV_x-} and \eqref{c_4}, 
\begin{align*}
	  \partial_t [\nabla_xf(t,x,v)]&= \lim_{\epsilon \rightarrow 0+} \frac{ \nabla_x f(t+\epsilon,x,v) - \nabla_xf(t,x,v)}{\epsilon}\\
	&=\lim_{\epsilon \rightarrow 0+} \frac{1}{\epsilon} \left( \nabla_x \left[ f_0(X(0;t+\epsilon,x,v),V(0;t+\epsilon,x,v))\right] \right. \\ 
	&\left.\quad - \left(\nabla_x f_0(X(0),R_{x^1}v)R_{x^1} -2\nabla_v f_0(X(0),R_{x^1}v)A_{v,x^1} \right) \right)\\
	&=\lim_{\epsilon \rightarrow 0+} \frac{1}{\epsilon} \left( \nabla_x f_0(X^{\epsilon}(0),V^{\epsilon}(0))\nabla_x X^{\epsilon}(0) +\nabla_v f_0(X^{\epsilon}(0),V^{\epsilon}(0)) \nabla_x V^{\epsilon}(0)\right. \\
	 &\quad \left. -\left( \nabla_x f_0(X(0),R_{x^1}v)R_{x^1} -2\nabla_v f_0(X(0),R_{x^1}v)A_{v,x^1}\right) \right)\\
	 &=\lim _{\epsilon \rightarrow 0+} \frac{1}{\epsilon} \left( \nabla_x f_0 (X^{\epsilon}(0),R_{x^1}v) \nabla_x X^{\epsilon}(0) -\nabla_x f_0(X(0),R_{x^1}v)R_{x^1}\right) \\
	 &\quad + \lim_{\epsilon \rightarrow 0+} \frac{1}{\epsilon} \left( \nabla_v f_0(X^{\epsilon}(0),R_{x^1}v) \nabla_xV^{\epsilon}(0) +2\nabla_v f_0(X(0),R_{x^1}v)A_{v,x^1}\right) \\
	 &:= I_{tx,1}+I_{tx,2},
\end{align*}
where
\begin{align*}
	I_{tx,1}&:=\lim_{\epsilon \rightarrow 0+}\frac{1}{\epsilon} \left(\nabla_x f_0 (X^{\epsilon}(0),R_{x^1}v) \nabla_x X^{\epsilon}(0) - \nabla_xf_0(X^\epsilon(0),R_{x^1}v) \lim_{s\rightarrow 0-}\nabla_x X(s) \right. \\
	&\quad \left. +\nabla_x f_0(X^\epsilon(0),R_{x^1}v) \lim_{s\rightarrow 0-} \nabla_x X(s) -\nabla_x f_0(X(0),R_{x^1}v)R_{x^1}\right)\\
	&\stackrel{r\leftrightarrow c}{=} \left[\nabla_xf_0(x^1,R_{x^1}v)\lim_{\epsilon \rightarrow 0+} \frac{1}{\epsilon} \left( \nabla_x X^{\epsilon}(0)-\lim_{s \rightarrow 0-} \nabla_x X(s) \right)\right]^T\\
	&\quad + R_{x^1}\nabla_{xx}f_0(x^1,R_{x^1}v)(-R_{x^1}v)\\
	&=\begin{bmatrix}
	\nabla_x f_0(x^1,R_{x^1}v)\nabla_x(R_{x^1(x,v)}^1) \\ \nabla_x f_0(x^1,R_{x^1}v) \nabla_x (R_{x^1(x,v)}^2) 
	\end{bmatrix} (-v)+R_{x^1}\nabla_{xx} f_0(x^1,R_{x^1}v) (-R_{x^1}v), \\
	I_{tx,2}&:= \lim_{\epsilon \rightarrow 0+}\frac{1}{\epsilon} \left(\nabla_v f_0 (X^{\epsilon}(0),R_{x^1}v) \nabla_x V^{\epsilon}(0) - \nabla_vf_0(X^\epsilon(0),R_{x^1}v) \lim_{s\rightarrow 0-}\nabla_x V(s) \right. \\
	&\quad \left. +\nabla_v f_0(X^\epsilon(0),R_{x^1}v) \lim_{s\rightarrow 0-} \nabla_x V(s) -2\nabla_v f_0(X(0),R_{x^1}v)A_{v,x^1}\right)\\
	&\stackrel{r\leftrightarrow c}{=} \left[\nabla_vf_0(x^1,R_{x^1}v)\lim_{\epsilon \rightarrow 0+} \frac{1}{\epsilon} \left( \nabla_x V^{\epsilon}(0)-\lim_{s \rightarrow 0-} \nabla_x V(s) \right)\right]^T\\
	&\quad +(-2A^T_{v,x^1})\nabla_{xv} f_0(x^1,R_{x^1}v)\lim_{\epsilon \rightarrow 0+} \frac{ X^{\epsilon}(0)-X(0)}{\epsilon}\\ 
	&= \begin{bmatrix}
	\nabla_v f_0(x^1,R_{x^1}v) \nabla_x (-2A_{v,x^1(x,v)}^1) \\ \nabla_v f_0(x^1,R_{x^1}v) \nabla_x(-2A_{v,x^1(x,v)}^2) 
	\end{bmatrix} (-v)+ (-2A^T_{v,x^1}) \nabla_{xv}f_0(x^1,R_{x^1}v) (-R_{x^1}v). 
\end{align*}
Thus, 
\begin{equation}\label{nabla_tx f case2}
	\begin{split}
	\partial_t [\nabla_x f(t,x,v)] &= (-v)^T \begin{bmatrix}
	\nabla_x f_0(x^1,R_{x^1}v) \nabla_x (R_{x^1(x,v)}^1) \\ \nabla_x f_0(x^1,R_{x^1}v) \nabla_x (R_{x^1(x,v)}^2)
	\end{bmatrix}^T+(-v)^T \begin{bmatrix}
	\nabla_v f_0(x^1,R_{x^1}v) \nabla_x (-2A_{v,x^1(x,v)}^1) \\ \nabla_v f_0(x^1,R_{x^1}v) \nabla_x(-2A_{v,x^1(x,v)}^2) 
	\end{bmatrix}^T\\
	& \quad +(-v^T)R_{x^1} \nabla_{xx} f_0(x^1,R_{x^1}v)R_{x^1}+(-v^T) R_{x^1} \nabla_{vx} f_0(x^1,R_{x^1}v)(-2A_{v,x^1}). 
	\end{split}
\end{equation}
From \eqref{nabla_tx f case1} and \eqref{nabla_tx f case2}, we have the following condition 
\begin{equation} \label{tx comp}
	\begin{split}
	 (-v^T) \nabla_{xx}f_0(x^1,v) &= (-v)^T \begin{bmatrix}
	\nabla_x f_0(x^1,R_{x^1}v) \nabla_x (R_{x^1(x,v)}^1) \\ \nabla_x f_0(x^1,R_{x^1}v) \nabla_x (R_{x^1(x,v)}^2)
	\end{bmatrix}^T+(-v)^T \begin{bmatrix}
	\nabla_v f_0(x^1,R_{x^1}v) \nabla_x (-2A_{v,x^1(x,v)}^1) \\ \nabla_v f_0(x^1,R_{x^1}v) \nabla_x(-2A_{v,x^1(x,v)}^2) 
	\end{bmatrix}^T\\
	& \quad +(-v^T)R_{x^1} \nabla_{xx} f_0(x^1,R_{x^1}v)R_{x^1}+(-v^T) R_{x^1} \nabla_{vx}f_0(x^1,R_{x^1}v) (-2A_{v,x^1}). 
	\end{split}
\end{equation}

\subsubsection{$C^2_{t,v}$} Similar to $C^2_{t,x}$, we use \eqref{c_1} and the perturbation \eqref{Perb_t} for $\epsilon<0$ to obtain
\begin{equation}\label{nabla_tv f case1}
	\begin{split}
		\partial_{t}[\nabla_{v}f(t,x,v)] &= \lim_{\epsilon \rightarrow 0-} \frac{ \nabla_v f(t+\epsilon,x,v) -\nabla_v f(t,x,v)}{ \epsilon}\\
		&= \lim_{\epsilon \rightarrow 0-} \frac{1}{\epsilon} \left( \nabla_v \left[ f_0(X(0;t+\epsilon,x,v),V(0;t+\epsilon,x,v))\right]-(-t\nabla_x f_0(X(0),v)+\nabla_vf_0(X(0),v)) \right)\\
		&= \lim_{\epsilon \rightarrow 0-} \frac{1}{\epsilon} \left( -(t+\epsilon) \nabla_x f_0(X^{\epsilon}(0),v) +\nabla_v f_0(X^{\epsilon}(0),v) +t\nabla_x f_0(X(0),v) -\nabla_v f_0(X(0),v) \right)\\
		&=-\nabla_x f_0(x^1,v) -t(-v^T) \nabla_{xx}f_0(x^1,v) + (-v^T) \nabla_{vx}f_0(x^1,v),
		\end{split}
\end{equation}
where we have used $\nabla_v X^{\epsilon}(0) = -(t+\epsilon) I_2, \nabla_v V^{\epsilon}(0) = I_2$.
For $\epsilon>0$, the perturbation \eqref{Perb_t} becomes 
\begin{equation*}
	X^{\epsilon}(0):=X(0;t+\epsilon,x,v) =X(0;t,x-\epsilon v,v) =x^1 -(t^1+\epsilon)R_{x^1}v, \quad V^{\epsilon}(0):=V(0;t+\epsilon,x,v)=R_{x^1}v.
\end{equation*}
By product rule, Lemma \ref{nabla xv b} and Lemma \ref{d_n}, one obtains that 
\begin{align*}
	\nabla_v [X^{\epsilon}(0)]&=\nabla_v [x^1-(t^1+\epsilon)R_{x^1}v] =-t\left(I-\frac{v \otimes n(x^1)}{v\cdot n(x^1)} \right) -R_{x^1}v \otimes \nabla_v t^1 -\epsilon \nabla_v (R_{x^1}v)\\
	&=-t \left(I-\frac{v \otimes n(x^1)}{v\cdot n(x^1)} \right)-t R_{x^1}v \otimes \frac{n(x^1}{v\cdot n(x^1)} -\epsilon (R_x^1 + 2t A_{v,x^1})\\
	&= -tR_{x^1} -\epsilon (R_{x^1}+2tA_{v,x^1}), \\
	\nabla_v[V^{\epsilon}(0)]&= \nabla_v [R_{x^1}v] = R_{x^1}+2tA_{v,x^1}. 
\end{align*}
Through the $v$-derivative of $X^{\epsilon}(0),V^{\epsilon}(0)$ above and \eqref{c_2}, 
\begin{equation} \label{nabla_tv f case2}
	\begin{split}
	 \partial_t[\nabla_{v} f(t,x,v)]&= \lim_{\epsilon \rightarrow 0+} \frac{\nabla_v f(t+\epsilon,x,v) -\nabla_v f(t,x,v)}{\epsilon}\\
	&=\lim_{\epsilon \rightarrow 0+} \frac{1}{\epsilon} (\nabla_v [f_0(X(0;t+\epsilon,x,v),V(0;t+\epsilon,x,v)]\\
	&\quad  -(-t\nabla_x f_0(X(0),R_{x^1}v)R_{x^1}+ \nabla_v f_0(X(0),R_{x^1}v)(R_{x^1}+2tA_{v,x^1})))\\
	&=-\nabla_x f_0(x^1,R_{x^1}v)\left(R_{x^1}+2tA_{v,x^1}\right) -t \left[ \lim_{\epsilon\rightarrow 0+} \frac{1}{\epsilon} \left(\nabla_xf_0(X^{\epsilon}(0),R_{x^1}v) -\nabla_xf_0(X^{\epsilon}(0),R_{x^1}v)\right)\right] R_{x^1} \\ 
	&\quad + \left [\lim_{\epsilon \rightarrow 0+} \frac{1}{\epsilon}\left( \nabla_v f_0(X^{\epsilon}(0),R_{x^1}v) -\nabla_v f_0(X(0),R_{x^1}v) \right)\right]\left(R_{x^1}+2tA_{v,x^1}\right)\\
	&\stackrel{r\leftrightarrow c}{=} -\left(R_{x^1}+2tA_{v,x^1}\right)^T\nabla_x f_0(x^1,R_{x^1}v)^T -tR_{x^1} \nabla_{xx}f_0(x^1,R_{x^1}v) \lim_{\epsilon \rightarrow 0+} \frac{X^{\epsilon}(0)-X(0)}{\epsilon}  \\
	&\quad + \left(R_{x^1}+2tA_{v,x^1}\right)^T \nabla_{xv} f_0(x^1,R_{x^1}v)  \lim_{\epsilon \rightarrow 0+} \frac{X^{\epsilon}(0)-X(0)}{\epsilon}\\
	&\stackrel{c\leftrightarrow r}{=} -\nabla_xf_0(x^1,R_{x^1}v) (R_{x^1}+2tA_{v,x^1}) -t(-v^T)R_{x^1} \nabla_{xx} f_0(x^1,R_{x^1}v) R_{x^1} \\
	&\quad +(-v^T) R_{x^1} \nabla_{vx} f_0(x^1,R_{x^1}v) (R_{x^1}+2tA_{v,x^1}).
	\end{split}
\end{equation}
Summing \eqref{nabla_tv f case1} and \eqref{nabla_tv f case2} yields that 
	\begin{equation} \label{tv comp}
		\begin{split}
			&-\nabla_x f_0(x^1,v) -t(-v^T) \nabla_{xx}f_0(x^1,v) + (-v^T) \nabla_{vx}f_0(x^1,v)\\
			&= -\nabla_xf_0(x^1,R_{x^1}v) (R_{x^1}+2tA_{v,x^1})-t(-v^T)R_{x^1} \nabla_{xx} f_0(x^1,R_{x^1}v) R_{x^1}  \\
			&\quad +(-v^T) R_{x^1} \nabla_{vx} f_0(x^1,R_{x^1}v) (R_{x^1}+2tA_{v,x^1}).
		\end{split}
	\end{equation}

\subsubsection{$C^2_{x,t}$} Similar to the $\nabla_{xv}$ case, using the same perturbation $\hat{r}_1$ of \eqref{set R_sp} and \eqref{c_3}, we have 
\begin{equation*}
	\begin{split}
	 \nabla_x[\partial_t f(t,x,v)]\hat{r}_1&=\lim_{\epsilon \rightarrow 0+} \frac{ \partial_t f(t,x+\epsilon \hat{r}_1,v)- \partial_t f(t,x,v)}{\epsilon} \\
	&= \lim_{\epsilon\rightarrow 0+} \left(\frac{ \nabla_x f(t,x+\epsilon \hat{r}_1,v) - \nabla_x f(t,x,v)}{\epsilon}\right)(-v)\\
	&=  \lim_{\epsilon \rightarrow 0+} \frac{1}{\epsilon} \left(\nabla_x [f_0(X(0;t,x+\epsilon \hat{r}_1,v),V(0;t,x+\epsilon \hat{r}_1,v))] -\nabla_x f_0(X(0),v)\right)(-v)\\
	&= (-v^T) \nabla_{xx} f_0(x^1,v) \hat{r}_1,
	\end{split} 
\end{equation*}
where we have used $\nabla_x X^{\epsilon}(0)=I_2, \nabla_x V^{\epsilon}(0)=0$. Next, for $\hat{r}_2$ of \eqref{set R_sp}, using \eqref{Av=0} in Lemma \ref{lem_RA}, \eqref{c_4},\eqref{xx star1}, and \eqref{xx star2} gives 
\begin{equation*}
	\begin{split}
	\nabla_x[\partial_t f(t,x,v)] \hat{r}_2 &= \lim_{\epsilon \rightarrow 0+} \frac{ \partial_t f(t,x+\epsilon \hat{r}_2,v)- \partial_t f(t,x,v)}{\epsilon}\\
	&= \lim_{\epsilon\rightarrow 0+} \left(\frac{ \nabla_x f(t,x+\epsilon \hat{r}_2,v) - \nabla_x f(t,x,v)}{\epsilon}\right)(-v)\\
	&= \lim_{\epsilon \rightarrow 0+} \frac{1}{\epsilon} (\nabla_x [f_0(X(0;t,x+\epsilon \hat{r}_2,v),V(0;t,x+\epsilon \hat{r}_2,v))] \\
	&\quad - (\nabla_x f_0(X(0),R_{x^1}v)R_{x^1} -2 \nabla_v f_0(X(0),R_{x^1}v) A_{v,x^1}) )(-v)\\
	&=\lim_{\epsilon \rightarrow 0+} \frac{1}{\epsilon} \left( \nabla_x f_0(X^{\epsilon}(0),V^{\epsilon}(0))\nabla_ xX^{\epsilon}(0) - \nabla_x f_0(X(0),R_{x^1}v)R_{x^1}\right) (-v) \\
	&\quad + \lim_{\epsilon \rightarrow 0+} \frac{1}{\epsilon} \left( \nabla_v f_0(X^{\epsilon}(0),V^{\epsilon}(0))\nabla_x V^{\epsilon}(0) -\nabla_v f_0(X(0),R_{x^1}v)(-2A_{v,x^1}) \right) (-v) \\
	&:=I_{xt,1}+I_{xt,2}, 
	\end{split}
\end{equation*}
where
\begin{align*}
	I_{xt,1}&=\lim_{\epsilon \rightarrow 0+} \frac{1}{\epsilon} \left( \nabla_x f_0(X^{\epsilon}(0),V^{\epsilon}(0))\nabla_ xX^{\epsilon}(0) - \nabla_x f_0(X^{\epsilon}(0),V^{\epsilon}(0)) \lim_{s \rightarrow 0-} \nabla_x X(s) \right. \\
	&\quad + \left. \nabla_x f_0(X^{\epsilon}(0),V^{\epsilon}(0))\lim_{s\rightarrow 0-} \nabla_x X(s) - \nabla_x f_0(X(0),R_{x^1}v)R_{x^1}\right)(-v)\\
	&= (-v^T) \begin{bmatrix}
		\nabla_x f_0(x^1,R_{x^1}v) \nabla_x (R_{x^1(x,v)}^1) \\ 
		\nabla_x f_0(x^1,R_{x^1}v) \nabla_x (R_{x^1(x,v)}^2)
	\end{bmatrix} \hat{r}_2\\
	&\quad +(-v^T)R_{x^1} \nabla_{xx} f_0(x^1,R_{x^1}v)R_{x^1}\hat{r}_2 +(-v^T)R_{x^1} \nabla_{vx} f_0(x^1,R_{x^1}v)(-2A_{v,x^1})\hat{r}_2,\\
	I_{xt,2} &=  \lim_{\epsilon \rightarrow 0+} \frac{1}{\epsilon} \left( \nabla_v f_0(X^{\epsilon}(0),V^{\epsilon}(0))\nabla_x V^{\epsilon}(0) -\nabla_v f_0(X^\epsilon(0),V^{\epsilon}(0))\lim_{s\rightarrow 0-} \nabla_x V(s) \right. \\
	&\quad + \left. \nabla_v f_0(X^\epsilon(0),V^{\epsilon}(0)) \lim_{s \rightarrow 0-} \nabla_x V(s) -\nabla_v f_0(X(0),R_{x^1}v)(-2A_{v,x^1})\right)(-v) \\
	&=(-v^T) \begin{bmatrix} 
	\nabla_vf_0(x^1,R_{x^1}v) \nabla_x (-2A_{v,x^1(x,v)}^1) \\ \nabla_v f_0(x^1,R_{x^1}v) \nabla_x (-2A_{v,x^1(x,v)}^2)
	\end{bmatrix}\hat{r}_2 \\
	& \quad +(-v^T)(-2A^T_{v,x^1}) \left( \nabla_{xv} f_0(x^1,R_{x^1}v) R_{x^1} +\nabla_{vv} f_0(x^1,R_{x^1}v) (-2A_{v,x^1}) \right) \hat{r}_2,\\
	&=(-v^T) \begin{bmatrix} 
	\nabla_vf_0(x^1,R_{x^1}v) \nabla_x (-2A_{v,x^1(x,v)}^1) \\ \nabla_v f_0(x^1,R_{x^1}v) \nabla_x (-2A_{v,x^1(x,v)}^2)
	\end{bmatrix}\hat{r}_2.
\end{align*}
To sum up the above, we get the following condition: 
\begin{equation} \label{xt comp}
	\begin{split}
		(-v^T) \nabla_{xx} f_0(x^1,v)&= (-v^T) \begin{bmatrix}
		\nabla_x f_0(x^1,R_{x^1}v) \nabla_x (R_{x^1(x,v)}^1)\\ 
		\nabla_x f_0(x^1,R_{x^1}v) \nabla_x (R_{x^1(x,v)}^2)\end{bmatrix}
		+(-v^T) 
		\begin{bmatrix} 
	\nabla_vf_0(x^1,R_{x^1}v) \nabla_x (-2A_{v,x^1(x,v)}^1) \\ \nabla_v f_0(x^1,R_{x^1}v) \nabla_x (-2A_{v,x^1(x,v)}^2)
	\end{bmatrix}\\
	&\quad  +(-v^T)R_{x^1} \nabla_{xx} f_0(x^1,R_{x^1}v)R_{x^1}+(-v^T)R_{x^1} \nabla_{vx} f_0(x^1,R_{x^1}v)(-2A_{v,x^1}).
	\end{split}
\end{equation}

\subsubsection{$C^2_{v,t}$} Using the perturbation $\hat{r}_1$ of \eqref{set R_sp} and \eqref{c_3}, 

\begin{equation*}
	\begin{split}
		\nabla_v [\partial_t f(t,x,v)] \hat{r}_1 &=\lim_{\epsilon\rightarrow 0+} \frac{ \partial_t f(t,x,v+\epsilon \hat{r}_1) -\partial_t f(t,x,v)}{\epsilon}\\
		&=\lim_{\epsilon \rightarrow 0+} \left(\frac{\nabla_x f(t,x,v+\epsilon \hat{r}_1) (-(v+\epsilon \hat{r}_1)) -\nabla_x f(t,x,v) (-v)}{\epsilon} \right) \\
		&=-\lim_{\epsilon \rightarrow 0+} \nabla_x [f_0(X(0;t,x,v+\epsilon \hat{r}_1), V(0;t,x,v+\epsilon \hat{r}_1))]\hat{r}_1 \\
		&+\lim_{\epsilon \rightarrow 0+} \frac{1}{\epsilon} ( \nabla_x [f_0(X(0;t,x,v+\epsilon \hat{r}_1), V(0;t,x,v+\epsilon \hat{r}_1))]-\nabla_x f_0(X(0),v))(-v)\\
		&=-\nabla_x f_0(X(0),v) \hat{r}_1 +\lim_{\epsilon \rightarrow 0+} \frac{1}{\epsilon} (\nabla_x f_0(X^{\epsilon}(0),V^{\epsilon}(0))-\nabla_x f_0(X(0),v))(-v)\\ 
		&=-\nabla_x f_0(x^1,v)\hat{r}_1 +(-v^T) \nabla_{xx} f_0(x^1,v) (-t\hat{r}_1) +(-v^T) \nabla_{vx} f_0(x^1,v) \hat{r}_1,
	\end{split}
\end{equation*}
where $X^{\epsilon}(0):=X(0;t,x,v+\epsilon \hat{r}_1) = x-t(v+\epsilon \hat{r}_1), V^{\epsilon}(0):=V(0;t,x,v+\epsilon \hat{r}_1) =v+\epsilon \hat{r}_1$.  Similar to the case $\nabla_{vx}$, for the perturbation $\hat{r}_2$ of \eqref{set R_sp}, using \eqref{nabla XV_x-}, \eqref{c_4} and \eqref{Av=0} in Lemma \ref{lem_RA} yields:
\begin{equation*}
	\begin{split}
		 \nabla_v[\partial_t f(t,x,v)] \hat{r}_2 &= \lim_{\epsilon \rightarrow 0+} \frac{\partial_t f(t,x,v+\epsilon \hat{r}_2) -\partial_t f(t,x,v)}{\epsilon}\\
		&=\lim_{\epsilon \rightarrow 0+} \left( \frac{ \nabla_x f(t,x,v+\epsilon \hat{r}_2)(-(v+\epsilon \hat{r}_2))-\nabla_x f(t,x,v)(-v)}{\epsilon}\right)\\
		&=-\lim_{\epsilon \rightarrow 0+} \nabla_x [f_0(X(0;t,x,v+\epsilon \hat{r}_2),V(0;t,x,v+\epsilon \hat{r}_2))]\hat{r}_2 \\
		&\quad +\lim_{\epsilon \rightarrow 0+} \frac{1}{\epsilon} \left( \nabla_x [f_0(X(0;t,x,v+\epsilon \hat{r}_2), V(0;t,x,v+\epsilon \hat{r}_2))] \right. \\
		&\quad - \left. \left(\nabla_x f_0(x^1,R_{x^1}v)R_{x^1}-2\nabla_v f_0(x^1,R_{x^1}v)A_{v,x^1} \right)\right)(-v)\\
		&=-\left(\nabla_x f_0(x^1,R_{x^1}v) R_{x^1} +\nabla_v f_0(x^1,R_{x^1}v)(-2A_{v,x^1})\right) \hat{r}_2 \\
		&\quad + \lim_{\epsilon \rightarrow 0+} \frac{1}{\epsilon} \left( \nabla_x f_0(X^{\epsilon}(0),V^{\epsilon}(0))\nabla_x X^{\epsilon}(0) -\nabla_x f_0(X(0),R_{x^1}v)R_{x^1}\right)(-v) \\
		&\quad + \lim_{\epsilon \rightarrow 0+} \frac{1}{\epsilon} \left(  \nabla_v f_0(X^{\epsilon}(0),V^{\epsilon}(0))\nabla_x V^{\epsilon}(0)-\nabla_v f_0(X(0),R_{x^1}v)(-2A_{v,x^1})\right)(-v)\\ 
		&:=-\left(\nabla_x f_0(x^1,R_{x^1}v) R_{x^1} +\nabla_v f_0(x^1,R_{x^1}v)(-2A_{v,x^1})\right) \hat{r}_2 + I_{vt,1}+I_{vt,2},
	\end{split} 
\end{equation*}
where
\begin{align*}
	I_{vt,1}&:=  \lim_{\epsilon \rightarrow 0+} \frac{1}{\epsilon} \left( \nabla_x f_0(X^{\epsilon}(0),V^{\epsilon}(0))\nabla_x X^{\epsilon}(0)-\nabla_x f_0(X^{\epsilon}(0),V^{\epsilon}(0))\lim_{s\rightarrow 0-} \nabla_x X(s) \right. \\
	&\quad + \left. \nabla_x f_0(X^\epsilon(0),V^\epsilon(0))\lim_{s \rightarrow 0-} \nabla_x X(s) - \nabla_xf_0(X(0),R_{x^1}v)R_{x^1}\right) (-v) \\
	&=(-v^T) \begin{bmatrix} 
	\nabla_x f_0(x^1,R_{x^1}v) \nabla_v (R_{x^1(x,v)}^1) \\ \nabla_x f_0(x^1,R_{x^1}v) \nabla_v (R_{x^1(x,v)}^2) 
	\end{bmatrix}\hat{r}_2 \\
	&\quad + (-v^T)R_{x^1} \left( \nabla_{xx} f_0(x^1,R_{x^1}v) (-tR_{x^1}) +\nabla_{vx} f_0(x^1,R_{x^1}v) (R_{x^1}+2tA_{v,x^1})\right)\hat{r}_2,\\
	I_{vt,2}&:=   \lim_{\epsilon \rightarrow 0+} \frac{1}{\epsilon} \left(  \nabla_v f_0(X^{\epsilon}(0),V^{\epsilon}(0))\nabla_x V^{\epsilon}(0)-\nabla_v f_0(X^\epsilon(0),V^\epsilon(0))\lim_{s\rightarrow 0-} \nabla_xV(s) \right. \\
	& \quad \left.+ \nabla_v f_0(X^\epsilon(0),V^\epsilon(0))\lim_{s\rightarrow 0-} \nabla_xV(s) -\nabla_v f_0(X(0),R_{x^1}v)(-2A_{v,x^1}) \right)(-v)\\
	&=(-v^T)\begin{bmatrix} 
	\nabla_v f_0(x^1,R_{x^1}v) \nabla_v (-2A_{v,x^1(x,v)}^1) \\ \nabla_v f_0(x^1,R_{x^1}v) \nabla_v (-2A_{v,x^1(x,v)}^2) 
	\end{bmatrix}\hat{r}_2\\
	&\quad +(-v^T)(-2A^T_{v,x^1}) \left( \nabla_{xv} f_0(x^1,R_{x^1}v)(-tR_{x^1}) +\nabla_{vv} f_0(x^1,R_{x^1}v)(R_{x^1}+2tA_{v,x^1}) \right) \hat{r}_2\\
	&=(-v^T)\begin{bmatrix} 
	\nabla_v f_0(x^1,R_{x^1}v) \nabla_v (-2A_{v,x^1(x,v)}^1) \\ \nabla_v f_0(x^1,R_{x^1}v) \nabla_v (-2A_{v,x^1(x,v)}^2) 
	\end{bmatrix}\hat{r}_2.
\end{align*}
Thus, we have the following compatibility condition: 
	\begin{equation} \label{vt comp}
		\begin{split}
		&-\nabla_x f_0(x^1,v) +tv^T\nabla_{xx} f_0(x^1,v) +(-v^T) \nabla_{vx} f_0(x^1,v) \\
		&=- \left(\nabla_x f_0(x^1,R_{x^1}v) R_{x^1} +\nabla_v f_0(x^1,R_{x^1}v)(-2A_{v,x^1})\right)\\
		&\quad +(-v^T) \begin{bmatrix} 
	\nabla_x f_0(x^1,R_{x^1}v) \nabla_v (R_{x^1(x,v)}^1) \\ \nabla_x f_0(x^1,R_{x^1}v) \nabla_v (R_{x^1(x,v)}^2) 
	\end{bmatrix} + (-v^T) \begin{bmatrix} 
	\nabla_v f_0(x^1,R_{x^1}v) \nabla_v (-2A_{v,x^1(x,v)}^1) \\ \nabla_v f_0(x^1,R_{x^1}v) \nabla_v (-2A_{v,x^1(x,v)}^2) 
	\end{bmatrix}\\
	&\quad +tv^T R_{x^1} \nabla_{xx} f_0(x^1,R_{x^1}v) R_{x^1} +(-v^T)R_{x^1} \nabla_{vx} f_0(x^1,R_{x^1}v)(R_{x^1}+2tA_{v,x^1}).
		\end{split}
	\end{equation}

\subsubsection{Derive $C^2_{tt},C^2_{tx}, C^2_{tv},C^2_{xt},C^2_{vt}$ compatibility conditions from \eqref{Cond2 1}--\eqref{Cond2 4},\eqref{Cond3} and \eqref{Cond4}} So far, we have derived \eqref{Cond2 1}--\eqref{Cond2 4} to satisfy $f\in C^2_{xv},C^2_{xx},C^2_{vx},C^2_{vv}$. In \eqref{Cond2 1}--\eqref{Cond2 4}, since $\nabla_{xv} f_0(x^1,v)$ is the same as $\nabla_{vx} f_0(x^1,v)^T$, we need to assume \eqref{Cond3}. Similarly, we obtained \eqref{Cond3} because $\nabla_{xx} f_0(x^1,v)$ is a symmetric matrix. In this subsection, we will show that the compatibility conditions $C^2_{tt}$ \eqref{time cond}, $C^2_{tx}$ \eqref{tx comp}, $C^2_{tv}$ \eqref{tv comp}, $C^2_{xt}$ \eqref{xt comp}, and $C^2_{vt}$ \eqref{vt comp} are induced under \eqref{Cond2 1}--\eqref{Cond2 4},\eqref{Cond3}, and \eqref{Cond4}. Firstly, we consider $C^2_{tt}$ compatibility condition. Using \eqref{Av=0} in Lemma \ref{lem_RA}, \eqref{prop d_R}, and \eqref{prop d_A}, one has 
\begin{equation*}
	\begin{split}
	 v^T \nabla_{xx}f_0(x^1,v) v &= v^T \Bigg(R_{x^{1}} \nabla_{xx}f_{0}(x^{1}, R_{x^1}v) R_{x^{1}} + R_{x^{1}} \nabla_{vx}f_{0}(x^{1}, R_{x^1}v)(-2A_{v,x^{1}})  \\
&\quad + (-2A^{T}_{v,x^{1}}) \nabla_{xv}f_{0}(x^{1}, R_{x^1}v) R_{x^{1}} + (-2A^{T}_{v,x^{1}}) \nabla_{vv}f_{0}(x^{1}, R_{x^1}v)  (-2A_{v,x^{1}})  \\
&\quad +  \begin{bmatrix}
\nabla_{x}f_{0}(x^{1}, R_{x^1}v) \nabla_{x}(R_{x^{1}(x,v)}^1)
\\
\nabla_{x}f_{0}(x^{1}, R_{x^1}v) \nabla_{x}(R_{x^{1}(x,v)}^2)
\end{bmatrix} 
- 2
\begin{bmatrix}
\nabla_{v}f_{0}(x^{1}, R_{x^1}v) \nabla_{x}(A_{v,x^{1}(x,v)}^1)
\\
\nabla_{v}f_{0}(x^{1}, R_{x^1}v) \nabla_{x}(A_{v,x^{1}(x,v)}^2)
\end{bmatrix}
 \Bigg) v\\
 &=v^TR_{x^1}\nabla_{xx}f_0(x^1,R_{x^1}v)R_{x^1} v +v^T  \begin{bmatrix}
\nabla_{x}f_{0}(x^{1}, R_{x^1}v) \nabla_{x}(R_{x^{1}(x,v)}^1)
\\
\nabla_{x}f_{0}(x^{1}, R_{x^1}v) \nabla_{x}(R_{x^{1}(x,v)}^2)
\end{bmatrix}  v \\
&\quad + v^T \begin{bmatrix}
\nabla_{v}f_{0}(x^{1}, R_{x^1}v) \nabla_{x}(-2A_{v,x^{1}(x,v)}^1)
\\
\nabla_{v}f_{0}(x^{1}, R_{x^1}v) \nabla_{x}(-2A_{v,x^{1}(x,v)}^2)
\end{bmatrix}v\\
&= (R_{x^1}v)^T \nabla_{xx} f_0(x^1,R_{x^1}v) (R_{x^1}v).
 	\end{split}
\end{equation*}
In \eqref{tx comp}, the left-hand side is 
\begin{align*}
	(-v^T) \nabla_{xx} f_0(x^1,v) &= (-v^T) \Bigg(R_{x^{1}} \nabla_{xx}f_{0}(x^{1}, R_{x^1}v) R_{x^{1}} + R_{x^{1}} \nabla_{vx}f_{0}(x^{1}, R_{x^1}v)(-2A_{v,x^{1}})  \\
&\quad + (-2A^{T}_{v,x^{1}}) \nabla_{xv}f_{0}(x^{1}, R_{x^1}v) R_{x^{1}} + (-2A^{T}_{v,x^{1}}) \nabla_{vv}f_{0}(x^{1}, R_{x^1}v)  (-2A_{v,x^{1}})  \\
&\quad +  \begin{bmatrix}
\nabla_{x}f_{0}(x^{1}, R_{x^1}v) \nabla_{x}(R_{x^{1}(x,v)}^1)
\\
\nabla_{x}f_{0}(x^{1}, R_{x^1}v) \nabla_{x}(R_{x^{1}(x,v)}^2)
\end{bmatrix} 
- 2
\begin{bmatrix}
\nabla_{v}f_{0}(x^{1}, R_{x^1}v) \nabla_{x}(A_{v,x^{1}(x,v)}^1)
\\
\nabla_{v}f_{0}(x^{1}, R_{x^1}v) \nabla_{x}(A_{v,x^{1}(x,v)}^2)
\end{bmatrix}
 \Bigg)\\
 &= (-v^T) R_{x^1}\nabla_{xx} f_0(x^1,R_{x^1}v) R_{x^1} + (-v^T)R_{x^1} \nabla_{vx} f_0(x^1,R_{x^1}v) (-2A_{v,x^1}) \\
 &\quad + (-v^T) \begin{bmatrix}
\nabla_{x}f_{0}(x^{1}, R_{x^1}v) \nabla_{x}(R_{x^{1}(x,v)}^1)
\\
\nabla_{x}f_{0}(x^{1}, R_{x^1}v) \nabla_{x}(R_{x^{1}(x,v)}^2)
\end{bmatrix}  +(-v^T) \begin{bmatrix}
\nabla_{v}f_{0}(x^{1}, R_{x^1}v) \nabla_{x}(-2A_{v,x^{1}(x,v)}^1)
\\
\nabla_{v}f_{0}(x^{1}, R_{x^1}v) \nabla_{x}(-2A_{v,x^{1}(x,v)}^2)
\end{bmatrix},
\end{align*}
where we have used \eqref{Av=0}. When we assume \eqref{Cond4}, it holds that $\nabla_{xx}f_0(x^1,v)$ is a symmetric matrix. In other words, 
\begin{align*}
&\left(\begin{bmatrix}
	\nabla_{x}f_{0}(x^{1}, R_{x^1}v) \nabla_{x}(R_{x^{1}(x,v)}^1)
\\
\nabla_{x}f_{0}(x^{1}, R_{x^1}v) \nabla_{x}(R_{x^{1}(x,v)}^2)
\end{bmatrix} +
\begin{bmatrix}
\nabla_{v}f_{0}(x^{1}, R_{x^1}v) \nabla_{x}(-2A_{v,x^{1}(x,v)}^1)
\\
\nabla_{v}f_{0}(x^{1}, R_{x^1}v) \nabla_{x}(-2A_{v,x^{1}(x,v)}^2)
\end{bmatrix}\right)^T \\
&= \begin{bmatrix}
	\nabla_{x}f_{0}(x^{1}, R_{x^1}v) \nabla_{x}(R_{x^{1}(x,v)}^1)
\\
\nabla_{x}f_{0}(x^{1}, R_{x^1}v) \nabla_{x}(R_{x^{1}(x,v)}^2)
\end{bmatrix} +
\begin{bmatrix}
\nabla_{v}f_{0}(x^{1}, R_{x^1}v) \nabla_{x}(-2A_{v,x^{1}(x,v)}^1)
\\
\nabla_{v}f_{0}(x^{1}, R_{x^1}v) \nabla_{x}(-2A_{v,x^{1}(x,v)}^2)
\end{bmatrix},
\end{align*}
which implies that 
\begin{equation} \label{vRA prop} 
	\begin{split}
	&(-v^T) \begin{bmatrix}
\nabla_{x}f_{0}(x^{1}, R_{x^1}v) \nabla_{x}(R_{x^{1}(x,v)}^1)
\\
\nabla_{x}f_{0}(x^{1}, R_{x^1}v) \nabla_{x}(R_{x^{1}(x,v)}^2)
\end{bmatrix}  +(-v^T) \begin{bmatrix}
\nabla_{v}f_{0}(x^{1}, R_{x^1}v) \nabla_{x}(-2A_{v,x^{1}(x,v)}^1)
\\
\nabla_{v}f_{0}(x^{1}, R_{x^1}v) \nabla_{x}(-2A_{v,x^{1}(x,v)}^2)
\end{bmatrix}\\
&=\left( \left( \begin{bmatrix}
\nabla_{x}f_{0}(x^{1}, R_{x^1}v) \nabla_{x}(R_{x^{1}(x,v)}^1)
\\
\nabla_{x}f_{0}(x^{1}, R_{x^1}v) \nabla_{x}(R_{x^{1}(x,v)}^2)
\end{bmatrix}  + \begin{bmatrix}
\nabla_{v}f_{0}(x^{1}, R_{x^1}v) \nabla_{x}(-2A_{v,x^{1}(x,v)}^1)
\\
\nabla_{v}f_{0}(x^{1}, R_{x^1}v) \nabla_{x}(-2A_{v,x^{1}(x,v)}^2)
\end{bmatrix}\right)(-v)\right)^T=0,
	\end{split}
\end{equation}
due to \eqref{prop d_R} and \eqref{prop d_A}. Therefore, the left-hand side in \eqref{tx comp} becomes 
\begin{equation} \label{tx comp left}
	 (-v^T) \nabla_{xx} f_0(x^1,v) = (-v^T) R_{x^1} \nabla_{xx} f_0(x^1,R_{x^1}v) R_{x^1} + (-v^T) \nabla_{vx} f_0(x^1,R_{x^1}v) (-2A_{v,x^1}).
\end{equation}
Using \eqref{vRA prop}, the right-hand side in \eqref{tx comp} is 
\begin{equation}\label{tx comp right}
	\begin{split}
	&(-v)^T \begin{bmatrix}
	\nabla_x f_0(x^1,R_{x^1}v) \nabla_x (R_{x^1(x,v)}^1) \\ \nabla_x f_0(x^1,R_{x^1}v) \nabla_x (R_{x^1(x,v)}^2)
	\end{bmatrix}^T+(-v)^T \begin{bmatrix}
	\nabla_v f_0(x^1,R_{x^1}v) \nabla_x (-2A_{v,x^1(x,v)}^1) \\ \nabla_v f_0(x^1,R_{x^1}v) \nabla_x(-2A_{v,x^1(x,v)}^2)
	\end{bmatrix}^T\\
	& \quad +(-v^T)R_{x^1} \nabla_{xx} f_0(x^1,R_{x^1}v)R_{x^1}+(-v^T) R_{x^1} \nabla_{vx}f_0(x^1,R_{x^1}v) (-2A_{v,x^1})\\
	& =(-v^T)R_{x^1} \nabla_{xx} f_0(x^1,R_{x^1}v)R_{x^1}+(-v^T) R_{x^1} \nabla_{vx}f_0(x^1,R_{x^1}v) (-2A_{v,x^1}).\\
	\end{split}
\end{equation}
From \eqref{tx comp left} and \eqref{tx comp right}, we derive \eqref{tx comp} under the assumption \eqref{Cond2 1}--\eqref{Cond2 4},\eqref{Cond3}, and \eqref{Cond4}. For the left-hand side in \eqref{tv comp}, we use \eqref{Av=0}, the $C^1$ compatibility condition \eqref{c_x}, \eqref{Cond2 1}--\eqref{Cond2 4}, and \eqref{vRA prop}: 
\begin{align*}
	&-\nabla_x f_0(x^1,v) +tv^T \nabla_{xx} f_0(x^1,v) + (-v^T) \nabla_{vx} f_0(x^1,v) \\
	&= -\nabla_x f_0(x^1,R_{x^1}v)R_{x^1} - \nabla_v f_0(x^1,R_{x^1}v) (-2A_{v,x^1}) \\
	&\quad +tv^TR_{x^1}\nabla_{xx} f_0(x^1,R_{x^1}v)R_{x^1} +tv^T R_{x^1} \nabla_{vx} f_0(x^1,R_{x^1}v) (-2A_{v,x^1})\\
	&\quad +(-v^T)R_{x^1} \nabla_{vx}f_0(x^1,R_{x^1}v)R_{x^1} +(-v)^T \begin{bmatrix}
		\nabla_v f_0(x^1,R_{x^1}v)\nabla_v(-2A_{v,x^1}^1)\\ \nabla_v f_0(x^1,R_{x^1}v)\nabla_v(-2A_{v,x^1}^2)
	\end{bmatrix}.
\end{align*}
Since $\nabla_{xv}f_0(X(0),v)^T = \nabla_{vx} f_0(X(0),v)$ under \eqref{Cond3}, it holds that 
\begin{equation} \label{RA prop}
	\begin{bmatrix}
		\nabla_v f_0(x^1,R_{x^1}v)\nabla_v(-2A_{v,x^1}^1)\\ \nabla_v f_0(x^1,R_{x^1}v)\nabla_v(-2A_{v,x^1}^2)
	\end{bmatrix}^T = \begin{bmatrix}
		\nabla_v f_0(x^1,R_{x^1}v) \nabla_x (R_{x^1(x,v)}^1)\\ \nabla_v f_0(x^1,R_{x^1}v) \nabla_x (R_{x^1(x,v)}^2) 
	\end{bmatrix}.
\end{equation} 
Since \eqref{RA} in Lemma \ref{lem_RA}, \eqref{prop d_R}, \eqref{Cond3}, and the formula \eqref{RA prop} above,  it follows that 
\begin{equation} \label{A prop}
	\begin{split}
	&\nabla_v f_0(x^1,R_{x^1}v) (-2A_{v,x^1}) = C(R_{x^1}v)^T (-2A_{v,x^1})=-\frac{2C}{v\cdot n(x^1)} v^T (Qv) \otimes (Qv) =0, \\
	&(-v)^T\begin{bmatrix}
		\nabla_v f_0(x^1,R_{x^1}v)\nabla_v(-2A_{v,x^1}^1)\\ \nabla_v f_0(x^1,R_{x^1}v)\nabla_v(-2A_{v,x^1}^2)
	\end{bmatrix}= \left( \begin{bmatrix}
		\nabla_v f_0(x^1,R_{x^1}v) \nabla_x (R_{x^1(x,v)}^1) \\ \nabla_v f_0(x^1,R_{x^1}v) \nabla_x (R_{x^1(x,v)}^2) 
	\end{bmatrix} (-v) \right)^T=0, 
	\end{split}
\end{equation}
where $C$ is an arbitrary constant. And then, one obtains that  
\begin{equation} \label{tv comp left}
	\begin{split}
	&-\nabla_x f_0(x^1,v) +tv^T \nabla_{xx} f_0(x^1,v) + (-v^T) \nabla_{vx} f_0(x^1,v) \\ 
	&= -\nabla_x f_0(x^1,R_{x^1}v)R_{x^1} +tv^TR_{x^1} \nabla_{xx} f_0(x^1,R_{x^1}v)R_{x^1}  +tv^T R_{x^1} \nabla_{vx} f_0(x^1,R_{x^1}v) (-2A_{v,x^1})\\
	&\quad +(-v^T)R_{x^1} \nabla_{vx}f_0(x^1,R_{x^1}v)R_{x^1}.
	\end{split}
\end{equation} 
By \eqref{Av=0} and \eqref{Cond4}, the right-hand side in \eqref{tv comp} is 
\begin{equation*}
	\begin{split}
	& -\nabla_xf_0(x^1,R_{x^1}v) (R_{x^1}+2tA_{v,x^1})-t(-v^T)R_{x^1} \nabla_{xx} f_0(x^1,R_{x^1}v) R_{x^1}  \\
	&\quad +(-v^T) R_{x^1} \nabla_{vx} f_0(x^1,R_{x^1}v) (R_{x^1}+2tA_{v,x^1})  \\
	&=-\nabla_x f_0(x^1,R_{x^1}v) R_{x^1} -2Ct (R_{x^1}v)^TA_{v,x^1} +tv^T R_{x^1} \nabla_{xx}f_0(x^1,R_{x^1}v) R_{x^1}  \\
	&\quad +tv^T R_{x^1} \nabla_{vx} f_0(x^1,R_{x^1}v) (-2A_{v,x^1}) + (-v^T) R_{x^1} \nabla_{vx} f_0(x^1,R_{x^1}v) R_{x^1}\\
	& = -\nabla_x f_0(x^1,R_{x^1}v)R_{x^1} +tv^TR_{x^1} \nabla_{xx} f_0(x^1,R_{x^1}v)R_{x^1} +tv^T R_{x^1} \nabla_{vx} f_0(x^1,R_{x^1}v) (-2A_{v,x^1})\\
	&\quad +(-v^T)R_{x^1} \nabla_{vx}f_0(x^1,R_{x^1}v)R_{x^1},
	\end{split}
\end{equation*}
where $C$ is an arbitrary constant. Thus, the left-hand side in \eqref{tv comp} is the same as the right-hand side in \eqref{tv comp} under \eqref{Cond2 1}--\eqref{Cond2 4}, \eqref{Cond3}, and \eqref{Cond4}. The left-hand side in \eqref{xt comp} is as follows:  
\begin{equation*}
	 (-v^T) \nabla_{xx}f_0(x^1,v) = (-v^T) R_{x^1} \nabla_{xx} f_0(x^1,R_{x^1}v) R_{x^1} +(-v^T) \nabla_{vx} f_0(x^1,R_{x^1}v) (-2A_{v,x^1}),
\end{equation*}
by \eqref{tx comp left}. Using \eqref{vRA prop}, the right-hand side in \eqref{xt comp} can be further computed by 
 
\begin{equation*}
	\begin{split}
	 &(-v^T) \begin{bmatrix}
		\nabla_x f_0(x^1,R_{x^1}v) \nabla_x (R_{x^1(x,v)}^1) \\ 
		\nabla_x f_0(x^1,R_{x^1}v) \nabla_x (R_{x^1(x,v)}^2)\end{bmatrix}
		+(-v^T) 
		\begin{bmatrix} 
	\nabla_vf_0(x^1,R_{x^1}v) \nabla_x (-2A_{v,x^1(x,v)}^1) \\ \nabla_v f_0(x^1,R_{x^1}v) \nabla_x (-2A_{v,x^1(x,v)}^2)
	\end{bmatrix}\\
	&\quad  +(-v^T)R_{x^1} \nabla_{xx} f_0(x^1,R_{x^1}v)R_{x^1}+(-v^T)R_{x^1} \nabla_{vx} f_0(x^1,R_{x^1}v)(-2A_{v,x^1})\\
	&= (-v^T)R_{x^1} \nabla_{xx} f_0(x^1,R_{x^1}v)R_{x^1}+(-v^T)R_{x^1} \nabla_{vx} f_0(x^1,R_{x^1}v)(-2A_{v,x^1}).
	\end{split}
\end{equation*}
 
Hence, the \eqref{xt comp} condition can be deduced by \eqref{Cond2 1}--\eqref{Cond2 4},\eqref{Cond3}, and \eqref{Cond4}. Finally, the \eqref{vt comp} condition is the last remaining case. The left-hand side in \eqref{vt comp} comes from \eqref{tv comp left}:
\begin{align*}
	&-\nabla_x f_0(x^1,v) +tv^T \nabla_{xx} f_0(x^1,v) + (-v^T) \nabla_{vx} f_0(x^1,v) \\ 
	&= -\nabla_x f_0(x^1,R_{x^1}v)R_{x^1} +tv^TR_{x^1} \nabla_{xx} f_0(x^1,R_{x^1}v)R_{x^1} +tv^T R_{x^1} \nabla_{vx} f_0(x^1,R_{x^1}v) (-2A_{v,x^1})\\
	&\quad +(-v^T)R_{x^1} \nabla_{vx}f_0(x^1,R_{x^1}v)R_{x^1}.
\end{align*}

Since \eqref{Av=0} in Lemma \ref{lem_RA}, \eqref{vRA prop}, \eqref{A prop}, and 
\begin{align*}
	\quad \begin{bmatrix}
	\nabla_x f_0(x^1,R_{x^1}v) \nabla_v(R_{x^1(x,v)}^1) \\ \nabla_x f_0(x^1,R_{x^1}v) \nabla_v(R_{x^1(x,v)}^2)
	\end{bmatrix}&= (-t)\begin{bmatrix}
	\nabla_x f_0(x^1,R_{x^1}v) \nabla_x(R_{x^1(x,v)}^1) \\\nabla_x f_0(x^1,R_{x^1}v) \nabla_x(R_{x^1(x,v)}^2)
	\end{bmatrix},\\
	 \begin{bmatrix} 
	\nabla_v f_0(x^1,R_{x^1}v) \nabla_v (-2A_{v,x^1(x,v)}^1) \\ \nabla_v f_0(x^1,R_{x^1}v) \nabla_v (-2A_{v,x^1(x,v)}^2) 
	\end{bmatrix}&=(-t)  \begin{bmatrix} 
	\nabla_v f_0(x^1,R_{x^1}v) \nabla_x (-2A_{v,x^1(x,v)}^1) \\ \nabla_v f_0(x^1,R_{x^1}v) \nabla_x (-2A_{v,x^1(x,v)}^2) 
	\end{bmatrix}+ \begin{bmatrix} 
	\nabla_v f_0(x^1,R_{x^1}v) \nabla_v (-2A_{v,x^1}^1)\\ \nabla_v f_0(x^1,R_{x^1}v) \nabla_v (-2A_{v,x^1}^2)
	\end{bmatrix},
\end{align*}
 the right-hand side in \eqref{vt comp} can be simplified as 
\begin{equation*}
	\begin{split}
	& - \left(\nabla_x f_0(x^1,R_{x^1}v) R_{x^1} +\nabla_v f_0(x^1,R_{x^1}v)(-2A_{v,x^1})\right)\\
		&\quad +(-v^T) \begin{bmatrix} 
	\nabla_x f_0(x^1,R_{x^1}v) \nabla_v (R_{x^1(x,v)}^1) \\ \nabla_x f_0(x^1,R_{x^1}v) \nabla_v (R_{x^1(x,v)}^2) 
	\end{bmatrix} + (-v^T) \begin{bmatrix} 
	\nabla_v f_0(x^1,R_{x^1}v) \nabla_v (-2A_{v,x^1(x,v)}^1) \\ \nabla_v f_0(x^1,R_{x^1}v) \nabla_v (-2A_{v,x^1(x,v)}^2) 
	\end{bmatrix}\\
	&\quad +  tv^T R_{x^1} \nabla_{xx} f_0(x^1,R_{x^1}v) R_{x^1} +(-v^T)R_{x^1} \nabla_{vx} f_0(x^1,R_{x^1}v)(R_{x^1}+2tA_{v,x^1})  \\
	&=-\nabla_x f_0(x^1,R_{x^1}v)R_{x^1} +tv^T \begin{bmatrix}
	\nabla_x f_0(x^1,R_{x^1}v) \nabla_x(R_{x^1(x,v)}^1) \\\nabla_x f_0(x^1,R_{x^1}v) \nabla_x(R_{x^1(x,v)}^2)
	\end{bmatrix}+tv^T \begin{bmatrix} 
	\nabla_v f_0(x^1,R_{x^1}v) \nabla_x (-2A_{v,x^1(x,v)}^1) \\ \nabla_v f_0(x^1,R_{x^1}v) \nabla_x (-2A_{v,x^1(x,v)}^2) 
	\end{bmatrix} \\
	&\quad +(-v^T) \begin{bmatrix} 
	\nabla_v f_0(x^1,R_{x^1}v) \nabla_v (-2A_{v,x^1}^1) \\ \nabla_v f_0(x^1,R_{x^1}v) \nabla_v (-2A_{v,x^1}^2)
	\end{bmatrix} 
	+ tv^T R_{x^1} \nabla_{xx} f_0(x^1,R_{x^1}v) R_{x^1} \\
	&\quad +(-v^T)R_{x^1} \nabla_{vx} f_0(x^1,R_{x^1}v)(R_{x^1}+2tA_{v,x^1})\\
	&= -\nabla_x f_0(x^1,R_{x^1}v)R_{x^1}+tv^T R_{x^1}\nabla_{xx} f_0(x^1,R_{x^1}v) R_{x^1}+tv^T R_{x^1}\nabla_{vx} f_0(x^1,R_{x^1}v)(-2A_{v,x^1}) \\
	&\quad +(-v^T)R_{x^1} \nabla_{vx} f_0(x^1,R_{x^1}v) R_{x^1}.
	\end{split}
\end{equation*}
Hence, the \eqref{vt comp} condition can be obtained under \eqref{Cond2 1}--\eqref{Cond2 4},\eqref{Cond3}, and \eqref{Cond4}. \\
\hide
\subsubsection{Symmetric presentation of \eqref{Cond} under assumption \eqref{if}}
 Above can be more simplified. Note that the 3rd condition of \eqref{Cond2 1}--\eqref{Cond2 4} is just obvious by specular reflection BC (taking $\nabla_{v}$ twice).  From 1st and 4th condition, we have
\begin{equation}
\begin{split}
	(-2A^{T}_{v,x^{1}}) \nabla_{xv}f_{0}(x^{1}, R_{x^1}v) R_{x^{1}} &= (-2A^{T}_{v,x^{1}}) R_{x^{1}} \nabla_{xv}f_{0}(x^{1}, v) - (-2A^{T}_{v,x^{1}}) \nabla_{vv}f_{0}(x^{1}, R_{x^1}v) (-2A_{v,x^{1}}) \\
	R_{x^{1}} \nabla_{vx}f_{0}(x^{1}, R_{x^1}v)(-2A_{v,x^{1}}) &= \nabla_{vx}f_{0}(x^{1}, v)R_{x^{1}}(-2A_{v,x^{1}}) - (-2A^{T}_{v,x^{1}})\nabla_{vv}f_{0}(x^{1}, R_{x^1}v)(-2A_{v,x^{1}}).
\end{split}
\end{equation}
Plugging into 2nd condition, \eqref{Cond2 2} is rewritten as
\begin{equation} \label{re 2}
\begin{split}
 	 & \nabla_{xx}f_{0}(x^{1},v)  + \nabla_{vx}f_{0}(x^{1}, v)R_{x^{1}}A_{v,x^{1}}   + (R_{x^{1}}A_{v,x^{1}})^{T}  \nabla_{xv}f_{0}(x^{1}, v)  \\
 	 &= R_{x^{1}}\nabla_{xx}f_{0}(x^{1}, R_{x^1}v)R_{x^{1}}  +  R_{x^{1}}\nabla_{vx}f_{0}(x^{1}, R_{x^1}v)(-A_{v,x^{1}} ) \\
 	 &\quad + (-A^{T}_{v,x^{1}}) \nabla_{xv}f_{0}(x^{1}, R_{x^1}v) R_{x^{1}}  
 	 +
 	 {\color{blue}
 	 \begin{bmatrix}
 	 \nabla_{x}f_{0}(x^{1}, R_{x^1}v) \nabla_{x}[R_{x^{1}}]_{1}
 	 \\
 	 \nabla_{x}f_{0}(x^{1}, R_{x^1}v) \nabla_{x}[R_{x^{1}}]_{2}
 	 \end{bmatrix}
 	} 
\end{split}
\end{equation}
\\
{\bf Conclusion}
From  \eqref{Cond2 3} and Lemma \ref{lem_RA}, \\
the \eqref{Cond2 1} can be written as symmetric form \\
\begin{equation}  \label{sym Cond2_1} 
\begin{split}	
R_{x^{1}} \Big[ \nabla_{xv}f_{0}(x^{1},v) + \nabla_{vv}f_{0}(x^{1},v) \frac{ (Qv)\otimes (Qv)}{v\cdot n} \Big] R_{x^{1}}
&= \nabla_{xv}f_{0}(x^{1}, R_{x^1}v)  + \nabla_{vv}f_{0}(x^{1}, R_{x^1}v) \frac{(QR_{x^1}v)\otimes (QR_{x^1}v)}{R_{x^1}v\cdot n(x^1)} .
\end{split}
\end{equation}
Similarly, 4th one give
\begin{equation}    \label{sym Cond2_2} 
\begin{split}	
 R_{x^{1}} \Big[ \nabla_{vx}f_{0}(x^{1},v) +  \frac{ (Qv)\otimes (Qv)}{v\cdot n} \nabla_{vv}f_{0}(x^{1}, v) \Big] R_{x^{1}}
 &= \nabla_{vx}f_{0}(x^{1}, R_{x^1}v)  + \frac{(QR_{x^1}v)\otimes (QR_{x^1}v)}{R_{x^1}v\cdot n(x^1)} \nabla_{vv}f_{0}(x^{1}, R_{x^1}v) .
\end{split}
\end{equation}
\eqref{re 2} condition (\eqref{Cond2 2}) yields
\begin{equation}   \label{sym Cond2_3} 
\begin{split}
&R_{x^{1}}\Big[ \nabla_{xx}f_{0}(x^{1},v)  + \nabla_{vx}f_{0}(x^{1}, v) \frac{ (Qv)\otimes (Qv)}{v\cdot n}  + \frac{ (Qv)\otimes (Qv)}{v\cdot n}  \nabla_{xv}f_{0}(x^{1}, v) \Big]  R_{x^{1}}  \\
&= \nabla_{xx}f_{0}(x^{1}, R_{x^1}v)  +  \nabla_{vx}f_{0}(x^{1}, R_{x^1}v)\frac{(QR_{x^1}v)\otimes (QR_{x^1}v)}{R_{x^1}v\cdot n(x^1)}
+ \frac{(QR_{x^1}v)\otimes (QR_{x^1}v)}{R_{x^1}v\cdot n(x^1)} \nabla_{xv}f_{0}(x^{1}, R_{x^1}v)   \\
&\quad + 
{\color{blue}
	\underbrace{
R_{x^{1}}
\begin{bmatrix}
\nabla_{x}f_{0}(x^{1}, R_{x^1}v) \nabla_{x}[R_{x^{1}}]_{1}
\\
\nabla_{x}f_{0}(x^{1}, R_{x^1}v) \nabla_{x}[R_{x^{1}}]_{2}
\end{bmatrix} 
R_{x^{1}}
}_{(?)}
 	} 
\end{split}
\end{equation}
\unhide

\subsection{Proof of  Theorem \ref{thm 2}}

\begin{proof} [Proof of Theorem \ref{thm 2}]
	By the same argument of the proof of Theorem \ref{thm 1}, it suffices to set $k=1$. Through this section,  we have shown that \eqref{Cond2 1}--\eqref{Cond2 4}, \eqref{Cond3}, and \eqref{Cond4} yield $C^{2}_{t,x,v}$ regularity of $f(t,x,v)$ of \eqref{solution}. However, \eqref{Cond2 3} is just an obvious consequence of \eqref{BC} and \eqref{Cond2 4} is identical to \eqref{Cond2 1} since we assume \eqref{C2 cond34} which is the same as \eqref{Cond3} and \eqref{Cond4}. So, we omit \eqref{Cond2 3} and \eqref{Cond2 4} in the statement. 

In Remark \ref{extension C2 cond34}, under \eqref{C2 cond34}, we derived that 
\begin{equation*}
	\nabla_x f_0(x,v)R_x = \nabla_x f_0(x,R_xv) \quad \textrm{and} \quad \nabla_v f_0(x,v) \frac{(Qv) \otimes (Qv)}{v\cdot n} R_x = \nabla_v f_0(x,R_xv)\frac{(QR_xv)\otimes(QR_xv)}{R_x v\cdot n},
\end{equation*}
for all $(x,v) \in \gamma_- \cup \gamma_+$. In Remark \ref{example}, we showed that 
\begin{equation*}
	f_0(x,v)=G(x,\vert v \vert), \quad (x,v) \in \partial \O \times \R^2,
\end{equation*}
where $G$ is a $C^1_{x,v}$ function. Notice that the function $G$ must be $C^2_{x,v}$ to be $f_0 \in C^2_{x,v}(\bar\O\times \R^2)$ in Theorem \ref{thm 2}. Since $f_0(x,v)=G(x,\vert v \vert)$ be a radial function with respect to $v$ and $\nabla_x f_0(x,v) \parallel v^T$ for all $(x,v)\in \gamma_-\cup \gamma_+$, $\nabla_x f_0(x,v)$ must be $0$ on $\partial \O$.

	Now let us change \eqref{Cond2 1} and \eqref{Cond2 2} into symmetric forms. First, we multiply \eqref{Cond2 1} by $R_{x^1}$ from both left and right. Then applying \eqref{Cond2 3} and \eqref{RA}, we obtain 
	\begin{align*}
	R_{x^1} \Big[ \nabla_{xv}f_{0}(x^1,v) + \nabla_{vv}f_{0}(x^1,v) \frac{ (Qv)\otimes (Qv)}{v\cdot n(x^1)} \Big] R_{x^1}
	&= \nabla_{xv}f_{0}(x^1, R_{x^1}v)  + \nabla_{vv}f_{0}(x^1, R_{x^1}v) \frac{(QR_{x^1}v)\otimes (QR_{x^1}v)}{R_{x^1}v\cdot n(x^1)}  \notag \\
	&\quad+ 
	R_{x^1}
	\begin{bmatrix}
	\nabla_{v}f_{0}(x^1 , R_{x^1}v) \nabla_x(R^1_{x^1(x,v)})
	\\
	\nabla_{v}f_{0}(x^1, R_{x^1}v) \nabla_x(R^2_{x^1(x,v)})
	\end{bmatrix}
	R_{x^1}.
	\end{align*}
	Also, plugging the above into \eqref{Cond2 2} and using \eqref{RA} again, we obtain  
	\begin{align*}
		&R_{x^1}\Big[ \nabla_{xx}f_{0}(x^1,v)  + \nabla_{vx}f_{0}(x^1, v) \frac{ (Qv)\otimes (Qv)}{v\cdot n(x^1)}  + \frac{ (Qv)\otimes (Qv)}{v\cdot n(x^1)}  \nabla_{xv}f_{0}(x^1, v) \Big]  R_{x^1}   \\
	&= \nabla_{xx}f_{0}(x^1, R_{x^1}v)  +  \nabla_{vx}f_{0}(x^1, R_{x^1}v)\frac{(QR_{x^1}v)\otimes (QR_{x^1}v)}{R_{x^1}v\cdot n(x^1)}
	+ \frac{(QR_{x^1}v)\otimes (QR_{x^1}v)}{R_{x^1}v\cdot n(x^1)} \nabla_{xv}f_{0}(x^1, R_{x^1}v)   \\
	&\quad	-2R_x^1
		\begin{bmatrix}
		\nabla_{v}f_{0}(x^{1}, R_{x^1}v) \nabla_{v}A^{1}_{v,x^1}
		\\
		\nabla_{v}f_{0}(x^{1}, R_{x^1}v) \nabla_{v}A^{2}_{v,x^1} 
		\end{bmatrix} R_{x^1}A_{v,x^1}R_{x^1} 
		+ 
		A_{v,x^1}\begin{bmatrix}
		\nabla_{v}f_{0}(x^{1}, R_{x^1}v) \nabla_x(R^1_{x^1(x,v)})  \\
		\nabla_{v}f_{0}(x^{1}, R_{x^1}v) \nabla_x(R^2_{x^1(x,v)})
		\end{bmatrix}R_{x^1}   \\  
	&\quad
		+ R_{x^1} \begin{bmatrix}
		\nabla_{x}f_{0}(x^{1}, R_{x^1}v) \nabla_x (R^1_{x^1(x,v)}) 
		\\
		\nabla_{x}f_{0}(x^{1}, R_{x^1}v) \nabla_x(R^2_{x^1(x,v)}) 
		\end{bmatrix} R_{x^1} 
		- 2 R_{x^1}
		\begin{bmatrix}
		\nabla_{v}f_{0}(x^{1}, R_{x^1}v) \nabla_x(A^1_{v,x^1(x,v)})
		\\
		\nabla_{v}f_{0}(x^{1}, R_{x^1}v) \nabla_x(A^2_{v,x^1(x,v)})
		\end{bmatrix} R_{x^1}.
	\end{align*}
	By Lemma \ref{d_RA} and Lemma \ref{dx_A}, $\nabla_x(R^1_{x^1(x,v)}), \nabla_x(R^2_{x^1(x,v)}), \nabla_x(A^1_{v,x^1(x,v)})$, and $\nabla_v(A^2_{v,x^1(x,v)})$ depend only on $n(x^1)$ and $v$. We rewrite $x^1$ as $x$ for $(x,v) \in \gamma_-$ because $n(x^1)=x^1$. Since $\nabla_x f_0(x,v)=0$ for $x\in \partial \O$, we obtain \eqref{C2 cond 1} and \eqref{C2 cond 2}. 

Lastly, we will prove that $f(t,x,v)$ is not of class $C^2_{t,x,v}$ at time $t$ such that $t^k(t,x,v)=0$ for some $k$ if one of these conditions \eqref{C2 cond34}, \eqref{C2 cond 1}, and \eqref{C2 cond 2} for $(x,v)\in \gamma_-$ does not hold. Similar to the proof of Theorem \ref{thm 1}, it suffices to set $k=1$ and prove that $f(t,x,v)$ is not a class of $C^2_{t,x,v}$ at time $t$ satisfying $t^1(t,x,v)=0$. 

 Let $t^*$ be time $t$ such that $t^1(t,x,v)=0$. Remind that the condition \eqref{C2 cond34} was necessary to satisfy $\nabla_{xv}^Tf_0(x,v)=\nabla_{vx}f_0(x,v)$ and $\nabla_{xx}^T f_0(x,v) = \nabla_{xx}f_0(x,v)$ for $x\in \partial \O$. In other words, $\nabla_{xv}f_0(x,v)^T \neq \nabla_{vx} f_0(x,v)$ and $\nabla_{xx}^Tf_0(x,v)\neq \nabla_{xx}f_0(x,v)$ without \eqref{C2 cond34}. In $\nabla_{xv}$ and $\nabla_{vx}$ with direction $\hat{r}_1$, we derived \eqref{nabla_xv f case1} and \eqref{nabla_vx f case1} at $t^*$:
\begin{equation*}
	\nabla_{xv}f(t,x,v) = (-t) \nabla_{xx} f_0(x^1,v) + \nabla_{xv} f_0(x^1,v), \quad \nabla_{vx}f(t,x,v)=(-t) \nabla_{xx}f_0(x^1,v)+\nabla_{vx}f_0(x^1,v).
\end{equation*}
Thus, if \eqref{C2 cond34} is not provided, $\nabla_{xv}^T f(t,x,v) \neq \nabla_{vx}f(t,x,v)$. This implies that $f(t,x,v)$ is not $C^2_{t,x,v}$ at time $t^*$. Next, we do not assume \eqref{C2 cond 1} for $(x,v)\in \gamma_-$. The condition \eqref{C2 cond 1} is derived from $\nabla_{xv}$ compatibility condition \eqref{Cond2 1}. Therefore, directional derivatives  \eqref{nabla_xv f case1} and \eqref{nabla_xv f case2}  with respect to $\hat{r}_1$ and $\hat{r}_2$ are not the same. It means that $f(t,x,v)$ is not $C^2_{t,x,v}$ at time $t^*$. 

Finally, we assume that \eqref{C2 cond 2} does not hold for $(x,v)\in \gamma_-$.  
The condition \eqref{C2 cond 2} comes from $\nabla_{xx},\nabla_{xv}$ and $\nabla_{vx}$ compatibility conditions \eqref{Cond2 1}, \eqref{Cond2 2}, and \eqref{Cond2 4}. One may assume without loss of generality that the initial data $f_0$ satisfies \eqref{C2 cond34} and \eqref{C2 cond 1}. Then, only $\nabla_{xx}$ compatibility condition \eqref{Cond2 2} is not satisfied. Similar to the above, directional derivatives $\nabla_{xx}$ with respect to $\hat{r}_1$ and $\hat{r}_2$ are not the same. Then, $f(t,x,v)$ is not $C^2_{t,x,v}$ at time $t^*$ without \eqref{C2 cond 2}. This finishes the proof.   
\end{proof}

\section{Regularity estimate of $f$}
\subsection{First order estimates of characteristics} Using Definition \ref{notation},
\begin{equation*}
	V(0;t,x,v) = R_{\ell} R_{\ell-1} \cdots R_{2} R_{1} v, \quad \text{for some $\ell$ such that}\quad  t^{\ell+1} < 0 \leq t^{\ell},
\end{equation*}
where 
\[
	R_{j} = I - 2 n(x^{j})\otimes n(x^{j}).
\]
For above $\ell$,
\begin{equation*}
	X(0;t,x,v) = x^{\ell} - v^{\ell}t^{\ell},  
\end{equation*}
where inductively,
\[
	x^{k} = x^{k-1} - v^{k-1}(t^{k-1} - t^{k}),\quad 2\leq k \leq \ell,
\]
and
\[
	x^{1} = x - v(t-t^{1}) = x - v\tb.
\]
Or using rotational symmetry, we can also express
\[
	x^{\ell} = Q_{\theta}^{\ell-1}x^{1},
\]
where $Q_{\theta}$ is operator (matrix) which means rotation(on the boundary of the disk) by $\theta$. $\theta$ is uniquely determined by its first (backward in time) bounce angle $v\cdot n(\xb)$. \\

\begin{lemma} \label{der theta}
Here, $\theta$ is the angle at which $v$ is rotated to $v^1$. Moreover, $\theta>0$ is the same as the angle of rotation from $x^{k}$ to $x^{k+1}$ for $k=1,2,\cdots,l-1$. Then, derivatives of $\theta$ with respect to $x$ and $v$ are
\begin{equation} \label{d_theta}
		\nabla_x \theta =-\frac{2}{\sin \frac{\theta}{2}} Q_{-\frac{\theta}{2}}n(x^{1}), \quad \nabla_v \theta = 2\left( \frac{\tb}{\sin \frac{\theta}{2}} - \frac{1}{\vert v \vert}\right) Q_{-\frac{\theta}{2}}n(x^{1}),
\end{equation}
provided $n(x^1)\cdot v \neq0$. 
\end{lemma}

\begin{proof}
	From the definition of $\theta$, 
	\begin{equation} \label{theta}
		\cos \left( \frac{\pi}{2} - \frac{\theta}{2} \right) = \sin \left( \frac{\theta}{2} \right) = - \left[ n(x^1) \cdot \frac{v}{\vert v \vert} \right].
	\end{equation}
	Thus, taking $\nabla_x$ yields 
	\begin{align*}
		\frac{1}{2} \cos \frac{\theta}{2} \nabla_x \theta= -\frac{v}{\vert v \vert} \nabla_x \left( n(x^1)\right)=-\frac{v}{\vert v \vert} \left( I - \frac{v \otimes n(x^1)}{v \cdot n(x^1)} \right)=-\frac{v}{\vert v \vert}+ \frac{\vert v \vert}{v\cdot n(x^1)} n(x^1),
	\end{align*}	
	where we used the product rule in Lemma \ref{matrix notation} and \eqref{normal} in Lemma \ref{d_n}. Note that rotating an angle $\phi=\frac{\pi}{2}-\frac{\theta}{2}>0$ on a normal vector $n(x^1)$ gives the vector $- \frac{v}{\vert v \vert}$. In other words, it holds that  
	\begin{equation} \label{v_n}
		-\frac{v}{\vert v \vert} =  Q_{\phi} n (x^1),
	\end{equation} 
	where $Q_{\phi}= \begin{bmatrix} \cos \phi & -\sin \phi \\ \sin \phi & \cos \phi \end{bmatrix} =\begin{bmatrix} \sin \frac{\theta}{2} & -\cos \frac{\theta}{2} \\ \cos \frac{\theta}{2} & \sin \frac{\theta}{2} \end{bmatrix}$. Thus, 
	\begin{align*}
		\nabla_x \theta &= \frac{2}{\cos \frac{\theta}{2}} \left( Q_{\phi} - \frac{1}{\sin \frac{\theta}{2}} I \right) n(x^1)=\frac{2}{\cos \frac{\theta}{2}\sin \frac{\theta}{2}} \begin{bmatrix} \sin^2 \frac{\theta}{2} -1 & -\cos \frac{\theta}{2} \sin \frac{\theta}{2} \\ \sin\frac{\theta}{2}\cos \frac{\theta}{2}& \sin^2 \frac{\theta}{2} -1\end{bmatrix}n(x^1)  \\
		&= -\frac{2}{\sin \frac{\theta}{2}}  \begin{bmatrix} \cos \frac{\theta}{2} & \sin \frac{\theta}{2} \\ -\sin \frac{\theta}{2} & \cos \frac{\theta}{2} \end{bmatrix}n(x^1) = -\frac{2}{\sin \frac{\theta}{2}} Q_{-\frac{\theta}{2}}n(x^{1}).
	\end{align*}
	Similarly, taking the derivative $\nabla_v$ of both sides in \eqref{theta}: 
	\begin{align*}
		\frac{1}{2} \cos \frac{\theta}{2} \nabla_v \theta=-\frac{v}{\vert v \vert} \nabla_v \left( n(x^1)\right) -n(x^1) \left( \frac{1}{\vert v \vert} I- \frac{v \otimes v}{\vert v \vert^3} \right)&=\tb\frac{v}{\vert v \vert} \Big(I - \frac{v\otimes n(x^1)}{v\cdot n(x^1)} \Big)-n(x^1) \left( \frac{1}{\vert v \vert} I- \frac{v \otimes v}{\vert v \vert^3} \right)\\
		&=\tb \frac{v}{\vert v \vert} -\tb \frac{ \vert v \vert}{v\cdot n(x^1)} n (x^1) - \frac{1}{\vert v \vert} n(x^1) + \frac{v \cdot n(x^1)}{\vert v \vert^2} \frac{v}{\vert v \vert},
	\end{align*}
where we used the product rule in Lemma \ref{matrix notation} and \eqref{normal} in Lemma \ref{d_n}. From \eqref{v_n}, 
\begin{align*}
	\nabla_v \theta &= \frac{2}{\cos \frac{\theta}{2} \sin \frac{\theta}{2}} \left( -\tb\sin\frac{\theta}{2} \left[Q_{\phi}-\frac{1}{\sin \frac{\theta}{2}}I \right]n(x^1)+\frac{\sin^2\frac{\theta}{2}}{\vert v \vert} \left[ Q_{\phi} - \frac{1}{\sin \frac{\theta}{2}} I\right]n(x^1) \right)\\
	&=\frac{2 \tb}{\sin \frac{\theta}{2}} \begin{bmatrix} \cos \frac{\theta}{2}& \sin \frac{\theta}{2} \\ -\sin \frac{\theta}{2} &  \cos \frac{\theta}{2} \end{bmatrix}n(x^1) -\frac{2}{\vert v \vert}  \begin{bmatrix} \cos \frac{\theta}{2} & \sin\frac{\theta}{2} \\ -\sin\frac{\theta}{2} & \cos \frac{\theta}{2}  \end{bmatrix} n(x^1)\\
	&=2\left( \frac{\tb}{\sin \frac{\theta}{2}} - \frac{1}{\vert v \vert}\right) \begin{bmatrix} \cos \frac{\theta}{2} & \sin\frac{\theta}{2} \\ -\sin\frac{\theta}{2} & \cos \frac{\theta}{2}  \end{bmatrix} n(x^1)
	= 2\left( \frac{\tb}{\sin \frac{\theta}{2}} - \frac{1}{\vert v \vert}\right) Q_{-\frac{\theta}{2}}n(x^{1}).
\end{align*}
\end{proof}
\begin{lemma} \label{X,V} 
Let $(t,x,v) \in \R_+\times \O\times \R^2$. The specular characteristics $X(0;t,x,v)$ and $V(0;t,x,v)$ are defined in Definition \ref{notation}. Whenever $n(x^1)\cdot v\neq0$,  we have derivatives of the characteristics $X(0;t,x,v)$ and $V(0;t,x,v)$:
\begin{align} \label{n_x,v}
	\begin{split}
		\nabla_x X(0;t,x,v) &=  Q_{\theta}^{l-1}\left( I - \frac{v \otimes n(x^1)}{v\cdot n(x^1)} \right) +t^l l Q_{l\theta-\frac{\pi}{2}} \left( v \otimes \nabla_x \theta \right) - \frac{1}{\vert v \vert \sin \frac{\theta}{2}}Q_\theta^l  \left(v \otimes n(x^1)\right) \\
		&\quad -\frac{\vert v \vert(t-\tb-t^l)}{2}Q_{(l-\frac{1}{2})\theta -\pi} \left(n(x^1) \otimes \nabla_x \theta \right),  \\
			\nabla_v X(0;t,x,v)&=-\tb Q_{\theta}^{l-1}\left( I - \frac{v \otimes n(x^1)}{v\cdot n(x^1)} \right) -t^l Q_\theta ^l+t^l l Q_{l\theta-\frac{\pi}{2}} \left( v \otimes \nabla_v \theta \right)+\frac{\tb}{\vert v \vert \sin \frac{\theta}{2}} Q_{\theta}^l \left(v \otimes n(x^1)\right)  \\
			&\quad  -   \frac{2(l-1)\sin\frac{\theta}{2}}{\vert v \vert^3} Q_\theta^l \left(v \otimes v\right) -\frac{\vert v \vert(t-\tb-t^l)}{2}Q_{(l-\frac{1}{2})\theta - \pi}\left(n(x^1) \otimes \nabla_v \theta\right), \\		
			\nabla_x V(0;t,x,v)&= -lQ_{l\theta-\frac{\pi}{2}} \left( v\otimes \nabla_x \theta \right),\\
			\nabla_v V(0;t,x,v)&= Q_{\theta} ^l  -lQ_{l\theta-\frac{\pi}{2}} \left( v\otimes \nabla_v \theta \right),
	\end{split}
\end{align}
where $\theta$ is the angle given in Lemma \ref{der theta}, $\tb$ is the backward exit time defined in Definition \ref{notation}, $l$ is the bouncing number, and $Q_\theta$ is a rotation matrix by $\theta$. 
\end{lemma}

\begin{proof}
	Recall 
	\begin{align*}
		X(0;t,x,v) = x^l - v^l t^l , \quad V(0;t,x,v) = v^l.
	\end{align*}
	Using the rotation matrix $Q_\theta$, $x^l$ and $v^l$ can be expressed by 
	\begin{align} \label{x,v_l} 
		x^l = Q_{\theta}^{l-1} x^1, \quad v^l = Q_{\theta}^l v. 
	\end{align}
	By the chain rule, 
	\begin{align*}
		\frac{\partial{(X(0;t,x,v),V(0;t,x,v)})}{\partial{(x,v)}}= \frac{\partial{(X(0;t,x,v),V(0;t,x,v))}}{\partial{(t^l,x^l,v^l)}} \frac{\partial(t^l,x^l,v^l)}{\partial(x,v)}= \begin{bmatrix} -v^l & I & -t^l I \\ \textbf{0}_{2\times 1} & \textbf{0}_{2 \times 2} & I\end{bmatrix} \begin{bmatrix} \nabla_x t^l & \nabla_v t^l \\ \nabla_x x^l & \nabla_v x^l \\ \nabla_x v^l & \nabla_v v^l \end{bmatrix},
	\end{align*}
	where $I$ is a $2\times 2$ identity matrix. For the derivative of $X(0;t,x,v),V(0;t,x,v)$, it is necessary to find the derivative of $t^l,x^l,$ and $v^l$. Using the expression \eqref{x,v_l} and \eqref{d_matrix} in Lemma \ref{matrix notation}, we derive 
	\begin{align*}
		\nabla_x x^l &= \nabla_x \left[ Q_\theta ^{l-1} x^1 \right]=Q_{\theta}^{l-1} \nabla_x x^1 -(l-1)\left(\begin{bmatrix} \sin(l-1)\theta & \cos(l-1)\theta \\ - \cos (l-1)\theta & \sin(l-1)\theta \end{bmatrix} x^1\right) \otimes \nabla_x \theta\\
		&\hspace{.3cm} \qquad \qquad \qquad =Q_{\theta}^{l-1}\left( I - \frac{v \otimes n(x^1)}{v\cdot n(x^1)} \right)  -(l-1)\left(\begin{bmatrix} \sin(l-1)\theta & \cos(l-1)\theta \\ - \cos (l-1)\theta & \sin(l-1)\theta \end{bmatrix} x^1\right) \otimes \nabla_x \theta,\\
		\nabla_v x^l &=  \nabla_v \left[ Q_\theta ^{l-1} x^1 \right]=Q_{\theta}^{l-1} \nabla_v x^1 -(l-1)\left(\begin{bmatrix} \sin(l-1)\theta & \cos(l-1)\theta \\ - \cos (l-1)\theta & \sin(l-1)\theta \end{bmatrix} x^1\right) \otimes \nabla_v \theta\\
		&\hspace{.3cm} \qquad \qquad \qquad =-\tb Q_{\theta}^{l-1}\left( I - \frac{v \otimes n(x^1)}{v\cdot n(x^1)} \right)  -(l-1)\left(\begin{bmatrix} \sin(l-1)\theta & \cos(l-1)\theta \\ - \cos (l-1)\theta & \sin(l-1)\theta \end{bmatrix} x^1\right) \otimes \nabla_v \theta,\\
		\nabla_x v^l &= \nabla_x \left[ Q_\theta^l v \right]= -l \left( \begin{bmatrix} \sin l \theta & \cos l \theta \\ -\cos l \theta & \sin l \theta \end{bmatrix} v\right) \otimes \nabla_x \theta, \\
		\nabla_v v^l &= \nabla_v \left [ Q_\theta^l v \right]= Q_{\theta} ^l  -l \left( \begin{bmatrix} \sin l \theta & \cos l \theta \\ -\cos l \theta & \sin l \theta \end{bmatrix} v\right) \otimes \nabla_v \theta.
	\end{align*}
	For the derivative of $t^l$, we rewrite $t^l$ as
	\begin{align*}
		t^l = t-(t-t^1) - \sum_{k=1}^{l-1}(t^k-t^{k+1})= t- \tb -\sum_{k=1}^{l-1} (t^k-t^{k+1}).
	\end{align*}
	Since $\displaystyle t^k-t^{k+1}=\frac{2\sin\frac{\theta}{2}}{\vert v \vert}$ for all $k=1,2,\dots,l-1$, it holds that
	\begin{align} \label{t_ell}
		t^l = t-\tb- \frac{2(l-1)\sin \frac{\theta}{2}}{\vert v \vert}, \quad l-1 = \frac{\vert v \vert}{ 2 \sin \frac{\theta}{2}} \left(t-\tb -t^l\right).
	\end{align}
	Taking the derivative of $t^l$ with respect to $x,v$ 
	\begin{align} \label{nabla x,v t_ell}
		\begin{split}
			\nabla_x t^l &= -\nabla_x \tb -\frac{(l-1) \cos \frac{\theta}{2}}{\vert v \vert} \nabla_x \theta=-\frac{n(x^1)}{v \cdot n(x^1)}-\frac{(l-1) \cos \frac{\theta}{2}}{\vert v \vert} \nabla_x \theta=\frac{1}{\vert v \vert \sin\frac{\theta}{2}} n(x^1) -\frac{(l-1) \cos \frac{\theta}{2}}{\vert v \vert} \nabla_x \theta, \\
			\nabla_v t^l &= -\nabla_v \tb + \frac{2(l-1)\sin \frac{\theta}{2}}{\vert v \vert^3} v -\frac{(l-1) \cos \frac{\theta}{2}}{\vert v \vert} \nabla_v \theta=\tb \frac{n(x^1)}{v \cdot n(x^1)} +\frac{2(l-1)\sin \frac{\theta}{2}}{\vert v \vert^3} v -\frac{(l-1) \cos \frac{\theta}{2}}{\vert v \vert} \nabla_v \theta\\
			&\hspace{7.1cm}=-\frac{\tb}{\vert v \vert \sin \frac{\theta}{2}} n(x^1) +\frac{2(l-1)\sin \frac{\theta}{2}}{\vert v \vert^3} v -\frac{(l-1) \cos \frac{\theta}{2}}{\vert v \vert} \nabla_v \theta.
		\end{split}
	\end{align}
	Also note that, from \eqref{v_n} and \eqref{t_ell}, we have
	\begin{equation} \label{cancel}
	\begin{split}
		&-(l-1)Q_{(l-1)\theta-\frac{\pi}{2}} \left(x^1 \otimes \nabla  \theta\right)  +\frac{(l-1)\cos\frac{\theta}{2}}{\vert v \vert} Q_{\theta}^l \left(v \otimes \nabla  \theta\right) \\
		&= -(l-1)\left(Q_{(l-1)\theta -\frac{\pi}{2}} +\cos \frac{\theta}{2} Q_{l\theta}Q_{\frac{\pi}{2}-\frac{\theta}{2}}\right) \left(n(x^1) \otimes \nabla  \theta \right) \\
		&= - (l-1) Q_{(l-1)\theta -\frac{\pi}{2}} \begin{bmatrix} \sin^2 \frac{\theta}{2} & \sin \frac{\theta}{2} \cos \frac{\theta}{2} \\ -\sin \frac{\theta}{2} \cos \frac{\theta}{2} & \sin^2 \frac{\theta}{2} \end{bmatrix} \left(n(x^1) \otimes \nabla \theta \right) \\
		&= -\frac{\vert v \vert(t-\tb-t^l)}{2}Q_{(l-\frac{1}{2})\theta -\pi} \left(n(x^1) \otimes \nabla \theta \right).
	\end{split}
	\end{equation}
	Hence, using \eqref{cancel} and $x^{1}=n(x^{1})$,
	\begin{align*}
		\begin{split}
			\nabla_x X(0;t,x,v) &= \nabla_x x^l - t^l \nabla_x v^l -v^l \otimes \nabla_x t^l\\
			&= Q_{\theta}^{l-1}\left( I - \frac{v \otimes n(x^1)}{v\cdot n(x^1)} \right)  -(l-1)\left(\begin{bmatrix} \sin(l-1)\theta & \cos(l-1)\theta \\ - \cos (l-1)\theta & \sin(l-1)\theta \end{bmatrix} x^1\right) \otimes \nabla_x \theta \\
			& \quad +t^l l  \left( \begin{bmatrix} \sin l \theta & \cos l \theta \\ -\cos l \theta & \sin l \theta \end{bmatrix} v\right) \otimes \nabla_x \theta-\frac{1}{\vert v \vert \sin \frac{\theta}{2}}Q_\theta^l \left(v \otimes n(x^1)\right) +\frac{(l-1)\cos\frac{\theta}{2}}{\vert v \vert} Q_{\theta}^l \left(v \otimes \nabla_x \theta\right) \\
			&=Q_{\theta}^{l-1}\left( I - \frac{v \otimes n(x^1)}{v\cdot n(x^1)} \right) +t^l l Q_{l\theta-\frac{\pi}{2}} \left( v \otimes \nabla_x \theta \right) - \frac{1}{\vert v \vert \sin \frac{\theta}{2}}Q_\theta^l  \left(v \otimes n(x^1)\right) \\
			&\quad -(l-1)Q_{(l-1)\theta-\frac{\pi}{2}} \left(x^1 \otimes \nabla_x \theta\right)  +\frac{(l-1)\cos\frac{\theta}{2}}{\vert v \vert} Q_{\theta}^l \left(v \otimes \nabla_x \theta\right) \\
			&=Q_{\theta}^{l-1}\left( I - \frac{v \otimes n(x^1)}{v\cdot n(x^1)} \right) +t^l l Q_{l\theta-\frac{\pi}{2}} \left( v \otimes \nabla_x \theta \right) - \frac{1}{\vert v \vert \sin \frac{\theta}{2}}Q_\theta^l  \left(v \otimes n(x^1)\right) \\
			&\quad -\frac{\vert v \vert(t-\tb-t^l)}{2}Q_{(l-\frac{1}{2})\theta -\pi} \left(n(x^1) \otimes \nabla_x \theta \right), \\ 
			\nabla_v X(0;t,x,v)&=\nabla_v x^l -t^l \nabla_v v^l-v^l \otimes \nabla_v t^l\\
			&=-\tb Q_{\theta}^{l-1}\left( I - \frac{v \otimes n(x^1)}{v\cdot n(x^1)} \right)  -(l-1)\left(\begin{bmatrix} \sin(l-1)\theta & \cos(l-1)\theta \\ - \cos (l-1)\theta & \sin(l-1)\theta \end{bmatrix} x^1\right) \otimes \nabla_v \theta  \\
			&\quad   -t^l Q_\theta ^l+t^l l  \left( \begin{bmatrix} \sin l \theta & \cos l \theta \\ -\cos l \theta & \sin l \theta \end{bmatrix} v\right) \otimes \nabla_v \theta\\
			&\quad +\frac{\tb}{\vert v \vert \sin \frac{\theta}{2}} Q_{\theta}^l \left(v \otimes n(x^1)\right) -   \frac{2(l-1)\sin\frac{\theta}{2}}{\vert v \vert^3} Q_\theta^l \left(v \otimes v\right)+ \frac{(l-1)\cos \frac{\theta}{2}}{\vert v \vert} Q_\theta^l\left( v \otimes \nabla_v \theta \right),\\ 
			&= -\tb Q_{\theta}^{l-1}\left( I - \frac{v \otimes n(x^1)}{v\cdot n(x^1)} \right)  -t^l Q_\theta ^l+t^l l Q_{l\theta-\frac{\pi}{2}} \left( v \otimes \nabla_v \theta \right)
			+\frac{\tb}{\vert v \vert \sin \frac{\theta}{2}} Q_{\theta}^l \left(v \otimes n(x^1)\right)\\
			&\quad  - \frac{2(l-1)\sin\frac{\theta}{2}}{\vert v \vert^3} Q_\theta^l \left(v \otimes v\right)   -(l-1)Q_{(l-1)\theta-\frac{\pi}{2}}\left(x^1 \otimes \nabla_v \theta \right) + \frac{(l-1)\cos \frac{\theta}{2}}{\vert v \vert} Q_\theta^l\left( v \otimes \nabla_v \theta \right)\\
			&=-\tb Q_{\theta}^{l-1}\left( I - \frac{v \otimes n(x^1)}{v\cdot n(x^1)} \right) - t^l Q_\theta ^l+t^l l Q_{l\theta-\frac{\pi}{2}} \left( v \otimes \nabla_v \theta \right)+\frac{\tb}{\vert v \vert \sin \frac{\theta}{2}} Q_{\theta}^l \left(v \otimes n(x^1)\right) \\
			&\quad - \frac{2(l-1)\sin\frac{\theta}{2}}{\vert v \vert^3} Q_\theta^l \left(v \otimes v\right) 
			 -\frac{\vert v \vert(t-\tb-t^l)}{2}Q_{(l-\frac{1}{2})\theta - \pi}\left(n(x^1) \otimes \nabla_v \theta\right) ,\\
			\nabla_x V(0;t,x,v)&=\nabla_x v^l =  -l \left( \begin{bmatrix} \sin l \theta & \cos l \theta \\ -\cos l \theta & \sin l \theta \end{bmatrix} v\right) \otimes \nabla_x \theta=-lQ_{l\theta-\frac{\pi}{2}} \left( v\otimes \nabla_x \theta \right),\\
			\nabla_v V(0;t,x,v)&= \nabla_v v^l = Q_{\theta} ^l  -l \left( \begin{bmatrix} \sin l \theta & \cos l \theta \\ -\cos l \theta & \sin l \theta \end{bmatrix} v\right) \otimes \nabla_v \theta=Q_{\theta}^l -l Q_{l \theta-\frac{\pi}{2}} \left( v\otimes \nabla_v \theta \right).
		\end{split}
	\end{align*}
	\end{proof}
\begin{lemma} 
The exit backward time $\tb$ and the $l$-th bouncing backward time $t^l$ are defined in Definition \ref{notation}. Then, it holds that 
\begin{align}\label{tb esti}
	\tb \leq \frac{2\sin \frac{\theta}{2}}{ \vert v \vert}, \quad t^l \leq \frac{2\sin \frac{\theta}{2}}{ \vert v \vert}.
\end{align}
	
\end{lemma}
\begin{proof}
	Note that 
	\begin{align*}
		\tb = t-t^1= \frac{\vert x - x^1\vert}{ \vert v \vert }, \quad t^l = \frac{\vert x^l - X(0;t,x,v) \vert}{\vert v^l \vert}.
	\end{align*}
	Whenever $\theta$ is the angle at which $v$ is rotated to $v^1$, one obtains that 
	\begin{align*}
	\vert x-x^1 \vert\leq 2 \sin \frac{\theta}{2}, \quad  \vert x^l - X(0;t,x,v)\vert \leq 2 \sin \frac{\theta}{2}.
	\end{align*}
	From the above inequalities and $\vert v^l \vert = \vert v \vert $, we obtain
	\begin{align*}
		\tb \leq \frac{2\sin \frac{\theta}{2}}{ \vert v \vert}, \quad t^l \leq \frac{2\sin \frac{\theta}{2}}{ \vert v \vert}.
	\end{align*}
\end{proof}
\begin{lemma} \label{est der X,V}
	Under the same assumption in Lemma \ref{X,V}, we have estimates of derivatives for the characteristics $X(0;t,x,v)$ and $V(0;t,x,v)$
	\begin{align*}
		\begin{split}
			\vert \nabla_x X(0;t,x,v) \vert &\lesssim \frac{\vert v \vert} { \vert v \cdot n(\xb) \vert}\left( 1 + \vert v \vert t\right),\\
			\vert \nabla_v X(0;t,x,v) \vert &\lesssim \frac{1} { \vert v  \vert}\left( 1 + \vert v \vert t \right), \\ 
			\vert \nabla_x V(0;t,x,v) \vert & \lesssim \frac{\vert v \vert^3}{ \vert v \cdot n(\xb) \vert^2} \left( 1+ \vert v \vert t \right), \\
			\vert \nabla_v V(0;t,x,v) \vert & \lesssim \frac{\vert v \vert}{ \vert v \cdot n(\xb) \vert} \left( 1+ \vert v \vert t \right),
		\end{split}
	\end{align*}
	where $n(\xb)$ is outward unit normal vector at $\xb = x-\tb v \in \p\O$.  \\
\end{lemma}	
\begin{remark}
	First-order derivatives of characteristics $(X,V)$ for general 3D convex domain were obtained in \cite{GKTT2017}. Lemma \ref{est der X,V} is simple version in 2D disk and its singular orders coincide with the results of \cite{GKTT2017}.  \\
\end{remark}
\begin{proof}
	By \eqref{n_x,v} in Lemma \ref{X,V}, we have
	\begin{align*}
	\nabla_x X(0;t,x,v) &=  Q_{\theta}^{l-1}\left( I - \frac{v \otimes n(\xb)}{v\cdot n(\xb)} \right)-\frac{\vert v \vert(t-\tb-t^l)}{2}Q_{(l-\frac{1}{2})\theta -\pi} \left(n(\xb) \otimes \nabla_x \theta \right)\\
			&\quad +t^l l Q_{l\theta-\frac{\pi}{2}} \left( v \otimes \nabla_x \theta \right) - \frac{1}{\vert v \vert \sin \frac{\theta}{2}}Q_\theta^l  \left(v \otimes n(\xb)\right).
	\end{align*}
	We define a matrix norm by 
	\begin{equation*}
		\vert A \vert = \max _{i,j} a_{i,j},
	\end{equation*}
	where $a_{i,j}$ is the $(i,j)$ component of the matrix $A$. Then, we can easily check that
	\begin{equation*}
		\vert a \otimes b \vert \leq \vert a \vert \vert b \vert, 
	\end{equation*}
	for any $a,b \in \R^n$. To find upper bound of $\nabla_x X(0;t,x,v)$, we only need to consider $\nabla_x \theta$ and $t^l \times l$. By \eqref{d_theta},\eqref{t_ell}, and \eqref{tb esti}, 
	\begin{align} \label{e_1}
		\vert \nabla_x \theta \vert =\left \vert \frac{2}{\sin \frac{\theta}{2}} Q_{-\frac{\theta}{2}} n(\xb)  \right \vert \leq \frac{2 \vert v \vert}{\vert v \cdot n(\xb) \vert}, \quad t^l \times l \leq \frac{2 \sin\frac{\theta}{2}}{\vert v \vert}\times \left(\frac{\vert v \vert}{2\sin \frac{\theta}{2}}t +1 \right) \leq t+\frac{2}{\vert v \vert}. 
	\end{align}
	Using the above inequalities, we derive that 
	\begin{align*}
		\vert \nabla_x X(0;t,x,v) \vert &\leq 1+ \frac{ \vert v \vert}{\vert v \cdot n(\xb) \vert} + \frac{ \vert v \vert t}{2} \vert \nabla_x \theta \vert + t^l l \vert v \vert \vert \nabla_x \theta \vert + \frac{1}{ \vert v \cdot n(\xb) \vert} \vert v \vert \\
		&\leq 1+\frac{ \vert v \vert}{ \vert v \cdot n(\xb) \vert } + \frac{ \vert v \vert^2}{ \vert v \cdot n(\xb) \vert} t + \frac{ 2\vert v \vert^2}{\vert v \cdot n(\xb) \vert } \left( t + \frac{2}{\vert v \vert} \right) +\frac{\vert v \vert}{ \vert v \cdot n(\xb) \vert}\\
		&\lesssim \frac{\vert v \vert}{\vert v \cdot n(\xb) \vert} \left( 1+ \vert v \vert t\right).
	\end{align*} 
\end{proof}
Recall the derivative $\nabla_v X(0;t,x,v)$ in Lemma \ref{X,V}. 
\begin{align*}
\nabla_v X(0;t,x,v)&=-\tb Q_{\theta}^{l-1}\left( I - \frac{v \otimes n(\xb)}{v\cdot n(\xb)} \right)-\frac{\vert v \vert(t-\tb-t^l)}{2}Q_{(l-\frac{1}{2})\theta - \pi}\left(n(\xb) \otimes \nabla_v \theta\right) \\ 
			&\quad -t^l Q_\theta ^l+t^l l Q_{l\theta-\frac{\pi}{2}} \left( v \otimes \nabla_v \theta \right)+\frac{\tb}{\vert v \vert \sin \frac{\theta}{2}} Q_{\theta}^l \left(v \otimes n(\xb)\right)-   \frac{2(l-1)\sin\frac{\theta}{2}}{\vert v \vert^3} Q_\theta^l \left(v \otimes v\right).
\end{align*}
Similarly, to estimate $\nabla_v X(0;t,x,v)$, we need to estimate $\nabla_v \theta$. From \eqref{d_theta} and \eqref{tb esti}, we directly compute
\begin{align} \label{e_2} 
	\vert \nabla_v \theta \vert = 2 \left \vert \left( \frac{\tb}{\sin \frac{\theta}{2}}- \frac{1}{\vert v \vert} \right) Q_{-\frac{\theta}{2}}n(\xb) \right \vert \leq \frac{6}{\vert v \vert}.
\end{align}
Thus, 
\begin{align*}
	\vert \nabla_v X(0;t,x,v) \vert &\leq \tb \left( 1+ \frac{\vert v \vert}{ \vert v \cdot n(\xb) \vert} \right) +\frac{\vert v \vert t}{2} \vert \nabla_v \theta \vert + t^l + \left( t^l \times l \right) \vert v \vert \vert \nabla_v \theta \vert + \frac{ \tb}{ \vert v \vert \sin \frac{\theta}{2}} \vert v \vert + \frac{2(l-1) \sin\frac{\theta}{2}}{\vert v \vert^3} \vert v \vert^2 \\
	&\leq \frac{2\sin \frac{\theta}{2}}{\vert v \vert} \left( 1+ \frac{\vert v \vert}{ \vert v \cdot n(\xb) \vert} \right)+\frac{\vert v \vert t}{2} \times \frac{6}{\vert v \vert}+\frac{2\sin \frac{\theta}{2}}{\vert v \vert} + 6\left(t +\frac{2}{\vert v \vert} \right)\\
	&\quad + \frac{2\sin \frac{\theta}{2}}{\vert v \vert}  \frac{1}{\vert v \vert \sin \frac{\theta}{2}} \vert v \vert + \frac{(t-\tb-t^l)}{\vert v \vert^2} \vert v \vert^2\\
	&\lesssim \frac{1}{\vert v \vert} \left (1+ \vert v \vert t \right),
\end{align*}
where we used \eqref{tb esti} and \eqref{e_1}. For $\nabla_{x,v} V(0;t,x,v)$, using \eqref{n_x,v}, \eqref{t_ell}, \eqref{e_1}, and \eqref{e_2} gives
\begin{align*}
	\vert \nabla_x V(0;t,x,v) \vert &= \left \vert -l Q_{l\theta-\frac{\pi}{2}} \left( v \otimes \nabla_x \theta \right) \right \vert  \leq \left(\frac{\vert v \vert}{\vert 2\sin \frac{\theta}{2}\vert} t +1\right) \vert v \vert \vert \nabla_x \theta \vert \lesssim \frac{ \vert v \vert^3}{\vert v \cdot n(\xb) \vert^2}(1+ \vert v \vert t), \\
	\vert \nabla_v V(0;t,x,v) \vert&= \left \vert Q_{\theta}^l - lQ_{l\theta -\frac{\pi}{2}} (v \otimes \nabla_v \theta) \right \vert \leq 1+  \left(\frac{\vert v \vert}{\vert 2 \sin \frac{\theta}{2}\vert} t +1\right) \vert v \vert \vert \nabla_v \theta \vert\lesssim \frac{\vert v \vert}{\vert v \cdot n(\xb) \vert} (1+ \vert v \vert t).  
\end{align*}

\subsection{Second-order estimates of characteristics} 
\begin{lemma}
	$n(\xb)$ is outward unit normal vector at $\xb\in \partial \O$. For $(\xb,v) \notin \gamma_0$, it follows that 
	\begin{equation} \label{est der n}
		\vert \nabla_x [n(\xb)] \vert \lesssim \frac{\vert v \vert}{\vert v \cdot n(\xb) \vert}, \quad \vert \nabla_v [n(\xb)] \vert \lesssim \frac{1}{\vert v \vert}. 
	\end{equation}
\end{lemma}
\begin{proof}
We denote the components of $v$ and $n(\xb)$ by $(v_1,v_2)$ and $(n_1,n_2)$. By \eqref{normal} in Lemma \ref{d_n} and \eqref{tb esti}, we have 
\begin{align*}
	\nabla_x[n(\xb)] &= I- \frac{v\otimes n(\xb)}{v\cdot n(\xb)} = \frac{1}{v\cdot n(\xb)}\begin{bmatrix}
	v_2n_2 & -v_1n_2 \\
	-v_2n_1 & v_1n_1
	\end{bmatrix},\\
	\nabla_v [n(\xb)]&=-\tb\left(I- \frac{v\otimes n(\xb)}{v\cdot n(\xb)}\right)=\frac{-\tb}{v\cdot n(\xb)} \begin{bmatrix}
	v_2n_2 & -v_1n_2 \\
	-v_2n_1 & v_1n_1
	\end{bmatrix},
\end{align*}
	which is further bounded by 
\begin{equation*}
	\vert \nabla_x [n(\xb)] \vert \lesssim \frac{\vert v \vert  }{\vert v \cdot n(\xb)\vert}, \quad \vert \nabla_v [n(\xb)] \vert \lesssim \frac{1}{\vert v \vert}. 
\end{equation*}
\end{proof}
\begin{lemma}
The exit backward time $\tb$ and the $l$-th bouncing backward time $t^l$ are defined in Definition \ref{notation}. Then, we have the following estimates 
\begin{equation} \label{est der t_ell}
	\begin{split}
	&\vert \nabla_x t^1 \vert \lesssim \frac{1}{\vert v \vert \vert \sin\frac{\theta}{2} \vert}, \quad \vert \nabla_v t^1 \vert \lesssim \frac{1}{\vert v \vert^2},\\
	&\vert \nabla_x t^l \vert \lesssim \frac{1}{\vert v \vert \sin^2\frac{\theta}{2}}(1+\vert v\vert t), \quad \vert \nabla_v t^l \vert \lesssim \frac{1}{\vert v \vert^2 \vert \sin \frac{\theta}{2} \vert }(1+\vert v \vert t),
	\end{split}
\end{equation}
whenever $v\cdot n(\xb) \neq 0$. 
\end{lemma}

\begin{proof}
Since $t^1 = t-t_b$, it follows from Lemma \ref{nabla xv b} that 
\begin{align*}
	\nabla_x t^1 = -\nabla_x \tb = -\frac{n(\xb)}{v\cdot n(\xb)}, \quad \nabla_v t^1 = -\nabla_v \tb = \tb \frac{n(\xb)}{v\cdot n(\xb)}.
\end{align*}
Using the above and \eqref{tb esti} implies that
\begin{align*}
	\vert \nabla_x t^1 \vert \lesssim \frac{1}{\vert v \vert \left \vert \sin \frac{\theta}{2}\right \vert}, \quad \vert \nabla_v t^1 \vert \lesssim \frac{1}{\vert v \vert^2}. 
\end{align*}
By \eqref{nabla x,v t_ell} in the proof of Lemma \ref{X,V}, we have 
\begin{align*}
	\nabla_x t^l &= \frac{1}{\vert v \vert \sin \frac{\theta}{2}} n(\xb) - \frac{(l-1)\cos \frac{\theta}{2}}{\vert v \vert} \nabla_x \theta, \\
	\nabla_v t^l &= -\frac{\tb}{\vert v \vert \sin \frac{\theta}{2}} n(\xb) +\frac{2(l-1)\sin \frac{\theta}{2}}{\vert v \vert^3} v -\frac{(l-1)\cos \frac{\theta}{2}}{\vert v \vert} \nabla_v \theta. 
\end{align*}
By \eqref{t_ell} in the proof of Lemma \ref{X,V}, the bouncing number $l$ can be bounded by
\begin{equation} \label{ell est}
	l= 1+ \frac{\vert v \vert}{2\sin \frac{\theta}{2}} (t-\tb-t^l) \leq 1+ \frac{\vert v \vert }{2\left \vert \sin \frac{\theta}{2}\right \vert} t \lesssim \frac{1}{\left \vert \sin\frac{\theta}{2}\right \vert} (1+\vert v \vert t). 
\end{equation}
Then, from \eqref{tb esti}, \eqref{e_1},\eqref{e_2}, and \eqref{ell est}, one obtains that 
\begin{align*}
	\vert \nabla_x t^l \vert &\lesssim \frac{1}{\vert v \vert \vert \sin \frac{\theta}{2}\vert} + \frac{1}{\vert v \vert \sin^2 \frac{\theta}{2}}(1+\vert v \vert t)\lesssim \frac{1}{\vert v \vert \sin^2 \frac{\theta}{2}} (1+\vert v \vert t), \\
	\vert\nabla_v t^l \vert &\lesssim \frac{1}{\vert v \vert^2}+\frac{1}{\vert v \vert^2}(1+\vert v \vert t) + \frac{1}{\vert v \vert^2}(1+\vert v \vert t) \lesssim \frac{1}{\vert v \vert^2} (1+\vert v \vert t).
\end{align*}
\end{proof}

\begin{lemma} \label{2nd est der X,V}
The characteristics $X(0;t,x,v)$ and $V(0;t,x,v)$ are defined in Definition \ref{notation}. Under the same assumption in Lemma \ref{X,V}, we have estimates for the second derivatives of characteristics
\begin{equation*}
	\begin{split}
		&\vert \nabla_{xx} X(0;t,x,v) \vert \lesssim \frac{\vert v \vert^4}{\vert v \cdot n(\xb)\vert^4}(1+\vert v \vert^2 t^2), \quad \vert \nabla_{vx} X(0;t,x,v) \vert \lesssim \frac{\vert v \vert^2}{\vert v \cdot n(\xb) \vert^3}(1+\vert v \vert^2 t^2), \\
		&\vert \nabla_{xv} X(0;t,x,v) \vert \lesssim \frac{\vert v \vert^2}{\vert v \cdot n(\xb) \vert^3}(1+\vert v \vert^2 t^2), \quad \vert \nabla_{vv}X(0;t,x,v) \vert \lesssim \frac{1}{\vert v \cdot n(\xb) \vert^2}(1+\vert v \vert^2 t^2),\\ 
		&\vert \nabla_{xx} V(0;t,x,v) \vert \lesssim \frac{\vert v \vert^5}{\vert v \cdot n(\xb) \vert^4} (1+\vert v \vert^2t^2), \quad \vert \nabla_{vx} V(0;t,x,v) \vert \lesssim \frac{\vert v \vert^3}{\vert v \cdot n(\xb)\vert^3}(1+\vert v \vert^2 t^2),\\ 
		&\vert \nabla_{xv} V(0;t,x,v) \vert \lesssim \frac{\vert v \vert^3}{\vert v \cdot n(\xb)\vert^3}(1+\vert v \vert^2 t^2), \quad \vert \nabla_{vv} V(0;t,x,v) \vert \lesssim \frac{\vert v \vert}{\vert v \cdot n(\xb)\vert^2} (1+\vert v \vert^2 t^2), 		
	\end{split}
\end{equation*}
where $\vert \nabla_{xx,xv,vv} X(0;t,x,v) (or \; V(0;t,x,v))\vert $ is given by $\sup_{i,j} \vert \nabla_{ij}X(0;t,x,v) (or \; \nabla_{ij}V(0;t,x,v))\vert$ for $i,j \in \{x_1,x_2,v_1,v_2\}$. 
\end{lemma}
\begin{proof}
We denote the components of $v$ and $n(\xb)$ by $(v_1,v_2)$ and $(n_1,n_2)$. To estimate $\vert \nabla_{xx} X(0;t,x,v)\vert$, we need to determine which component in the matrix $\nabla_x X(0;t,x,v)$ has the highest singularity $\frac{1}{\sin \frac{\theta}{2}}$ and travel length $(1+\vert v \vert t)$ order when we take the derivative with respect to $x$. In estimates \eqref{e_1},\eqref{est der n},\eqref{est der t_ell}, and \eqref{ell est}, we already checked singularity and travel length order for some terms. Considering these estimates, we get the highest singularity and travel length order in the $x$-derivative of the (1,1) component of the matrix $\nabla_x X(0;t,x,v)$. Hence, we only consider the (1,1) component among components in the matrix $\nabla_x X(0;t,x,v)$. In fact, from Lemma \ref{X,V}, the (1,1) component $[\nabla_x X(0;t,x,v)]_{(1,1)}$ of the matrix $\nabla_x X(0;t,x,v)$ is 
\begin{align*}
	&[\nabla_x X(0;t,x,v)]_{(1,1)}\\
	&= \cos((l-1)\theta) \frac{v_2n_2}{v\cdot n(\xb)} + \sin((l-1)\theta) \frac{v_2 n_1}{v\cdot n(\xb)}\\
	&\quad +\frac{\vert v \vert(t^1-t^l)}{\sin \frac{\theta}{2}}\left(-n_1^2 \cos ((l-\frac{1}{2})\theta) \cos \frac{\theta}{2} -n_1n_2 \cos ((l-\frac{1}{2})\theta) \sin \frac{\theta}{2}+n_1n_2 \sin((l-\frac{1}{2})\theta)\cos \frac{\theta}{2} \right. \\
	&\left. \qquad \qquad \qquad \qquad  +n_2^2 \sin ((l-\frac{1}{2})\theta) \sin \frac{\theta}{2} \right)\\
	&\quad-\frac{2t^l l}{\sin \frac{\theta}{2}} \left( v_1n_1 \sin l\theta \cos \frac{\theta}{2} +v_1 n_2 \sin l\theta \sin \frac{\theta}{2} +v_2n_1 \cos l\theta \cos \frac{\theta}{2} +v_2n_2 \cos l\theta \sin \frac{\theta}{2}\right)\\
	&\quad -\frac{1}{\vert v \vert \sin \frac{\theta}{2}}\left( v_1 n_1 \cos l\theta -v_2n_1 \sin l\theta\right)\\
	&\lesssim \frac{\vert v \vert}{\vert v \cdot n(\xb) \vert} + \frac{1}{\vert \sin \frac{\theta}{2} \vert} (1+\vert v \vert t) +\frac{1}{\vert \sin \frac{\theta}{2}\vert} \lesssim \frac{\vert v \vert }{\vert v \cdot n(\xb) \vert}(1+\vert v \vert t) ,
\end{align*}
where the first inequality comes from \eqref{tb esti}, \eqref{ell est}, and
\begin{equation*}
	t^1-t^l= \frac{2(l-1)\sin\frac{\theta}{2}}{\vert v \vert} \lesssim \frac{1}{\vert v \vert}(1+\vert v \vert t).
\end{equation*}
Similarly, the $(1,1)$ components of matrices $\nabla_v X(0;t,x,v), \nabla_x V(0;t,x,v)$, and $\nabla_v V(0;t,x,v)$ satisfy inequalities in Lemma \ref{est der X,V}. Similar as estimate $\vert \nabla_{xx} X(0;t,x,v)\vert$, we only consider $(1,1)$ components of derivative matrices for $X(0;t,x,v)$ and $V(0;t,x,v)$ to get estimates. When we differentiate $[\nabla_x X(0;t,x,v)]_{(1,1)}$ with respect to $x$, the terms containing $\frac{t^l l}{\sin \frac{\theta}{2}}$ are main terms that increase the singularity $\frac{1}{\sin \frac{\theta}{2}}$ and travel length $(1+\vert v \vert t)$ order. $\frac{t^l l}{\sin \frac{\theta}{2}}$ has a singularity order 1 and travel length order 1 because 
\begin{equation*}
	\left \vert \frac{t^l l}{\sin \frac{\theta}{2}}\right \vert  \lesssim \frac{1}{\vert \sin \frac{\theta}{2}\vert } \times \frac{\vert \sin \frac{\theta}{2}\vert }{\vert v \vert} \times \frac{1}{\vert \sin \frac{\theta}{2}\vert}(1+\vert v \vert t)=\frac{ \vert v \vert}{\vert v \cdot n(\xb)\vert}(1+\vert v \vert t),
\end{equation*} 
where we have used \eqref{tb esti} and \eqref{ell est}. On the other hand, if we take of the term $\frac{t^l l}{\sin \frac{\theta}{2}}$ with respect to $x$, the singularity and travel length order become $4$ and $2$ respectively: 
\begin{align*}
	\left \vert \nabla_x \left(\frac{t^l l}{\sin \frac{\theta}{2}}\right)\right \vert =\left \vert \frac{l}{\sin \frac{\theta}{2}} \nabla_x t^l -\frac{t^l l\cos\frac{\theta}{2}}{2\sin^2 \frac{\theta}{2}} \nabla_x \theta\right \vert &\lesssim \frac{1}{\vert v \vert \sin^4 \frac{\theta}{2}}(1+\vert v \vert^2t^2) + \frac{1}{\vert v \vert \vert \sin^3 \frac{\theta}{2}\vert}(1+\vert v \vert t)\\
	&\lesssim \frac{\vert v \vert^3}{\vert v\cdot n(\xb)\vert^4} (1+\vert v \vert ^2 t^2), 
\end{align*}
where \eqref{tb esti}, \eqref{e_1}, \eqref{est der t_ell}, and \eqref{ell est} have been used. Hence, it suffices to estimate the following terms in $[\nabla_x X(0;t,x,v)]_{(1,1)}$ 
\begin{equation*}
	-\frac{2t^l l}{\sin \frac{\theta}{2}} \left( v_1n_1 \sin l\theta \cos \frac{\theta}{2} +v_1 n_2 \sin l\theta \sin \frac{\theta}{2} +v_2n_1 \cos l\theta \cos \frac{\theta}{2} +v_2n_2 \cos l\theta \sin \frac{\theta}{2}\right):=I_{1},
\end{equation*}
to obtain estimate for $\vert \nabla_{xx} X(0;t,x,v) \vert$. Taking the $x$-derivative to the above terms, one obtains
\begin{align*}
	\nabla_x I_{1} &= \left ( \frac{-2l\nabla_x t^l}{\sin \frac{\theta}{2}}  +\frac{2t^l l\cos\frac{\theta}{2}\nabla_x \theta}{2\sin^2 \frac{\theta}{2}} \right )\Big( v_1n_1 \sin l\theta \cos \frac{\theta}{2} +v_1 n_2 \sin l\theta \sin \frac{\theta}{2} +v_2n_1 \cos l\theta \cos \frac{\theta}{2} +v_2n_2 \cos l\theta \sin \frac{\theta}{2}\Big)\\
	&\quad -\frac{2t^l l}{\sin \frac{\theta}{2}} \left( v_1 \sin l \theta \cos \frac{\theta}{2}\nabla_x n_1  +lv_1 n_1 \cos l\theta \cos \frac{\theta}{2}\nabla_x \theta -\frac{1}{2} v_1 n_1 \sin l\theta \sin \frac{\theta}{2} \nabla_x \theta \right. \\ 
	&\qquad \qquad \quad  \left. + v_1 \sin l \theta \sin \frac{\theta}{2}\nabla_x n_2  +lv_1 n_2 \cos l\theta \sin \frac{\theta}{2}\nabla_x \theta +\frac{1}{2} v_1 n_2 \sin l\theta \cos \frac{\theta}{2} \nabla_x \theta \right. \\
	&\qquad \qquad \quad  \left. + v_2  \cos l\theta \cos \frac{\theta}{2}\nabla_x n_1 -lv_2 n_1 \sin l \theta \cos \frac{\theta}{2} \nabla_x \theta -\frac{1}{2} v_2n_1\cos l\theta \sin \frac{\theta}{2}\nabla_x \theta \right. \\ 
	&\qquad \qquad \quad \left. + v_2 \cos l\theta \sin \frac{\theta}{2}\nabla_x n_2  -lv_2 n_2 \sin l\theta \sin \frac{\theta}{2} \nabla_x \theta + \frac{1}{2} v_2 n_2 \cos l \theta \cos \frac{\theta}{2}\nabla_x \theta \right).
\end{align*}
Using \eqref{tb esti},\eqref{e_1},\eqref{est der n},\eqref{est der t_ell}, and \eqref{ell est}, one can further bound the above as 
\begin{align*}
	\vert \nabla_x I_{1}\vert &\lesssim \frac{\vert v \vert^3}{\vert v \cdot n(\xb)\vert^4} (1+\vert v \vert^2t^2) \times \vert v \vert + \frac{1}{\vert v \cdot n(\xb)\vert}(1+\vert v \vert t) \times \left( \vert v \vert \vert \nabla_x n(\xb) \vert + l\vert v \vert\vert \nabla_x \theta \vert\right )\\
	&\lesssim \frac{\vert v \vert^4}{\vert v \cdot n(\xb) \vert^4} (1+\vert v \vert^2 t^2). 
\end{align*}
Therefore, we get 
\begin{align*}
	\vert \nabla_{xx} X(0;t,x,v) \vert \lesssim \frac{\vert v \vert^4}{\vert v\cdot n(\xb) \vert^4} (1+\vert v \vert^2t^2). 
\end{align*}
For estimate of $\vert \nabla_{vx} X(0;t;x,v)\vert$, similar to the case $\vert \nabla_{xx} X(0;t,x,v) \vert$, we only consider terms $I_1$. By taking the $v$-derivative to $I_1$, we obtain 
\begin{align*}
	\nabla_v I_1&=\left ( \frac{-2l \nabla_v t^l}{\sin \frac{\theta}{2}} +\frac{2t^l l\cos\frac{\theta}{2}\nabla_v \theta}{2\sin^2 \frac{\theta}{2}} \right )\Big( v_1n_1 \sin l\theta \cos \frac{\theta}{2} +v_1 n_2 \sin l\theta \sin \frac{\theta}{2} +v_2n_1 \cos l\theta \cos \frac{\theta}{2} +v_2n_2 \cos l\theta \sin \frac{\theta}{2}\Big)\\
	&\quad  -\frac{2t^l l}{\sin \frac{\theta}{2}} \left(  n_1 \sin l \theta \cos \frac{\theta}{2}\nabla_v v_1 + v_1  \sin l \theta \cos \frac{\theta}{2}\nabla_v n_1 +lv_1 n_1 \cos l\theta \cos \frac{\theta}{2}\nabla_v \theta -\frac{1}{2} v_1 n_1 \sin l\theta \sin \frac{\theta}{2} \nabla_v \theta \right. \\ 
	&\qquad \qquad \quad  \left. + n_2\sin l\theta \sin \frac{\theta}{2} \nabla_v v_1+ v_1\sin l \theta \sin \frac{\theta}{2} \nabla_v n_2  +lv_1 n_2 \cos l\theta \sin \frac{\theta}{2}\nabla_v \theta +\frac{1}{2} v_1 n_2 \sin l\theta \cos \frac{\theta}{2} \nabla_v \theta \right. \\
	&\qquad \qquad \quad  \left. + n_1 \cos l\theta \cos \frac{\theta}{2} \nabla_v v_2+v_2\cos l\theta \cos \frac{\theta}{2} \nabla_v n_1  -lv_2 n_1 \sin l \theta \cos \frac{\theta}{2} \nabla_v \theta -\frac{1}{2} v_2n_1\cos l\theta \sin \frac{\theta}{2}\nabla_v \theta \right. \\ 
	&\qquad \qquad \quad \left. +n_2\cos l \theta \sin \frac{\theta}{2} \nabla_v v_2+ v_2\cos l\theta \sin \frac{\theta}{2} \nabla_v n_2  -lv_2 n_2 \sin l\theta \sin \frac{\theta}{2} \nabla_v \theta + \frac{1}{2} v_2 n_2 \cos l \theta \cos \frac{\theta}{2}\nabla_v \theta \right).
\end{align*}
Using \eqref{tb esti},\eqref{e_2},\eqref{est der n},\eqref{est der t_ell}, and \eqref{ell est} yields that 
\begin{align*}
	\vert \nabla_v I_1 \vert &\lesssim \left( \frac{1}{\vert v \vert^2 \vert \sin^3 \frac{\theta}{2}\vert}  (1+\vert v \vert^2 t^2)+\frac{1}{\vert v \vert^2 \vert \sin ^2 \frac{\theta}{2}\vert}(1+\vert v \vert t) \right)\times \vert v \vert+\frac{1}{\vert v \cdot n(\xb) \vert} ( 1+ \vert v \vert \vert \nabla_v n(\xb) \vert +l\vert v \vert \vert \nabla_v \theta\vert)\\
	&\lesssim \frac{\vert v \vert^2}{\vert v \cdot n(\xb) \vert^3} (1+\vert v \vert ^2 t^2). 
\end{align*} 
Hence, one obtains that 
\begin{align*}
	\vert \nabla_{vx} X(0;t,x,v) \vert \lesssim \frac{\vert v \vert^2}{\vert v \cdot n(\xb) \vert^3} (1+\vert v \vert ^2 t^2). 
\end{align*}
By Lemma \ref{X,V}, we write the $(1,1)$ component of $\nabla_v X(0;t,x,v)$: 
\begin{align*}
	&[\nabla_v X(0;t,x,v)]_{(1,1)}\\
	&=-\tb \left(\cos (l-1)\theta \frac{v_2n_2}{v\cdot n(\xb)} +\sin (l-1)\theta \frac{v_2n_1}{v \cdot n(\xb)} \right) -t^l \cos l\theta \\
	&\quad +2lt^l  \left( \frac{\tb}{\sin \frac{\theta}{2}} -\frac{1}{\vert v \vert}\right) \left( v_1 n_1 \sin l\theta \cos \frac{\theta}{2} +v_1 n_2 \sin l\theta \sin \frac{\theta}{2} +v_2 n_1 \cos l\theta \cos \frac{\theta}{2} +v_2 n_2 \cos l\theta \sin \frac{\theta}{2}\right) \\
	&\quad + \frac{\tb}{\vert v \vert \sin \frac{\theta}{2}} (v_1n_1\cos l\theta -v_2 n_1 \sin l \theta) -\frac{2(l-1)\sin \frac{\theta}{2}}{\vert v \vert^3}(v_1^2\cos l\theta - v_1v_2 \sin l \theta) \\
	&\quad -\vert v \vert (t^1-t^l) \left(\frac{\tb}{\sin \frac{\theta}{2}}-\frac{1}{\vert v \vert}\right) \left(-n_1^2 \cos (l-\frac{1}{2})\theta \cos \frac{\theta}{2} +n_1n_2 \sin (l-1)\theta +n_2^2 \sin (l-\frac{1}{2})\theta \sin \frac{\theta}{2}\right).
\end{align*}
Similar to $\nabla_x X(0;t,x,v)$, main terms in $\nabla_v X(0;t,x,v)$ are 
\begin{align*}
	2lt^l  \left( \frac{\tb}{\sin \frac{\theta}{2}} -\frac{1}{\vert v \vert}\right) \left( v_1 n_1 \sin l\theta \cos \frac{\theta}{2} +v_1 n_2 \sin l\theta \sin \frac{\theta}{2} +v_2 n_1 \cos l\theta \cos \frac{\theta}{2} +v_2 n_2 \cos l\theta \sin \frac{\theta}{2}\right):=I_2.
\end{align*}
As we take derivative to $\nabla_v X(0;t,x,v)$ with respect to $x$ and $v$, $I_2$ mainly contributes to increase singularity and travel length order. Thus, we only differentiate terms $I_2$ to get estimate for $\vert \nabla_{xv} X(0;t,x,v) \vert $ and $\vert \nabla_{vv} X(0;t,x,v)\vert$. Firstly, taking $x$ derivative to $I_2$ gives 
\begin{align*} 
	\nabla_x I_2 &= 2l \nabla_x t^l \left( \frac{\tb}{\sin \frac{\theta}{2}} -\frac{1}{\vert v \vert}\right) \left( v_1 n_1 \sin l\theta \cos \frac{\theta}{2} +v_1 n_2 \sin l\theta \sin \frac{\theta}{2} +v_2 n_1 \cos l\theta \cos \frac{\theta}{2} +v_2 n_2 \cos l\theta \sin \frac{\theta}{2}\right)\\
	&\quad + 2lt^l \left( \frac{\nabla_x \tb}{\sin \frac{\theta}{2}} -\frac{\tb\cos \frac{\theta}{2}\nabla_x \theta}{2\sin^2 \frac{\theta}{2}} \right)\Big( v_1 n_1 \sin l\theta \cos \frac{\theta}{2} +v_1 n_2 \sin l\theta \sin \frac{\theta}{2} +v_2 n_1 \cos l\theta \cos \frac{\theta}{2} +v_2 n_2 \cos l\theta \sin \frac{\theta}{2}\Big)\\
	&\quad +2lt^l  \left( \frac{\tb}{\sin \frac{\theta}{2}} -\frac{1}{\vert v \vert}\right)\left( v_1 \sin l\theta \cos \frac{\theta}{2} \nabla_xn_1+lv_1 n_1 \cos l\theta \cos \frac{\theta}{2} \nabla_x \theta -\frac{1}{2}v_1 n_1 \sin l \theta \sin \frac{\theta}{2} \nabla_x \theta \right. \\ 
	&\qquad \qquad \qquad \qquad \qquad \quad +\left. v_1\sin l \theta \sin \frac{\theta}{2} \nabla_x n_2+lv_1n_2\cos l \theta \sin \frac{\theta}{2} \nabla_x \theta +\frac{1}{2} v_1 n_2 \sin l\theta \cos \frac{\theta}{2} \nabla_x \theta  \right.\\ 
	&\qquad \qquad \qquad \qquad \qquad \quad +\left. v_2\cos l\theta \cos \frac{\theta}{2}\nabla_x n_1 -lv_2n_1 \sin l \theta \cos \frac{\theta}{2} \nabla_x \theta -\frac{1}{2} v_2 n_1 \cos l\theta \sin \frac{\theta}{2} \nabla_x \theta \right. \\ 
	&\qquad \qquad \qquad \qquad \qquad \quad +\left. v_2 \cos l\theta \sin \frac{\theta}{2} \nabla_x n_2 -lv_2n_2 \sin l \theta \sin \frac{\theta}{2} \nabla_x \theta +\frac{1}{2} v_2n_2 \cos l\theta \cos \frac{\theta}{2} \nabla_x \theta \right).
\end{align*}
Hence, it follows from \eqref{tb esti},\eqref{e_1},\eqref{est der n},\eqref{est der t_ell}, and \eqref{ell est} that 
\begin{align*}
	\vert \nabla_x I_2 \vert &\lesssim \frac{1}{\vert v \vert \vert \sin ^3 \frac{\theta}{2} \vert} (1+\vert v \vert^2 t^2) \times \frac{1}{\vert v \vert} \times \vert v \vert + \frac{1}{\vert v\vert}(1+\vert v \vert t) \times\frac{1}{\vert v \vert \sin^2 \frac{\theta}{2}} \times \vert v \vert \\
	&\quad + \frac{1}{\vert v \vert}(1+\vert v \vert t)\times \frac{1}{\vert v \vert} \times \frac{\vert v \vert}{\sin ^2\frac{\theta}{2}} (1+\vert v \vert t)\\
	&\lesssim \frac{\vert v \vert^2}{\vert v \cdot n(\xb) \vert^3} (1+\vert v \vert^2 t^2),
\end{align*}
which yields $\vert \nabla_{xv} X(0;t,x,v) \vert$ estimate 
\begin{align*}
	\vert \nabla_{xv} X(0;t,x,v) \lesssim \frac{\vert v \vert^2}{\vert v \cdot n(\xb) \vert^3} (1+\vert v \vert^2 t^2).
\end{align*}
Similarly, we consider $\nabla_v I_2$: 
\begin{align*}
\nabla_v I_2 &= 2l \nabla_v t^l \left( \frac{\tb}{\sin \frac{\theta}{2}} -\frac{1}{\vert v \vert}\right) \left( v_1 n_1 \sin l\theta \cos \frac{\theta}{2} +v_1 n_2 \sin l\theta \sin \frac{\theta}{2} +v_2 n_1 \cos l\theta \cos \frac{\theta}{2} +v_2 n_2 \cos l\theta \sin \frac{\theta}{2}\right)\\
	&\quad + 2lt^l \left( \frac{\nabla_v \tb}{\sin \frac{\theta}{2}} -\frac{\tb\cos \frac{\theta}{2}}{2\sin^2 \frac{\theta}{2}} \nabla_v \theta+\frac{v}{\vert v \vert^3}\right)\left( v_1 n_1 \sin l\theta \cos \frac{\theta}{2} +v_1 n_2 \sin l\theta \sin \frac{\theta}{2} +v_2 n_1 \cos l\theta \cos \frac{\theta}{2} \right.\\
	&\left. \qquad \hspace{5.7cm} +v_2 n_2 \cos l\theta \sin \frac{\theta}{2}\right)\\
	&\quad +2lt^l  \left( \frac{\tb}{\sin \frac{\theta}{2}} -\frac{1}{\vert v \vert}\right)\left( n_1 \sin l\theta \cos \frac{\theta}{2} \nabla_v v_1+v_1 \sin l\theta \cos \frac{\theta}{2} \nabla_vn_1+lv_1 n_1 \cos l\theta \cos \frac{\theta}{2} \nabla_v \theta \right.\\
	&\qquad \qquad \qquad \qquad \qquad \quad\left. -\frac{1}{2}v_1 n_1 \sin l \theta \sin \frac{\theta}{2} \nabla_v \theta 
	 +n_2 \sin l \theta \sin \frac{\theta}{2} \nabla_v v_1 +v_1\sin l \theta \sin \frac{\theta}{2} \nabla_v n_2 \right. \\
	 &\qquad \qquad \qquad \qquad \qquad \quad\left. +lv_1n_2\cos l \theta \sin \frac{\theta}{2} \nabla_v\theta +\frac{1}{2} v_1 n_2 \sin l\theta \cos \frac{\theta}{2} \nabla_v \theta +n_1\cos l\theta \cos \frac{\theta}{2} \nabla_v v_2 \right.\\ 
	&\qquad \qquad \qquad \qquad \qquad \quad +\left. v_2\cos l\theta \cos \frac{\theta}{2}\nabla_v n_1 -lv_2n_1 \sin l \theta \cos \frac{\theta}{2} \nabla_v \theta -\frac{1}{2} v_2 n_1 \cos l\theta \sin \frac{\theta}{2} \nabla_v \theta \right. \\ 
	&\qquad \qquad \qquad \qquad \qquad \quad +\left. n_2 \cos l\theta \sin \frac{\theta}{2} \nabla_v v_2+v_2 \cos l\theta \sin \frac{\theta}{2} \nabla_v n_2 -lv_2n_2 \sin l \theta \sin \frac{\theta}{2} \nabla_v \theta \right.\\
	&\qquad \qquad \qquad \qquad \qquad \quad \left.+\frac{1}{2} v_2n_2 \cos l\theta \cos \frac{\theta}{2} \nabla_v \theta \right).
\end{align*}
By \eqref{tb esti},\eqref{e_2},\eqref{est der n},\eqref{est der t_ell}, and \eqref{ell est}, the above can be further bounded by 
\begin{align*}
	\vert \nabla_v I_2 \vert &\lesssim \frac{1}{\vert v \vert^2 \sin^2 \frac{\theta}{2}} (1+\vert v \vert^2 t^2) \times \frac{1}{\vert v \vert} \times \vert v \vert + \frac{1}{\vert v \vert}(1+\vert v \vert t)\times \frac{1}{\vert v \vert^2 \vert \sin \frac{\theta}{2}\vert} \times \vert v \vert +\frac{1}{\vert v \vert}(1+\vert v \vert t)\times \frac{1}{\vert v \vert \vert \sin \frac{\theta}{2} \vert} (1+\vert v \vert t)\\
	&\lesssim \frac{1}{\vert v \cdot n(\xb) \vert^2} (1+\vert v \vert^2 t^2) .
\end{align*}
Hence, $\vert \nabla_{vv} X(0;t,x,v)\vert$ is bounded by 
\begin{align*}
	\vert \nabla_{vv} X(0;t,x,v) \vert \lesssim \frac{1}{\vert v \cdot n(\xb) \vert^2} (1+\vert v \vert^2 t^2). 
\end{align*}
To get estimate for $\vert \nabla_{xx} V(0;t,x,v)\vert$ and $\vert \nabla_{vx} V(0;t,x,v)\vert$, we now consider $[\nabla_{x} V(0;t,x,v)]_{(1,1)}$: 
\begin{align*}
	[\nabla_x V(0;t,x,v)]_{(1,1)} = \frac{2l}{\sin \frac{\theta}{2}} \left(v_1n_1 \sin l \theta \cos \frac{\theta}{2} +v_1n_2 \sin l \theta \sin \frac{\theta}{2} +v_2n_1 \cos l\theta \cos \frac{\theta}{2} + v_2n_2 \cos l\theta \sin \frac{\theta}{2}\right), 
\end{align*}
by Lemma \ref{X,V}. In $[\nabla_x V(0;t,x,v)]_{(1,1)}$, the main terms are 
\begin{align*}
	\frac{2l}{\sin \frac{\theta}{2}} v_1n_1 \sin l \theta \cos \frac{\theta}{2} \quad \textrm{and} \quad \frac{2l}{\sin \frac{\theta}{2}} v_2 n_1 \cos l\theta \cos \frac{\theta}{2},
\end{align*}
because these terms have the highest singularity order in  $[\nabla_x V(0;t,x,v)]_{(1,1)}$. Thus, for $\vert \nabla_{xx}V(0;t,x,v) \vert$, we now take the $x$-derivative for main terms:
\begin{align*}
	&\nabla_x \left(\frac{2l}{\sin \frac{\theta}{2}}\left(v_1n_1\sin l \theta \cos \frac{\theta}{2} + v_2n_1 \cos l \theta \cos \frac{\theta}{2}\right)\right)\\
	&= -\frac{l\cos \frac{\theta}{2}}{\sin ^2 \frac{\theta}{2}}\nabla_x \theta \left(v_1n_1\sin l \theta \cos \frac{\theta}{2} + v_2n_1 \cos l \theta \cos \frac{\theta}{2}\right)\\
	&\quad + \frac{2l}{\sin \frac{\theta}{2}} \left(v_1 \sin l \theta \cos \frac{\theta}{2} \nabla_x n_1 +lv_1 n_1 \cos l\theta \cos \frac{\theta}{2} \nabla_x \theta -\frac{1}{2} v_1 n_1 \sin l\theta \sin \frac{\theta}{2} \nabla_x \theta \right. \\
	&\qquad \qquad \quad + \left. v_2\cos l \theta \cos \frac{\theta}{2} \nabla_x n_1 -lv_2 n_1 \sin l\theta \cos \frac{\theta}{2} \nabla_x \theta -\frac{1}{2} v_2n_1 \cos l\theta \sin \frac{\theta}{2} \nabla_x \theta \right)\\
	&:=I_3.
\end{align*}
By \eqref{e_1},\eqref{est der n}, and \eqref{ell est}, $I_3$ can be further bounded by 
\begin{align*}
	\vert I_3 \vert \lesssim \frac{\vert v \vert}{\sin^4 \frac{\theta}{2}}(1+\vert v \vert t) + \frac{\vert v \vert }{ \sin^4\frac{\theta}{2}}(1+\vert v \vert^2 t^2)\lesssim \frac{\vert v \vert^5}{\vert v \cdot n(\xb) \vert^4}(1+\vert v \vert^2 t^2),
\end{align*}
which implies that 
\begin{align*}
	\vert \nabla_{xx} V(0;t,x,v) \vert \lesssim \frac{\vert v \vert^5}{\vert v \cdot n(\xb) \vert^4} (1+\vert v \vert^2t^2). 
\end{align*}
Similarly, we firstly take the $v$-derivative for main terms in $[\nabla_x V(0;t,x,v)]_{(1,1)}$ and then estimate $v$-derivatives. Then, we deduce 
\begin{align*}
	\vert \nabla_{vx} V(0;t,x,v) \vert \lesssim \frac{ \vert v \vert^3}{\vert v \cdot n(\xb) \vert^3} (1+ \vert v \vert^2 t^2), 
\end{align*}
where we have used \eqref{e_2},\eqref{est der n}, and \eqref{ell est}. Lastly, it remains to estimate $\vert \nabla_{xv}V(0;t,x,v)\vert$ and $\vert \nabla_{vv} V(0;t,x,v)\vert$. Let us consider the $(1,1)$ component of $\nabla_v V(0;t,x,v)$:
\begin{align*}
	[\nabla_v V(0;t,x,v)]_{(1,1)}&=\cos l\theta -2l \left( \frac{\tb}{\sin \frac{\theta}{2}} -\frac{1}{\vert v \vert}\right) \left(v_1n_1 \sin l \theta \cos \frac{\theta}{2} +v_1 n_2 \sin l \theta \sin \frac{\theta}{2} \right.\\
	&\left. \qquad \qquad \qquad \qquad \qquad \qquad \quad + v_2 n_1 \cos l\theta \cos \frac{\theta}{2} +v_2n_2 \cos l\theta \sin \frac{\theta}{2}\right),
\end{align*}
by Lemma \ref{X,V}. Similar to previous cases, main terms in $[\nabla_v V(0;t,x,v)]_{(1,1)}$ are 
\begin{align*}
	 -2l \left( \frac{\tb}{\sin \frac{\theta}{2}} -\frac{1}{\vert v \vert}\right)\left(v_1n_1 \sin l \theta \cos \frac{\theta}{2} +v_2 n_1 \cos l\theta \cos \frac{\theta}{2}\right):=I_4,
\end{align*}
by the same reason. Taking the $x$-derivative for $I_4$, we get 
\begin{align*}
	\nabla_x I_4 &= -2l \left( \frac{\nabla_x \tb}{\sin \frac{\theta}{2}} -\frac{\tb \cos \frac{\theta}{2}}{2\sin^2\frac{\theta}{2}} \nabla_x \theta\right)\left(v_1n_1 \sin l \theta \cos \frac{\theta}{2} +v_2 n_1 \cos l\theta \cos \frac{\theta}{2}\right)\\
	&\quad -2l \left( \frac{\tb}{\sin \frac{\theta}{2}} -\frac{1}{\vert v \vert}\right)\left(v_1 \sin l \theta \cos \frac{\theta}{2} \nabla_x n_1 + lv_1n_1 \cos l\theta \cos \frac{\theta}{2}\nabla_x \theta -\frac{1}{2} v_1n_1 \sin l \theta \sin \frac{\theta}{2} \nabla_x \theta \right. \\ 
	& \qquad \qquad \qquad \qquad \qquad +\left. v_2 \cos l \theta \cos \frac{\theta}{2} \nabla_x n_1 -l v_2n_1 \sin l \theta \cos \frac{\theta}{2} \nabla_x \theta -\frac{1}{2} v_2 n_1 \cos l\theta \sin \frac{\theta}{2} \nabla_x \theta\right).
\end{align*}
Using \eqref{tb esti},\eqref{e_1},\eqref{est der n},\eqref{est der t_ell}, and \eqref{ell est}, one obtains that 
\begin{align*}
	\vert \nabla_x I_4 \vert \lesssim \frac{1}{\vert \sin\frac{\theta}{2}\vert } (1+\vert v \vert t)\times \frac{1}{\vert v \vert \sin^2 \frac{\theta}{2}} \times \vert v \vert +\frac{1}{\vert \sin \frac{\theta}{2} \vert}(1+\vert v \vert t) \times \frac{1}{\vert v \vert} \times \frac{\vert v \vert}{ \sin^2 \frac{\theta}{2}}(1+\vert v \vert t)\lesssim \frac{\vert v \vert^3}{\vert v \cdot n(\xb) \vert^3} (1+\vert v \vert^2t^2). 
\end{align*}
Hence, we get estimate for $\vert\nabla_{xv}V(0;t,x,v)\vert$
\begin{align*}
	\vert \nabla_{xv} V(0;t,x,v) \vert \lesssim \frac{\vert v \vert^3}{\vert v \cdot n(\xb) \vert^3} (1+\vert v \vert^2 t^2). 
\end{align*}
Similarly, we take the $v$-derivative to main terms $I_4$ and estimate $\nabla_v I_4$ to get $\vert\nabla_{vv} V(0;t,x,v)\vert$. From \eqref{tb esti},\eqref{e_2},\eqref{est der n}, \eqref{est der t_ell}, and \eqref{ell est}, we derive 
\begin{align*}
	\vert \nabla_{vv} V(0;t,x,v) \vert \lesssim \frac{\vert v \vert}{ \vert v \cdot n(\xb) \vert^2} (1+\vert v \vert^2 t^2). 
\end{align*}
\end{proof}

\subsection{Proof of  Theorem \ref{thm 3}}

\begin{proof} [Proof of Theorem \ref{thm 3}]
	\textit{Step 1} First, we prove $C^{1}$ estimate. Note that it is easy to derive 
	\Be \label{dt XV}
		\p_{t}X(0;t,x,v) = -v^{k},\quad \p_{tt}X(0;t,x,v) = 0,\quad \p_{t}V(0;t,x,v) = 0, \quad \p_{tt}V(0;t,x,v) = 0, 
	\Ee
	where we assumed $t^{k+1} < 0 < t^{k}$ for some integer $k$. For $i\in\{t,x,v\}$,
	\Be \label{chain}
		\nabla_{i}f(t,x,v) = \nabla_{x}f_0 \nabla_{i}X(0;t,x,v) + \nabla_{v}f_0 \nabla_{i} V(0;t,x,v). 
	\Ee
	Hence using Lemma \ref{est der X,V} and \eqref{dt XV}, we obtain 
	\begin{equation*}
	\begin{split}
			|\p_{t}f| &\lesssim \|f_0\|_{C^{1}}|v|, \\
			|\nabla_{x}f| &\lesssim \|f_0\|_{C^{1}} \frac{|v|^{2}}{|v\cdot n(\xb)|^{2}} \langle v \rangle (1 + |v|t), \\
			|\nabla_{v}f| &\lesssim \|f_0\|_{C^{1}} \frac{1}{|v\cdot n(\xb)|} \langle v \rangle (1 + |v|t), \\
	\end{split}	
	\end{equation*}
	where $\xb = \xb(x,v)$ and $\langle v \rangle := 1 + |v|$. So we obtain \eqref{C1 bound}. \\
	\textit{Step 2} Now we compute second-order estimate. For $\nabla_{xx}f$, from \eqref{chain}, Lemma\ref{est der X,V}, and Lemma \ref{2nd est der X,V}, we obtain
	\begin{equation*}
	\begin{split}
		|\nabla_{xx}f| &= |\nabla_{x} \big(  \nabla_{x}f_0 \nabla_{x}X(0;t,x,v) + \nabla_{v}f_0 \nabla_{x} V(0;t,x,v)  \big)|  \\
		&\lesssim \|f_0\|_{C^{1}} \big( |\nabla_{xx}X(0) | + |\nabla_{xx}V(0)| \big)  
			+ \|f_0\|_{C^{2}} \big( |\nabla_{x}X(0)| + |\nabla_{x}V(0)| \big)^{2} \\
		&\lesssim \|f_0\|_{C^{2}} \frac{|v|^{4}}{|v\cdot n(\xb)|^{4}} \langle v \rangle^{2} (1 + |v|t)^{2},
	\end{split}
	\end{equation*}
	\begin{equation*}
	\begin{split}
	|\nabla_{vx}f| &= |\nabla_{v} \big(  \nabla_{x}f_0 \nabla_{x}X(0;t,x,v) + \nabla_{v}f_0 \nabla_{x} V(0;t,x,v)  \big)|  \\
	&\lesssim \|f_0\|_{C^{1}} \big( |\nabla_{vx}X(0) | + |\nabla_{vx}V(0)| \big)  
	+ \|f_0\|_{C^{2}} \big( |\nabla_{x}X(0)| + |\nabla_{x}V(0)| \big)\big( |\nabla_{v}X(0)| + |\nabla_{v}V(0)| \big) \\
	&\lesssim \|f_0\|_{C^{2}} \frac{|v|^{2}}{|v\cdot n(\xb)|^{3}} \langle v \rangle^{2} (1 + |v|t)^{2},
	\end{split}
	\end{equation*}
	and
	\begin{equation*}
	\begin{split}
	|\nabla_{vv}f| &= |\nabla_{v} \big(  \nabla_{x}f_0 \nabla_{v}X(0;t,x,v) + \nabla_{v}f_0 \nabla_{v} V(0;t,x,v)  \big)|  \\
	&\lesssim \|f_0\|_{C^{1}} \big( |\nabla_{vv}X(0) | + |\nabla_{vv}V(0)| \big)  
	+ \|f_0\|_{C^{2}} \big( |\nabla_{v}X(0)| + |\nabla_{v}V(0)| \big)^{2} \\
	&\lesssim \|f_0\|_{C^{2}} \frac{1}{|v\cdot n(\xb)|^{2}} \langle v \rangle^{2} (1 + |v|t)^{2},
	\end{split}
	\end{equation*}
	where $|\nabla_{xx, vx, vv}X|$ means $\sup_{i,j,k}|\nabla_{ij}X_{k}(0;t,x,v)|$ for $i,j \in\{ x_{1}, x_{2}, v_{1}, v_{2}\}$ and $k \in \{1,2\}$. (Also similar for $\nabla_{ij}V$.) Combining above three estimates, we obtain \eqref{C2 bound}. Second derivative estimates which contain at least one $\p_{t}$ also yield the same upper bound from \eqref{dt XV}. We omit the details. \\
\end{proof}

\noindent{\bf Acknowledgments.}
The authors thank Haitao Wang for suggestion and fruitful discussion. Their research is supported by the National Research Foundation of Korea(NRF) grant funded by the Korean government(MSIT)(No. NRF-2019R1C1C1010915). DL is also supported by the POSCO Science Fellowship of POSCO TJ Park Foundation. The authors sincerely appreciate the anonymous referees for their valuable comments and suggestions on the paper. 

\bibliographystyle{plain}

\end{document}